\documentclass{article}
\usepackage{mathtools}
\usepackage{mathrsfs,enumerate}
\usepackage{amsmath}
\usepackage{amssymb}
\usepackage[numbers]{natbib}
\usepackage{hyperref}
\usepackage[draft]{todonotes}
\usepackage[a4paper,top=2cm,bottom=2cm,left=2cm,right=2cm,bindingoffset=5mm]{geometry}
\newcommand{\A}{\mathbb{A}}
\newcommand{\B}{\mathbb{B}}

\newcommand{\N}{\mathbb{N}}

\newcommand{\Q}{\mathbb{Q}}
\newcommand{\R}{\mathbb{R}}

\newcommand{\BB}{\mathscr{B}}

\newcommand{\DD}{\mathscr{D}}

\newcommand{\LL}{\mathscr{L}}

\newcommand{\ZZ}{\mathscr{Z}}
\newcommand{\cA}{{\ensuremath{\mathcal A}}}
\newcommand{\cB}{{\ensuremath{\mathcal B}}}

\newcommand{\cI}{{\ensuremath{\mathcal I}}}
\newcommand{\cL}{{\ensuremath{\mathcal L}}}
\newcommand{\cM}{{\ensuremath{\mathcal M}}}

\newcommand{\cP}{{\ensuremath{\mathcal P}}}

\newcommand{\cX}{{\ensuremath{\mathcal X}}}

\renewcommand{\gg}{{\mbox{\boldmath$g$}}}

\newcommand{\pp}{{\mbox{\boldmath$p$}}}
\newcommand{\qq}{{\mbox{\boldmath$q$}}}

\newcommand{\vv}{{\mbox{\boldmath$v$}}}
\newcommand{\ww}{{\mbox{\boldmath$w$}}}

\newcommand{\yy}{{\mbox{\boldmath$y$}}}
\newcommand{\zz}{{\mbox{\boldmath$z$}}}

\newcommand{\spp}{{\mbox{\scriptsize\boldmath$p$}}}

\newcommand{\svv}{{\mbox{\scriptsize\boldmath$v$}}}
\newcommand{\sww}{{\mbox{\scriptsize\boldmath$w$}}}

\newcommand{\eeta}{{\mbox{\boldmath$\eta$}}}

\newcommand{\mmu}{{\mbox{\boldmath$\mu$}}}
\newcommand{\nnu}{{\mbox{\boldmath$\nu$}}}

\newcommand{\ssigma}{{\mbox{\boldmath$\sigma$}}}

\newcommand{\seeta}{{\mbox{\scriptsize\boldmath$\eta$}}}

\newcommand{\smmu}{{\mbox{\scriptsize\boldmath$\mu$}}}

\newcommand{\sfd}{{\sf d}}
\newcommand{\sfe}{{\sf e}}

\newcommand{\sfu}{{\sf u}}

\newcommand{\sfw}{{\sf w}}

\newcommand{\sfF}{{\sf F}}

\newcommand{\sfH}{{\sf H}}

\newcommand{\sfL}{{\sf L}}

\newcommand{\sfY}{{\sf Y}}

\newcommand{\frm}{{\frak m}}

\newcommand{\rmA}{{\mathrm A}}

\newcommand{\rmC}{{\mathrm C}}
\newcommand{\rmD}{{\mathrm D}}
\newcommand{\rmE}{{\mathrm E}}

\newcommand{\rmL}{{\mathrm L}}

\newcommand{\Kliminf}{K\kern-3pt-\kern-2pt\mathop{\rm
lim\,inf}\limits}  
\newcommand{\Klimsup}{K\kern-3pt-\kern-2pt\mathop{\rm lim\,sup}\limits}  
\newcommand{\supp}{\mathop{\rm supp}\nolimits}   

\renewcommand{\d}{{\mathrm d}}
\newcommand{\dt}{{\d t}}

\newcommand{\dx}{{\d x}}

\newcommand{\restr}[1]{\lower3pt\hbox{$|_{#1}$}}
\newcommand{\topref}[2]{\stackrel{\eqref{#1}}#2}

\newcommand{\down}{\downarrow}              
\newcommand{\up}{\uparrow}
\newcommand{\weakto}{\rightharpoonup}

\newcommand{\eps}{\varepsilon}  
\newcommand{\nchi}{{\raise.3ex\hbox{$\chi$}}}

\newcommand{\forevery}{\text{for every }}

\def\qed{\ifmmode 
  \else \leavevmode\unskip\penalty9999 \hbox{}\nobreak\hfill
  \fi               
    \qquad           \hbox{\hskip.5em $\square$
                \hskip.1em}}
\def\endproofsym{\qed}

\newenvironment{proof}[1][Proof]{\def\endproofsym{\qed}\trivlist\item[\hskip\labelsep{%
\noindent{\normalfont\emph{#1}.}\hskip .321429\parindent}]\ignorespaces}
{\endproofsym\endtrivlist}

\newcommand{\nc}{\normalcolor}

\newcommand{\GGG}{\color{blue}}

\numberwithin{equation}{section}

\newtheorem{theorem}{Theorem}[section]
\newtheorem{proposition}[theorem]{Proposition}
\newtheorem {lemma}[theorem]{Lemma}

\newtheorem {definition}[theorem]{Definition}
\newtheorem{corollary}[theorem]{Corollary}
\newtheorem{remark}[theorem]{Remark}
\newtheorem{problem}[theorem]{Problem}
\newtheorem{hypothesis}[theorem]{Assumptions}

\newenvironment{system}
{\left\lbrace\begin{aligned}}
{\end{aligned}\right.}

\newcommand{\smallplus}{{\raise.3ex\hbox{$\scriptstyle+$}}}

\newcommand{\sign}{\operatorname{sign}}
\newcommand{\loc}{\operatorname{loc}}

\newcommand{\ac}{\operatorname{AC}}

\newcommand{\hj}{\operatorname{HJ}}

\newcommand{\weight}{\kappa}
\newcommand{\CE}[3]{\mathrm{CE}_{#1,#2}(#3)}
\newcommand{\KL}{K\kern-2pt L}
\newcommand{\K}{\mathbb{K}}

\newcommand{\apliminf}{\mathop{\operatorname{ap-liminf}}}  
\newcommand{\aplimsup}{\mathop{\operatorname{ap-limsup}}}  
\newcommand{\aplim}{\mathop{\operatorname{ap-lim}}}

\def\vep{\varepsilon}

\title{A variational approach to the mean field planning problem}

\author{Carlo Orrieri \thanks{Dipartimento di Matematica ``G. Castelnuovo'', Sapienza Universit\`a di Roma. Piazzale Aldo Moro 5, 00185 Roma, Italy. The author acknowledges the financial support provided by PRIN 20155PAWZB ``Large Scale Random Structures''.
Email: \texttt{orrieri@mat.uniroma1.it}} \and
Alessio Porretta \thanks{Dipartimento di Matematica, Universit\`a di Roma Tor Vergata. 
Via della Ricerca Scientifica 1, 00133 Roma, Italy. Partially supported by University of Tor Vergata (\lq\lq Consolidate The foundations 2015\rq\rq) project {\it Irreversibility in Dynamic Optimization}.  Email: \texttt{porretta@mat.uniroma2.it}} \and
Giuseppe Savar\'e \thanks{Dipartimento di Matematica ''F. Casorati'', Universit\`a di Pavia. 
Via Ferrata 5, 27100 Pavia, Italy. 
Partially supported by 
Cariplo foundation and Regione Lombardia via project \emph{Variational
evolution problems and optimal transport}, by MIUR PRIN 2015 project
\emph{Calculus of Variations}, and by IMATI-CNR. Email: \texttt{giuseppe.savare@unipv.it}}}

\date{\today}

\begin{document}

\maketitle
\begin{abstract}
We investigate a first-order mean field planning problem 
of the form
\begin{displaymath}
\begin{system}
-\partial_t u + H(x,Du) &= f(x,m) &&\text{in } (0,T)\times \R^d, \\
\partial_t m  - \nabla\cdot (m\,H_\spp(x,Du)) &= 0  &&\text{in }(0,T)\times \R^d,\\
m(0,\cdot) = m_0, \; m(T,\cdot) &= m_T &&\text{in } \R^d,
\end{system}
\end{displaymath}
associated to a convex Hamiltonian $H$ with quadratic growth
and a monotone interaction term $f$ with polynomial growth. 

We exploit the variational structure of the system, which encodes
the first order optimality condition of a convex dynamic optimal entropy-transport
problem with respect to the unknown density $m$ and of
its dual, involving the maximization of an integral functional among
all the subsolutions $u$ of an Hamilton-Jacobi equation.

Combining ideas from optimal transport, convex analysis and
renormalized solutions to the continuity equation, we will
prove existence and (at least partial) uniqueness of a weak solution
$(m,u)$.
A crucial step of our approach relies on a careful analysis of distributional subsolutions to Hamilton-Jacobi equations of the form $-\partial_t u + H(x,Du) \le \alpha$, 
under minimal summability conditions on $\alpha$, and to a
measure-theoretic description of the optimality via a suitable
contact-defect measure.
Finally, using the superposition principle, we are able to describe
the solution to the system 
by means of a measure on the path space encoding the local behavior of the players.
\end{abstract}

\tableofcontents

\section{Introduction}

Mean field games, as well as  mean field control problems, describe strategic interactions among large numbers of similar rational agents. Typically, the generic agent  aims at optimizing some functional depending both on its  own (controlled) dynamical state and on the average collective behavior, usually represented by the distribution law of the states. 
When the individual optimization and the collective evolution are consistent,  the system is described by  two PDEs satisfied, respectively,  by  the value function $u$ of the single agent and by the distribution law $m$ of the population. The simplest model is the following coupling of Hamilton-Jacobi-Bellman and Kolmogorov Fokker-Planck equations:
\begin{equation}\label{mfg-system}
\left\{
\begin{aligned}
-\partial_t u  - \vep \Delta u + H(x,Du) &= f(x,m),   \\
\partial_t m -\vep \Delta m - \nabla\cdot  (mH_p(x,Du)) &= 0\,,   
\end{aligned}
\right.
\end{equation}
where $\vep=0$ and  $\vep>0$ distinguish whether the agent's  dynamic is purely deterministic or, respectively, it  contains some Brownian noise.

Since the introduction of mean field game theory by J.-M. Lasry and P.-L. Lions, who derived \eqref{mfg-system} in connection with the limit of Nash equilibria of $N$-players games as $N\to \infty$ (see \cite{LL06cr1, LL06cr2}),  this kind of systems has been extensively studied, mainly under the stabilization condition of $f(x,m)$ being nondecreasing with respect to $m$.  
The system is usually closed with an initial condition for the density $m$ and a  terminal condition (a final pay-off) for the value function $u$. 

Otherwise, the {\it planning problem} consists in prescribing  both initial and terminal conditions for the density $m$. In this case,  the goal is to solve the following:
\begin{equation}\label{mfg-plan}
\left\{\begin{aligned}
-\partial_t u  - \vep \Delta u + H(x,Du) &= f(x,m),   \\
\partial_t m -\vep \Delta m - \nabla\cdot  (mH_p(x,Du)) &= 0,
\\
m(0)=m_0\,,\,\,m(T)&=m_1\,,   
\end{aligned}\right.
\end{equation}
in some finite horizon $T>0$. 

In the framework of mean field game  theory, the planning problem  was suggested and developed by P.-L. Lions in his courses at Coll\`{e}ge de France. In particular, when the problem is set on the flat torus and smooth initial-terminal densities $m_0,m_1$ are prescribed,  P.-L. Lions   proved the existence of smooth solutions both for the second order case ($\vep>0$ in \eqref{mfg-plan}) with  quadratic Hamiltonian $H(x,p)$ and for the first order case ($\vep=0$) when $f= f(m)$ is an increasing function (see \cite{L-coll}). Later, existence and uniqueness of weak solutions were proved for the second order case for more general Hamiltonians (see \cite{porretta2013planning, porrettaDGA}). 
Here the strategy used was to build solutions of the planning problem  by  penalizing the final pay-off $u(T)$   in a way to force the required density condition $m(T)$ to hold at the final time. This approach, which was also exploited for numerical schemes in \cite{achdou2012mean}, reminds that  exact controllability in finite time  can be  obtained as singular limit of optimal control problems through penalization of  the final pay-off, and therefore of the adjoint state at the final time. 
Indeed, J.-M. Lasry and P.-L. Lions already addressed, in their initial papers on mean field games (see \cite{LL06cr2}), that  system \eqref{mfg-plan} can be recast as the optimality system satisfied by state and adjoint state of an optimal control problem. Precisely, if $F(x,m)= \int_0^m f(x,s)ds$ and $L(x,\qq)$ is the Fenchel conjugate of the Hamiltonian $H(x,-\pp)$, system \eqref{mfg-system}  formally appears as the first order condition of the following minimization problem: 
 \begin{equation}\label{a-opt}
\min
\int_0^T\int_{\R^d} \big[L(x,\vv)\,m+
F(x,m)\big]\,\, \d x\,\d t \,:\quad \vv\in
L^2(m\,dxdt),\quad 
\begin{system} \partial_t m-\vep \Delta m+\nabla\cdot (m\,\vv )&=0  \\
m(0,\cdot)=m_0\,,m(T,\cdot)&=m_1  \end{system}
\end{equation}
It is well-known that, when $\vep=0$, $F=0$ and $H(x,\pp)=
\frac12|\pp|^2$, \eqref{a-opt} is the so-called fluid mechanics
formulation of the  Monge-Kantorovich mass transfer problem introduced
by Benamou and Brenier (see  \cite{BB})
and leading to the dynamic characterization of the
$L^2$-Kantorovich-Rubinstein-Wasserstein
distance $W_2$ between probability measures in $\R^d$ with
finite quadratic moment \cite{ambrosio2008gradient,
  villani2009optimal}. 
This approach 
has then been extended to general Lagrangian formulations
(see e.g.~\cite{Bernard-Buffoni05} and \cite[Chap.~7]{villani2009optimal}),
also exploiting the metric--Riemannian viewpoint intrinsic to the
dynamic approach \cite{Otto01,Otto-Villani00,ambrosio2008gradient}.

This way, the planning problem \eqref{mfg-plan} appears as a natural
generalization of optimal transport problems,
when the extra penalization term of entropic type (induced by 
the convex function $F$ in \eqref{a-opt}) prevents concentration of the 
transported density $m$. 

The study of various kind of entropic relaxation
of genuine optimal transport problems recently attracted a lot of
attention, due to the regularizing and convexification effect 
added by the entropic terms 
(see e.g.~\cite{Leonard12,CPVX18,LMS18}).
In the present case, since $F$ is nonlinear w.r.t.~$m$
(we will consider a typical power behavior of exponent $p\in
(1,+\infty)$),
the minimization of the dynamic cost functional 
cannot be reduced to a simpler and more explicit Kantorovich formulation 
involving transport plans, but it has to take into account
a complex interpolation dynamic interaction
between the transport and the local density terms.

In this article, our goal  is to use some ideas of optimal transport
theory, convex duality, and dynamic superposition principles
in order to study a suitable weak formulation of 
the deterministic mean field planning problem,
that in strong form and assuming $m>0$ everywhere can be
formally written as
\begin{equation}\label{eq:sMFPP}
\begin{system}
-\partial_t u + H(x,Du) &= f(x,m) &&\text{in } Q, \\
\partial_t m  - \nabla\cdot (m\,H_\spp(x,Du)) &= 0  &&\text{in } Q, \\
m(t,\cdot) = m_0, \; m(1,\cdot) &= m_1 &&\text{in } \R^d,
\end{system}
\tag{s-MFPP}
\end{equation}
where the final time has been normalized to $T=1$ and
$Q$ is the space-time cylinder $(0,1)\times \R^d$.

Compared to  previous results for mean field games systems, we
set the problem in the whole space, which seems more natural in the
viewpoint of optimal transport of probability measures. Otherwise, as
in previous results  we rely on two crucial structure conditions,
namely that $H(x,\pp)$ is convex with respect to $\pp$ and   
$f(x,m)$ is 
increasing with respect to $m$. 
The main result that we prove is the existence of weak solutions (and partial uniqueness for $m$ and $Du$) to the planning problem \eqref{eq:sMFPP} under fairly general growth conditions on $H$ and $f$, that will be discussed later in detail. The main effect of this generality is that no standard framework can be applied, in particular, to the Hamilton-Jacobi equation.

Let us stress that, in optimal transport theory, the duality between
the continuity equation and the Hamilton-Jacobi equation has  been
mostly exploited formally or 
under conditions of regularity which allow for the use of 
 explicit representation formula of Hopf-Lax type
(typical of the viscosity solutions' theory). 
Another source of difficulty lies on the possible vanishing of the 
density $m$ in large sets, so that 
the first equation of \eqref{eq:sMFPP} should be 
written in a relaxed form.

Unfortunately, the coupling appearing in  the system  \eqref{eq:sMFPP}
often leads outside the standard framework of continuous solutions to
Hamilton-Jacobi equations. By contrast, convex duality methods have
been successfully used in mean field game  theory under very general
growth conditions. In particular, existence and uniqueness results
were proved for the deterministic, or degenerate diffusion case,
using relaxed solutions of the Hamilton-Jacobi equations and weak
formulations of the system, see  \cite{cardaliaguet2015weak, cardaliaguet-graber, cardaliaguet2015second}. Those
relaxed formulations appear very naturally from the convex duality of
the variational problems, and we will follow a similar strategy here
when dealing with the planning problem.

\subsubsection*{A heuristic derivation of MFPP from the minimax principle}

Before describing the main contributions of our paper, let
us first briefly explain the heuristic derivation of the relaxed
formulation
of \eqref{eq:sMFPP} starting from the minimization of 
the functional
\begin{equation}\label{eq:Bintro}
  \begin{gathered}
    \cB(m,\vv):=
    \iint_Q \big[L(x,\vv)\,m+ F(x,m)\big]\,\, \d x\,\d t
    \quad \text{among all the solutions of}\\
    \partial_t m+\nabla\cdot(m\,\vv)=0\quad
    \text{in }\DD'((0,1)\times \R^d),\quad m(0,\cdot)=m_0,\
    m(1,\cdot)=m_1,\quad
    m\ge 0.
  \end{gathered}
\end{equation}
Arguing as in the formal discussion in \cite{Otto-Villani00}, 
we write the above constraint minimization as an inf-sup formulation
of
a saddle problem, where a new Lagrange multiplier $u\in C^1_c(\R\times
\R^d)$ is used to impose the continuity equation and the boundary conditions:
\begin{align*}
   \adjustlimits \inf_{m\ge0,\svv}\sup_u \iint_Q \big[L(x,\vv)\,m+
  F(x,m)\big]\,\d x\,\d t+
  \int_{\R^d}u_0m_0\,\d x-
  \int_{\R^d}u_1m_1\,\d x+
  \iint_Q \big[\partial_t u+Du\cdot\vv\big]\,m\,\d x \,\d t
\end{align*}
A standard trick to substitute $\vv$ with the new variable $\ww=m\vv$,
so that the saddle function is convex in $(m,\ww)$ and concave 
(in fact, linear) in $u$, suggests the possibility to interchange the
order of inf and sup, obtaining the dual problem
\begin{align*}
  \sup_u  & \int_{\R^d}u_0m_0\,\d x-
  \int_{\R^d}u_1m_1\,\d x+
  \inf_{m\ge 0,\svv}
  \iint_Q \big[L(x,\vv)\,m+
  F(x,m)+
  m\partial_t u+m Du\cdot\vv\big]\,\d x \,\d t
            \\&=
                \sup_u \int_{\R^d}u_0m_0\,\d x-
  \int_{\R^d}u_1m_1\,\d x+
                          \inf_{m\ge 0}
                          \iint_Q \Big[\big(-H(x,Du)
                +\partial_t u\big)\,m+F(x,m)\Big]\,\d x \,\d t
                \\&=
                \sup_u \int_{\R^d}u_0m_0\,\d x-
  \int_{\R^d}u_1m_1\,\d x-
                     \iint_Q F^*\Big(-\partial_t u+H(x,Du)\Big)\,\d x
                    \,\d t.
\end{align*}
where by Fenchel duality
\begin{displaymath}
  \inf_{\svv} L(x,\vv)+\pp\cdot\vv= - H(x,\pp),\quad
  \inf_{m\ge0} F(x,m)+\ell m=-F^*(x,\ell).
\end{displaymath}
The dual problem thus consists in the maximization of
\begin{equation}
  \label{eq:Aintro}
  \begin{gathered}
    \cA(u,\ell):= \int_{\R^d}u_0m_0\,\d x-
  \int_{\R^d}u_1m_1\,\d x-
                     \iint_Q F^*(\ell(t,x))\,\d x
                     \,\d t,\quad\\
                     \text{under the constraint}\quad
                     -\partial_t u+H(x,Du)=\ell\quad\text{in }Q.
  \end{gathered}
\end{equation}
Minimax principle yields $\cB(m,\vv)\ge \cA(u,\ell)$ 
whenever $m,\vv$ solve the continuity equation \eqref{eq:Bintro}
and $u,\ell$ are linked by the Hamilton-Jacobi equation
\eqref{eq:Aintro}.
The variational formulation of MFPP should then arise as
the optimality condition at minimizers $(m,\vv)$ of the primal problem
\eqref{eq:Bintro} and maximizers $(u,\ell)$ of the dual problem
\eqref{eq:Aintro}, if there is no duality gap.
Such conditions can be easily obtained by rearranging the 
(nonnegative) difference $\cB(m,\vv)-\cA(u,\ell)$
and assuming enough regularity in order to justify integration by
parts; we obtain the sum of two nonnegative terms:
\begin{align*}
  \cB(m,\vv)-\cA(u,\ell)
  &=\iint_Q \Big(L(x,\vv)+Du\cdot \vv+H(x,Du)\Big)m\,\d x\,\d t
+\iint_Q \Big(F(x,m)-\ell \,m+F^*(x,\ell)\Big)\,\d x\,\d t,
\end{align*}
so that, assuming differentiability of $H$ and continuity of $f$,
optimal pairs $(m,\vv)$ and $(u,\ell)$ 
with $\cB(m,\vv)=\cA(u,\ell)$ are characterized by 
\begin{align}
  \label{eq:285}
  \Big(L(x,\vv)+Du\cdot \vv+H(x,Du)\Big)m&=0\ \text{$\LL^{d+1}$ a.e.~in
  $Q$},&&\text{ i.e. }
  \vv=-H_\spp(x,Du)\quad \text{$m$ a.e.~in
  $Q$},\\
  \label{eq:282}
  F(x,m)-\ell \,m+F^*(x,\ell)&=0\
                               \text{$\LL^{d+1}$ a.e.~in
  $Q$},&&\text{ i.e. }
                               \begin{system}
                                 \ell&=f(x,m)&\text{if }m>0,\\
                                 \ell&\le f(x,0)&\text{if }m=0.
                               \end{system}
\end{align}
Combining \eqref{eq:285} and \eqref{eq:282} with the 
continuity equation of \eqref{eq:Bintro} and the Hamilton-Jacobi equation
of 
\eqref{eq:Aintro} we end up with the relaxed formulation of 
MFPP:
\begin{equation}\label{eq:wMFPP}
\begin{system}
-\partial_t u + H(x,Du) &\le f(x,m) &&\text{in } Q, \\
-\partial_t u + H(x,Du) &= f(x,m) &&  \text{if } m(x,t)>0,\\
\partial_t m  - \nabla\cdot (m\,H_\spp(x,Du)) &= 0  &&\text{in } Q,
\\
m(t,\cdot) = m_0, \; m(1,\cdot) &= m_1 &&\text{in } \R^d.
\end{system}
\tag{r-MFPP}
\end{equation}
One can notice that $u$ is just required to be a global
\emph{subsolution}
to the Hamilton-Jacobi equation with right-hand side $f(x,m)$;
the equality will be attained only in the set where $m$ is strictly positive.
\subsubsection*{The variational setting and the weak formulation of
  MFPP}

Since \eqref{eq:wMFPP} represents the optimality system of a saddle
point problem, there is a natural strategy to prove the existence of a solution
and to obtain a well posed weak formulation:
\begin{enumerate}[\rm S1.]
\item By the Direct method of the Calculus of Variations prove
  the existence of a minimizer for the primal problem 
  \eqref{eq:Bintro}: it can be formulated as the minimum
  of two convex and lower semicontinuous functions along curves of
  probability measures solving the continuity equation. 
  This part will be developed in Section
  \ref{subsec:variational-primal}; a preliminary discussion,
  related to the particular case when $H(x,\pp)=\frac 12|p|^2$ and
  $F(x,m)=\frac 12 m^p$
  (leading to the Kantorovich-Lebesgue $KL_{2,p}$ cost)
  is developed in Section \ref{subsec:KL}.  
\item 
  By (a suitably refined version of) the Von Neumann minimax
  principle (see \ref{subsec:minimax} in the Appendix) 
  prove that there is no duality gap and 
  $\min \cB(m,\vv)=\sup\cA(u,\ell)$ in a suitable class 
  of smooth functions (Section \ref{subsec:variational-dual}).
\item As it is typical in Optimal Transport problems, 
  existence of maximizers of the dual problem is a much subtler issue,
  due to the lack of compactness of the dual formulation
  in spaces of smooth functions.
  Here the first equation (in fact an inequality) of the relaxed
  formulation
  suggests to first study \emph{subsolutions} to the Hamilton-Jacobi 
  equation of the form
  \begin{equation}
    \label{eq:284}
    -\partial_t u+H(x,Du)\le \alpha
  \end{equation}
  just requiring minimal summability on $\alpha$ 
  (derived by an a priori estimate in $L^q(Q)$ 
  due to the
  growth
  of $F^*$)
  and minimal regularity of $u$, in order to give a distributional sense to
  \eqref{eq:284}.
  
  This preliminary study is the main topic of Section 
  \ref{sec:weak-subsolutions}: we will recover a suitable notion of 
  traces of $u$ at $t=0$ and $t=1$ in
  \S\ref{subsec:weak-subsol}, 
  we will prove that subsolutions exhibit a nice regularization
  effect,
  sufficient to gain upper semicontinuity of $\cA$ and
  enough compactness (Section \ref{subsec:stability}),
  to prove existence of a maximizer of 
  the relaxed formulation (Section \ref{subsec:existence-dual}):
  \begin{equation}
    \label{eq:288}
    \max\Big\{\cA(u,\alpha):\ -\partial_tu+H(x,Du)\le \alpha\text{ in }\DD'(Q)\Big\}.
  \end{equation}
  It is worth noticing that \eqref{eq:288} involves a
 \emph{convex}
  constraint on pairs $(u,\alpha)$, which is clearly more stable than  
  the condition $-\partial_t u+H(x,Du)=\ell$.
\item
  A crucial point concerns the duality 
  between the primal and the relaxed dual problem,
  in particular the fact that the optimal value of \eqref{eq:288}
  still coincides with the minimum of $\cB$.
  This fact will be addressed in 
  sections \ref{subsec:duality-transport}
  and \ref{subsec:variational-optimal}.
\item Having at our disposal minimizers of $\cB$ and maximizers of
  $\cA$,
  it is not difficult to check that the optimality condition yields
  $\vv=-H_\spp(x,Du)$ a.e.~on the set where $m>0$ and
  $\alpha=f(x,m)$. 
  The last technical question concerns 
  the ``contact'' condition $-\partial_t u+H(x,Du)=f(x,m)$ when $m>0$,
  since we have just a distributional subsolution to the
  Hamilton-Jacobi equation and $-\partial_t u$ 
  may have singular parts (see also \cite{carlier+al} for similar questions).

  In order to overcome this difficulty, we derive a distributional
  condition
  which can be formally obtained by combining the contact condition
  with the continuity equation satisfied by $m$: 
  assuming regular solutions, and multiplying the continuity equation
  by $u$, one can easily obtain
  \begin{align*}
    \partial_t(um)+\nabla\cdot(u\,m\vv)
    &=
      u\big(\partial_t m+\nabla\cdot(m\vv)\big)+
      \big(\partial_t u+Du\cdot \vv\big)m
    \\&=\big(\partial_t u-H(x,Du)+\alpha)m
        +\big(-\alpha+H(x,Du) +Du\cdot \vv\big)m
        \\&= \big(-\alpha+H(x,Du) +Du\cdot \vv\big)m
  \end{align*}
  so that we can in principle substitute the contact condition with
  a suitable (distributional and renormalized) version of
  \begin{equation}
    \label{eq:290}
    \partial_t(um)+\nabla\cdot(u\,m\vv)+
    \big(\alpha-H(x,Du) -Du\cdot \vv\big)m=0 \quad\text{ in }Q.
  \end{equation}
  In Section \ref{subsec:contact-defect} we will associate
  a nonnegative Radon ``contact-defect'' measure $\vartheta$ to every 
  pair of competitors $(m,\vv)$ for the primal problem and
  $(u,\alpha)$ for the dual problem and we will show
  that vanishing of $\vartheta$ is the right distributional way
  to impose the missing contact condition. 
\item With all the above tools at disposal,
  Section 
\ref{subsec:variational-optimal}
collects all the main result concerning the formulation,  the existence and the characterizations
of solutions to \eqref{eq:wMFPP}.
  \item
    A further analysis, carried out in the last Section
    \ref{sec:Lagrangian}, 
    concerns the Lagrangian viewpoint to 
    \eqref{eq:wMFPP}. Using the superposition principle,
    we are able to describe the solutions to 
    \eqref{eq:wMFPP} by means of a measure $\eeta$ 
    on the path space $\ac^2([0,1];\R^d)$ 
    (the characteristics associated to the velocity vector field
    $\vv$).
    $\eeta$ encodes the local behaviour of particles (or agents)
    which try to minimize a modified Lagrangian cost
    obtained by the sum of $L$ with 
    a potential induced by the mass distribution $\alpha=f(x,m)$.

    In this way, we can also recover a ``static'' description in terms
    of Optimal Transport, but where the transportation cost
    is affected by the density $m$ of the moving particles.
\end{enumerate}
Some preliminary material, concerning
Optimal transport, displacement interpolation, continuity equation 
weighted spaces, convergence in measure and precise representatives
of increasing functions is collected in Sections 
\ref{subsec:Wass} and \ref{sec:4}.

\section{Notation and assumptions}
\label{sec:notation}
In this section we collect some notions and results concerning convex functionals on measures, continuity equations and their connections with the theory of optimal transportation.
Moreover, we resume some useful properties of locally increasing functions which we will need in the sequel. 

\renewcommand{\GGG}{}
\subsection{Notation}
\label{subsec:notation}
\subsubsection*{List of main notation}\GGG
\halign{$#$\hfil\ &#\hfil\cr
I,Q&the interval $(0,1)$ and the space-time cylinder $(0,1)\times
\R^d$\cr
\LL^h,\ 
\lambda&the $h$-dimensional Lebesgue measure and 
its restriction to
$Q$\cr
L^p(\Omega,\frm)&Lebesgue space w.r.t.~the $\sigma$-finite Borel
measure $\frm$\cr
\kappa&the weight $1+|x|^2$ on $\R^d$ or on $Q$\cr
L^1_\kappa(\Omega)&weighted Lebesgue space of functions $f$ satisfying
$f\cdot \kappa\in L^1(\Omega)$\cr
L^\infty_{1/\kappa}(\Omega)&the dual space of $L^1_\kappa(\Omega)$: $f\cdot
1/\kappa\in L^\infty(\Omega)$\cr
\cX^q(\Omega) &the  space $L^q +L^\infty_{1/\weight}(\Omega)$
see Definition \ref{def:spaceX}\cr
H(x,\pp),\ L(x,\vv)&the Hamiltonian and the dual Lagrangian, see
\ref{h.1} and \eqref{eq:8}\cr
F(x,m),\ f(x,m),\ F^*(x,a)&the cost density function, its derivative, see
\ref{h.1} and \eqref{eq:def:F}\cr
c_H,\gamma^\pm_H,c_f,\gamma_f&structural constants and functions
related to $H,f$, see \ref{h.1}\cr
\cP(\R^d)&the set of Borel probability measure on $\R^d$\cr
\cP_2(\R^d)&Borel probability measures with finite quadratic moment\cr
\cP_2^r(\R^d), \cP_{2,p}^r(\R^d)&Absolutely continuous measures in
$\cP_2(\R^d)$ (with density in $L^p(\R^d)$)\cr
W_2(\mu_0,\mu_1)&the Kantorovich-Rubinstein-Wasserstein distance, see \eqref{wasserstein}\cr
T_\sharp\mu&push forward of a measure $\mu$ through the map $T$, see
Section \ref{subsec:Wass}\cr
\ac^2([0,1];X)&absolutely continuous curves with values in the metric
space $X$\cr
\tilde \mu&measure in $Q$ whose disintegration is $\mu\in
\ac^2([0,1];\cP_2(\R^d))$, see \eqref{eq:42}\cr
\CE 2p{Q},\ \CE2p{Q;\mu_0,\mu_1}&pairs $(m,\vv)$ solving the
continuity equation,
Def.~\ref{def:CE}\cr
KL_{2,p}^{(a)}&Kantorovich-Lebesgue costs, Definition \ref{def:KL}\cr
\ZZ,\ZZ_c&collections of pairs $(\zeta,Z)$ satisfying
\eqref{eq:239}\cr
\hj_q(Q,H)&subsolutions $(u,\alpha)$ of the Hamilton-Jacobi equations,
Def.~\ref{def:subsolutions}\cr
\cB(m,\vv),\ \cA(u,\alpha)&primal and dual functionals, \eqref{eq:85}
and \eqref{eq:98}\cr
L_\alpha(t,x,\vv),\ \cL_\alpha(m,\vv),\ c_\alpha(x_0,x_1)&modified
Lagrangians and induced transport cost
\eqref{eq:244}, \eqref{eq:245}, \eqref{eq:250}\cr
\Gamma=C^0([0,1];\R^d)&space of continuous curves with the uniform
topology\cr
\rmE_2[\gamma]&energy of a curve $\gamma\in \ac_2([0,1];\R^d)$,
\eqref{eq:212}\cr
\sfe,\sfe_t,\sfd&evaluation maps on $[0,1]\times \Gamma$,
\eqref{eq:49}, \eqref{eq:242}\cr
\hat\alpha,\ M\alpha&precise representative and maximal function
of $\alpha\in L^q(Q)$, \eqref{eq:274}, \eqref{eq:275}\cr
}
\nc
\bigskip
Throughout the paper  $I$ stands for the open interval $(0,1)$ and $Q:=
(0,1) \times \R^d $.
\GGG
\newcommand{\bOmega}{\Omega}

If $\Omega$ is a Polish topological space
(i.e.~its topology is induced by a complete and separable distance) 
we will denote by  $\BB$ its Borel $\sigma$-algebra
and by $\cP(\Omega)$ the set of Borel probability measures 
on $\Omega$, endowed with the topology of weak convergence, in duality
with the set of all continuous and bounded functions, denoted by  $C_b(\Omega)$. 
When $\Omega=\R^d$ we will also deal with the space $\cP_2(\R^d)$ of
measures
with finite quadratic moment, i.e.
\begin{equation}
  \label{eq:229}
  \mu\in \cP_2(\R^d)\quad
  \Leftrightarrow\quad
  \int_{\R^d}|x|^2\,\d\mu(x)<+\infty.
\end{equation}
If $\frm$ is a $\sigma$-finite measure on $\Omega$, 
$L^r(\bOmega,\frm;\R^d)$ ($L^r(\bOmega,\frm)$ if $d=1$) 
$r\in [1,\infty]$, will be
the usual Lebesgue space of (classes of) $\frm$-measurable and
$\R^d$-valued maps $r$-integrable w.r.t~$\frm$. 
We will also use the notation
$L^0(\bOmega)$ for the set of (classes of) measurable real
functions. 
Our typical examples of $\Omega$ consist in 
Borel subsets of some Euclidean
space $\R^h$ or in the space of continuous curves
$\Gamma:=\rmC^0([0,1];\R^d)$.

When $\Omega$ is a Borel subset of $\R^h$
(typically $I$, $\R^d$ or $Q$), 
we will denote by $\LL^h$ 
the (restriction of the) $h$-dimensional Lebesgue measure 
and when $\frm=\LL^h$ we will simply write
$L^r(\Omega;\R^d)$ ($L^r(\Omega)$ when $d=1$).
In the particular case of $Q$ we will also use the symbol
$\lambda:=\LL^{d+1}\restr Q$.
If $\Omega$ is open, the set of  $r$-integrable maps on 
compact subsets of $\Omega$ will be 
denoted by $L^r_{\loc}(\Omega)$. \GGG 

To every probability density
\begin{equation}
  \label{eq:13}
  m\in L^1(\Omega)\quad\text{with}\quad
  m\ge0\text{ a.e.~in $\Omega$}\quad\text{and}\quad
  \int_{\Omega}m(x)\,\d x=1
\end{equation}
we can associate a Borel probability measure
$\mu\in \cP(\Omega)$ by $\mu=m\LL^h$. Conversely, 
if $\mu\in \cP(\Omega)$ is absolutely continuous w.r.t.~$\LL^h$
then its Lebesgue density $m=\d\mu/\d\LL^h$ satisfies \eqref{eq:13}.
In this way, we will often switch between properties stated on
densities $m$ in the convex set of $L^1(\Omega)$
characterized by \eqref{eq:13} and analogous statements for 
measures $\mu=m\LL^h$ in the 
space $\cP^r(\Omega)$ of
Borel probability measures absolutely continuous w.r.t.~$\LL^h$.

If $\omega:\Omega\to (0,\infty)$ is a measurable weight, we
set 
\begin{equation}
  \label{eq:1}
  L^1_\omega(\Omega):=\{g\in L^0(\Omega):\omega g\in
  L^1(\Omega)\}\quad
  \text{endowed with the norm}\quad
  \|g\|_{L^1_\omega}:=\|\omega g\|_{L^1}.
\end{equation}
Our main example will be 
\begin{equation}
\text{the weight in $\R^d$ (or in $Q$)}\quad
\weight(x):=1+|x|^2.\label{eq:97}
\end{equation}
In this case, a function $m$ as in
\eqref{eq:13} belongs to
$L^1_\weight(\R^d)$ if and only if 
the corresponding measure $\mu=m\LL^d$ belongs to the
space $\cP^r_2(\R^d)$ of absolutely continuous, Borel probability
measures with finite quadratic momentum, i.e.
\begin{equation}
  \label{eq:11}
  \int_{\R^d} |x|^2m(x)\,\d x=
  \int_{\R^d}|x|^2\,\d\mu(x)<\infty,
\end{equation}
so that
\begin{equation}
  \label{eq:14}
  \|m\|_{L^1_\weight(\R^d)}=
  \int_{\R^d}(1+|x|^2)m(x)\,\d x=
  1+\int_{\R^d}|x|^2\,\d\mu(x)
  \quad\text{for every $\mu=m\LL^d\in \cP^r_2(\R^d)$.}
\end{equation}
The dual of the space $L^1_\omega(\Omega)$ can be naturally identified
with 
\begin{equation}
  \label{eq:3}
  L^\infty_{1/\omega}(\Omega):=\{h\in L^0(\Omega):\omega^{-1} h\in
  L^\infty(\Omega)\}\quad
  \text{endowed with the norm}\quad
  \|h\|_{L^\infty_{1/\omega}}:=\|\omega^{-1} h\|_{L^\infty}.
\end{equation}
In particular, functions $h\in L^\infty_{1/\weight}(\R^d)$ 
are naturally in duality with measures $\mu=m\LL^d\in \cP^r_2(\R^d)$ 
since $hm\in L^1(\R^d)$ so that
\begin{equation}
  \label{eq:4}
  \int_{\R^d}|h|\,\d \mu(x)=
  \int_{\R^d} |h|\,m\,\d x\le \|h\|_{L^\infty_{1/\weight}}\,\|m\|_{L^1_\weight}<\infty.
\end{equation}
Eventually we will set 
\begin{equation}
  \label{eq:84}
  \cP_{2,p}^r(\R^d):=\Big\{\mu=m\LL^d\in \cP^r_2(\R^d):m\in
  L^p(\R^d)\Big\}.
\end{equation}
A function 
$G:\R^d\times \R^h\to \R$ is a Carath\'eodory function if 
\begin{align*}
  \text{for a.e.~$x\in \R^d$}\quad&y\mapsto G(x,y)\text{ is continuous
  in }\R^h,\\
  \text{for every $y\in \R^h$} \quad&x\mapsto G(x,y)\text{ is Lebesgue measurable in }\R^d.
\end{align*}
If $z:Q\to \R^h$ is a measurable map, 
we adopt the convention to write $G(x,z)$
for the function $(t,x)\mapsto G(x,z(t,x))$ defined in $Q$.
The Carath\'eodory assumption on $G$ guarantees that such a composition
is also measurable.

We say that a real function $f$ defined on some interval $J$ of $\R$
is \emph{increasing} (resp.~\emph{strictly increasing}) if for every $r_1<r_2$ in $J$ it holds $f(r_1)\le
f(r_2)$ (resp.~$f(r_1)<f(r_2)$).

\subsection{Structural assumptions}
We will be mainly concerned with the following 
first order system
\begin{align}
 \label{eq:5}
&\left\{\begin{aligned}
-\partial_t u + H(x,Du) &= f(x,m) && \text{in } Q, \\
\partial_t m  - \nabla\cdot  (mH_\spp(x,Du)) &= 0  && \text{in } Q,
\end{aligned}
\right.
\intertext{with initial and final conditions}
&\quad \, m(0,\cdot) = m_0, \; m(1,\cdot) = m_1\qquad\qquad\ \text{in } \R^d.
\end{align}
We fix a pair of conjugate exponents $p,q \in (1,+\infty)$,
$p^{-1}+q^{-1}=1$ and the weight function $\weight(x):=1+|x|^2$ as in \eqref{eq:97}.
The following assumptions hold true throughout the paper.
\begin{hypothesis}\label{h.1}
\
\begin{enumerate}[\rm ({H}1)]
\item $m_0, m_1$ are nonnegative functions in
  $L^1_\weight(\R^d)$ with equal, normalized, mass:
  \begin{displaymath}
    \int_{\R^d}m_0(x)\,\d x=\int_{\R^d}m_1(x)\,\d x=1,\quad
        \int_{\R^d}|x|^2m_i(x)\,\d x=M_i<\infty;\quad
        \mu_i=m_i\LL^d\in \cP_2^r(\R^d).
  \end{displaymath}
  In most part of our analysis, we will also assume that $m_i\in L^p(\R^d)$.
\item $f: \R^d \times [0, +\infty) \to \R$ is 
  a   Caratheodory function,  
 increasing with respect to the second variable. 
There exist a constant $c_f \ge 1$ and a nonnegative function
$\gamma_f \in L^q(\R^d)$
such that 
\[ \frac{1}{c_f^p} |m|^{p-1} - \gamma_f(x) \leq f(x,m) \leq c_f^p
  |m|^{p-1} + \gamma_f(x) \qquad \forevery m\in [0,\infty)\text{ and
    a.e.~}x\in \R^d.\]
\item   The Hamiltonian $H: \R^d \times \R^d \to \R$ is a
  Caratheodory function  and it is convex and differentiable with respect to its second
variable,
with differential which will be denoted by $H_\spp:\R^d\times
\R^d\to\R^d$.

There exist  constants $c_H\ge 1$, $c_H^\pm>0$ 
with 
\begin{equation}
  \label{eq:128}
  \gamma_H^+(x):=c_H^+(1+|x|),\quad
  \gamma_H^-(x):c_H^-(1+|x|^2)
  \quad x\in \R^d,
\end{equation}
such that
\begin{equation}\label{H_least_quadratic2}
\frac{1}{2c_H}|\pp|^2 - \gamma^-_H(x) \leq H(x,\pp) \leq
\frac{c_H}{2}|\pp|^2 + \gamma^+_H(x)
\quad
\text{for every $\pp\in \R^d $ and a.e.~$x\in \R^d.$}
\end{equation}
\end{enumerate}
\end{hypothesis}
\nc
If we define  the function $F:\R^d\times [0,+\infty)\to \R$ 
and its extension $\tilde F:\R^d\times \R\to \R\cup\{+\infty\}$ by 
\begin{equation}\label{eq:def:F}
F(x,m) :=\int_0^m f(x,\tau)\,\d\tau,\quad
\tilde F(x,m):=
\begin{cases}
  F(x,m)&\text{ if } m\geq 0 \\
  +\infty&\text{ otherwise,}
\end{cases}
\end{equation}
then $F$ is (the restriction of a) Carath\'eodory 
function and \GGG for almost
every $x\in \R^d$ the map $m\mapsto F(x,m)$ 
is convex in $\R$ and differentiable in $(0,\infty)$.
$F$ is also \emph{strictly convex} if $f$ is strictly increasing
w.r.t.~$m$.

Due to Assumption (H2), $F$ satisfies 
\begin{equation}
  \label{eq:7}
 \frac{1}{p c_f^p} |m|^{p} - \gamma_f(x)\, m \leq F(x,m) \leq
 \frac{c_f^p}{p} |m|^{p} + \gamma_f(x)\, m \qquad \forevery m\in
 [0,\infty)\text{ and }x\in \R^d.
\end{equation}
We will denote by $F^*:\R^d\times \R\to \R$ the Fenchel conjugate of
$F$ with respect to the second variable
\begin{equation}
  \label{eq:2}
  F^*(x,a):=\sup_{m\in \R} [am-F(x,m)]=
  \sup_{m\ge 0} [am-F(x,m)].
\end{equation}
 Starting from \eqref{eq:7} it is not difficult to 
check that 
\begin{gather}
  \frac{1}{q c_f^q}\big(a-\gamma_f(x)\big)_+^q \leq F^*(x,a)
  \leq \frac{c_f^q}{q}\big(a+\gamma_f(x)\big)_+^q \qquad \forevery a\in
  \R\text{ and for a.e.~}x\in \R^d,\label{eq:6}\\
  a\mapsto F^*(x,a)\quad \text{is positive and increasing in $\R$
    for a.e.~$x\in \R^d$, $F^*(x,a)=0$ iff $a\le f(x,0)$.}
    \label{eq:6bis}
  \end{gather}
We will also consider the Lagrangian $L:\R^d\times \R^d\to \R$ 
obtained by evaluating the Fenchel conjugate of $H(x,\pp)$  in the variable $-\qq$:
\begin{equation}
  \label{eq:8}
  L(x,\qq)=H^*(x,-\qq):=
  \sup_{\spp\in \R^d} [-\qq \cdot \pp-H(x,\pp)]\quad
  x,\qq\in \R^d.
\end{equation}
In particular $L$ is a Carath\'eodory function, 
convex with respect to the second variable, and satisfies the growth conditions  
\begin{equation}
\frac{1}{2c_H}|\qq|^2 - \gamma^+_H(x) \leq L(x,\qq) \leq
\frac{c_H}{2}|\qq|^2 + \gamma^-_H(x)
\quad 
\forevery x,\qq\in \R^d.\label{eq:9}
\end{equation}
Notice that (H2) and \eqref{eq:7} allow for functions of the form
\begin{equation}
  \label{eq:10}
  f(x,m):= a(x)m^{p-1}+V_f(x),\quad 
  F(x,m):= \frac{1}p a(x)\, m^p+V_f(x)m
\end{equation}
where $a\in L^\infty(\R^d)$ satisfy
$\frac{1}{c_f^p}\le a(x)\le c_f^p$ a.e.~in $\R^d$ and $V_f\in
L^q(\R^d)$.

Similarly, (H3) and \eqref{eq:9} allow for Hamiltonians and Lagrangians
of the form
\begin{equation}
  \label{eq:12}
  \begin{aligned}
    H(x,\pp)&:=\frac12\sum_{i,j=1}^d g^{ij}(x)p_i p_j+\sum_{i=1}^d
    z_i(x)p_i-V_H(x),\\
    L(x,\qq)&:=\frac12\sum_{i,j=1}^d
    g_{ij}(x)(q_i-z_i(x))(q_j-z_j(x))+V_H(x)
  \end{aligned}
\end{equation}
where 
$(g^{ij}(x))_{i,j}$ are the coefficients of a symmetric elliptic
matrix $G(x)\in \mathbb M^{d\times d}$
satisfying $\frac 1{c_H}I\le G(x)\le c_H I$ (in the sense of quadratic
forms), $(g_{ij}(x))_{i,j}$ are the coefficients of the metric tensor $G^{-1}(x)$,
$(z_i(x))_{i}$ are the components of a measurable and bounded vector field
$\zz:\R^d\to\R^d$ and $V_H:\R^d\to \R$ is a measurable potential
satisfying
\begin{displaymath}
  -C(1+|x|)\le V_H(x)\le C(1+|x|^2)
  \quad\text{$\LL^d$-a.e.~in $\R^d$.}
\end{displaymath}

\section{Optimal transport distance, displacement interpolation, and continuity equation}
\label{subsec:Wass}
\subsection{Recaps on Optimal Transport and dynamic formulation}
The set $\cP_2(\R^d)$ of probability measures with finite quadratic
moment
\eqref{eq:11} can be naturally endowed
with the
so-called $L^2$-Kantorovich-Rubinstein-Wasserstein distance $W_2$. If $\mu_0, \mu_1 \in \cP_2(\R^d)$, we define
\begin{equation}\label{wasserstein}
W^2_2(\mu_0, \mu_1) := \min \left\lbrace \int_{\R^d \times \R^d}
  |x_0-x_1|^2\,\d \mmu(x_0,x_1) : \mmu \in \cP(\R^d \times \R^d), \
  \pi^0_\sharp \mmu= \mu_0, \ \pi^1_\sharp \mmu= \mu_1  \right\rbrace,
\end{equation} 
where $\pi^i: \R^d \times \R^d \to \R^d$, $i = 0,1$, stand for the
coordinate projections $\pi^i(x_0,x_1):=x_i$.

Recall that for every   Borel   map $T: X \to Y$ \GGG between two Borel subsets
$X,Y$ of some Euclidean space
and every Borel measure $\mu$ on $X$, $T_\sharp \mu$ denotes \nc the image
measure on $Y$ 
defined by $T_\sharp \mu (A) := \mu (T^{-1}(A))$, for every \GGG Borel
\nc set $A \in Y$. 

\GGG 
It is well known that 
$(\cP_2(\R^d),W_2)$ is a complete metric space  (see
e.g.~\cite[Prop.~7.1.5]{ambrosio2008gradient}) and
that a sequence $(\mu_n)_{n\in\N}$ in $\cP_2(\R^d)$ converges to $\mu$
w.r.t.~$W_2$ as $n\to\infty$ if and only if
\begin{equation}
  \label{eq:35}
  \lim_{n\to\infty}\int_{\R^d}\varphi(x)\,\d\mu_n(x)=
  \int_{\R^d}\varphi(x)\,\d\mu(x)
  \quad\text{for every $\varphi\in C(\R^d)\cap L^\infty_{1/\kappa}(\R^d)$,}
\end{equation}
where $\kappa$ is the usual weight $\kappa(x):=1+|x|^2$.

A plan $\mmu$ attaining the minimum in \eqref{wasserstein} 
is called optimal. 
If $\mmu$ is an optimal plan, McCann's displacement interpolation
\cite{McCann97}
\cite[Sect.~7.3]{ambrosio2008gradient}
\begin{equation}
  \label{eq:38}
  t\mapsto \mu_t:=((1-t)\pi^0+t\pi^1)_\sharp \mmu,\quad t\in [0,1],
\end{equation}
gives rise to 
a (mimimal, constant speed) geodesic curve $(\mu_t)_{t\in [0,1]}$
satisfying
\begin{equation}
  \label{eq:37}
  W_2(\mu_s,\mu_t)=|t-s|W_2(\mu_0,\mu_1)\quad \text{for every }s,t\in [0,1].
\end{equation}
The quadratic momentum is convex along a Wasserstein geodesic:
\begin{equation}
  \label{eq:65}
  \int_{\R^d} |x|^2\,\d\mu_t\le (1-t)
  \int_{\R^d} |x|^2\,\d\mu_0+t \int_{\R^d} |x|^2\,\d\mu_1.
\end{equation}
If 
$\mu_i=m_i\LL^d\in\cP_{2,p}^r(\R^d)$ 
then
all the measures $\mu_t$ given by \eqref{eq:38} admit a Lebesgue
density in $L^p(\R^d)$ satisfying the dispacement convexity inequality
\cite{McCann97}, \cite[Thm.~9.3.9]{ambrosio2008gradient}
\begin{equation}
  \label{eq:39}
  \mu_t=m_t\LL^d,\quad
  \|m_t\|_{L^p(\R^d)}^p\le (1-t) \|m_0\|_{L^p(\R^d)}^p
  +t\|m_1\|_{L^p(\R^d)}^p\quad\text{for every }t\in [0,1].
\end{equation}
\subsubsection*{Absolutely continuous curves in $\cP_2(\R^d)$ and the continuity equation}
It is clear from \eqref{eq:37} that every geodesic is a Lipschitz
curve 
from $[0,1]$ to $\cP_2(\R^d)$. 
Given a metric space $(Y,d_Y)$ (our main examples will be $\R^d$ and $\cP_2(\R^d)$), we will more generally consider
the class $\ac^2([a,b];Y)$ of
absolutely continuous curve with $L^2$ metric velocity: they are
maps $y:[a,b]\to Y$ satisfying
\begin{equation}
  \label{eq:40}
  d_Y(y(s),y(t))\le \int_s^t \rho(r)\,\d r\quad
  a\le s<t\le b\quad
  \text{for some }\rho\in L^2(a,b).
\end{equation}
Whenever $y\in \ac^2([a,b];Y)$, 
the minimal function $\rho$ 
providing the bound \eqref{eq:40}
is given by the metric velocity
\begin{equation}
  \label{eq:41}
  |\dot y|(t):=\limsup_{h\to0}\frac{d_Y\left(y(t+h),y(t)\right)}{|h|}.
\end{equation}
In order to clarify the connection between absolutely continuous
curves $\mu:[a,b]\to\cP_2(\R^d)$ and the continuity equation, let us first observe that 
if $(\mu_t)_{t\in [0,1]}$ is a Borel family of probability measures
(in particular a continuous curve), the formula
\begin{equation}
  \label{eq:42}
  \tilde\mu(h):=\int_0^1 \Big(\int_{\R^d}h(t,x)\,\d\mu_t(x)\Big)\,\d
  t,\quad
  \text{$h:Q\to \R$ bounded and Borel,}
\end{equation}
defines a probability
measure $\tilde\mu \in \cP(Q)$ such that 
\begin{equation}
  \label{eq:43}
  \pi^0_\sharp \tilde\mu=\LL^1\restr{[0,1]},\quad \text{where}\quad \pi^0:Q\to I,\ \pi^0(t,x):=t.
\end{equation}
Conversely, any measure $\tilde\mu \in \cP(Q)$ satisfying
\eqref{eq:43} can be associated to a Borel family $(\mu_t)_{t\in
  [0,1]}$ satisfying \eqref{eq:42} by disintegration
\cite[Sect.~5.3]{ambrosio2008gradient}.
We will occasionally identify $\tilde\mu$ with $\mu$ when  no risk of
ambiguity is possible.

If $\mu\in
\ac^2([0,1];\cP_2(\R^d))$ it is possible to find 
(\cite[Thm.~8.3.1]{ambrosio2008gradient})
a Borel vector field $\vv\in L^2(Q, {\tilde\mu };\R^d)$, i.e.
\begin{equation}
  \label{eq:175}
  \|\vv\|_{L^2(Q,{\tilde\mu };\R^d)}^2=\int_Q |\vv(t,x)|^2\,\d\tilde\mu(t,x)<+\infty,
\end{equation}
such that the pair $(\tilde\mu$, $\vv\,\tilde\mu)$ 
is a distributional solution to the continuity equation
\begin{equation}\label{cont_eq}
\partial_t \tilde\mu + \nabla \cdot (\vv\tilde\mu) = 0 \qquad \text{ in } \DD'(Q),
\end{equation}
even in duality with functions in $C^1_c(\R^{d+1})$:
\begin{equation}
  \label{eq:44}
  \int_Q \Big(\partial_t \phi+D \phi\cdot
  \vv\Big)\,\d\tilde\mu=
  \int_{\R^d}\phi(1,x)\,\d\mu_1(x)-
  \int_{\R^d}\phi(0,x)\,\d\mu_0(x)
  \quad\forevery \phi\in C^1_c(\R^{d+1}).
\end{equation}
It is interesting that any choice of vector field $\vv$ satisfying
\eqref{eq:175} 
induces a vector measure 
\begin{equation}
  \label{eq:176}
  \nnu=\vv\tilde\mu\ll\tilde\mu\text{ with finite total variation } 
  |\nnu|(Q)=\int_Q|\vv|\,\d\tilde\mu\le \|\vv\|_{L^2(Q, {\tilde\mu};\R^d)},
\end{equation}
and a function ($\LL^1$-a.e.~defined) $\rho(r):=\Big(\int_{\R^d}
|\vv(r,x)|^2\,\d\mu_r(x)\Big)^{1/2}$ 
in $L^2(I)$ satisfying \eqref{eq:40} for the distance $d_y:=W_2$ in $\cP_2(\R^d)$.
We can (uniquely) select a minimal one (called the \emph{minimal
  velocity field}) such that 
\begin{equation}
  \label{eq:45}
  \int_{\R^d}
  |\vv(r,x)|^2\,\d\mu_r(x)=\|\dot\mu_t\|_{W_2}^2\quad\text{for
    $\LL^1$-a.e.~$t\in I$.}
\end{equation}
Conversely if $\tilde\mu\in \cP_2(Q)$ 
satisfies \eqref{eq:43} and
together with $\vv\in L^2(Q, {\tilde\mu};\R^d)$ 
gives rise to a solution to the continuity equation
\eqref{cont_eq} then it is possible to prove 
\cite[Lemma 8.1.2, Theorem 8.3.1]{ambrosio2008gradient} that 
$\tilde\mu$ admits a unique disintegration $(\mu_t)_{t\in [0,1]}$ 
associated to a curve $\mu\in \ac^2([0,1];\cP_2(\R^d))$, so that
$\mu$ can be considered as the ``precise representative'' of
$\tilde\mu$.

Thanks to the previous results, there is a natural identification
between curves $\mu\in \ac^2([0,1];\cP_2(\R^d))$ and 
solutions $(\tilde\mu,\vv\tilde\mu)$ of the continuity equation
\eqref{cont_eq}
for some $\vv\in L^2(Q, {\tilde\mu};\R^d)$.
\subsection{$L^p$ probability densities with finite action and
  the Kantorovich-Lebesgue interpolation cost}
\label{subsec:KL}
In the setting of the mean field system \eqref{eq:5}, 
it will be more natural to start from 
a nonnegative density function $m\in L^1(Q)$
\newcommand{\tlm}{\tilde\mu}
associated to the measure $m\lambda$, $\lambda$ being the
restriction of $\LL^{d+1}$ to $Q$.
In order to express in an intrinsic way the regularity 
hidden in the continuity equation, we will introduce the following
definition.
\begin{definition}[Densities with finite $L^2$-action]
  \label{def:fla}
  We say that the probability density (see \eqref{eq:13})
  $m\in L^1_\weight(Q)$ has finite $L^2$-action if 
  there exists a constant $C>0$ such that
  \begin{equation}
    \label{eq:177}
    -\int_Q \partial_t\zeta(t,x)m(t,x)\,\d \lambda\le C
    \Big(\int_Q |D\zeta(t,x)|^2m(t,x)\,\d \lambda \Big)^{1/2}
    \quad\text{for every }\zeta\in C^1_c(Q).
  \end{equation}
  We will denote by $\rmA_2(Q)$ the convex subset of $L^1(Q)$
  of probability densities with finite $L^2$-action.
  \end{definition}
Notice that if $m$ is a probability density in $L^1_\weight(Q)$ and
there exists a measurable vector field $\vv:Q\to \R^d$ such that 
\begin{equation}
  \label{eq:178}
  \int_Q |\vv|^2m\,\d \lambda<+\infty,\quad
  \partial_t m+\nabla\cdot(m\vv)=0\quad\text{in }\DD'(Q),
\end{equation}
then $m\in \rmA_2(Q)$ and we can choose $C:=
\Big(\int_Q |\vv|^2m\,\d \lambda\Big)^{1/2}$ in \eqref{eq:177}. 
In the next lemma we show that the converse is also true.

\begin{lemma}[Precise representative and traces of densities with
  finite $L^2$-action]
  \label{le:precise-m}
  If $m\in \rmA_2(Q)$\\
  then there exists a unique curve
  $\mu\in \ac^2([0,1];\cP_2(\R^d))$ such that
  $m\lambda=\tilde\mu$.
  In particular:
  \begin{enumerate}[\rm 1.]
  \item The traces of $m$ at $t=0$ and at $t=1$ are well
    defined probability measures $\mu_0,\mu_1\in \cP_2(\R^d)$. 
  \item There exists a Borel velocity field $\vv:Q\to\R^d$ 
    such that \eqref{eq:178} holds, and for every 
    vector field satisfying \eqref{eq:178} we have 
    \eqref{eq:44}.
  \item If moreover $m\in L^p(Q)$ then 
    \begin{equation}
      \label{eq:47}
      D_p[\mu]:=\Big\{t\in [0,1]:\mu_t\ll \LL^d,\quad 
      \frac{\d\mu_t}{\d\LL^d}\in L^p(\R^d)\Big\},
    \end{equation}
    is a dense $F_\sigma$ subset in $[0,1]$ (a
countable union of closed sets) of full measure.    
  \end{enumerate}
\end{lemma}
\begin{proof}
We consider the measure $\tilde\mu=m\lambda$ and we introduce the linear subspace 
$V:=\{D\zeta:\zeta\in C^1_c(Q)\}$ in $L^2(Q, \tilde \mu;\R^d)$.
Since the linear map $L:\zeta\mapsto -\int_Q \partial_t
\zeta m\,\d\lambda$ defines a bounded functional in $V$, 
it admits a continuous extension to $L^2(Q, \tilde\mu;\R^d)$ which can be
represented by a vector field $\vv \in L^2(Q, \tilde\mu;\R^d)$ by Riesz
Theorem:
\begin{displaymath}
  L(\zeta)=-\int_Q \partial_t
  \zeta m\,\d\lambda=\int_Q D\zeta\cdot \vv\, m\,\d\lambda\quad
  \text{for every }\zeta\in C^1_c(Q).
\end{displaymath}
We thus find that $(\tilde\mu,\vv\tilde\mu)$ satisfies the continuity
equation \eqref{cont_eq} so that $\tilde\mu$ admits 
a precise representative $\mu\in \ac^2([0,1];\cP_2(\R^d))$ according
to the above considerations.

Finally, if $m\in L^p(Q)$ then Fubini's theorem yields
\begin{equation}
  \label{eq:180}
  \int_Qm^p\,\d\lambda=\int_0^1\Big(\int_{\R^d}m^p(t,x)\,\d x\Big)\,\d
  t<+\infty,\quad
  \int_{\R^d}m^p(t,x)\,\d x \,\,< +\infty,
  \quad\text{for $\LL^1$-a.e.~$t\in (0,1)$},  
\end{equation}
so that $D_p[\mu]$ is of full measure in $[0,1]$ (in
particular, it is dense).
$D_p[\mu]$ can also be considered as the finiteness domain of the
$p$-entropy functional
\begin{equation}
  \label{eq:48}
  \mathscr U_p[\mu]:=
  \begin{cases}
    \int_{\R^d} \big(m(x)\big)^p\,\d x&\text{if }\mu=m\LL^d\ll\LL^d,\\
    +\infty&\text{otherwise},
  \end{cases}
\end{equation}
which is lower semicontinuous w.r.t.~weak convergence of measures. 
$D_p[\mu]$ is therefore an $F_\sigma$ subset of $[0,1]$.
\end{proof}
Notice that if $m\in \rmA_2(Q)$ then 
$m(t,\cdot)$ is a probability
density for $\LL^1$-a.e.~$t\in (0,1)$.
We conclude this discussion with two definitions which will play a
crucial role in Sections \ref{subsec:weak-subsol} and \ref{sec:variational-approach}. 
\begin{definition}[$L^p$ solutions to the continuity equation with
  $L^2$-velocity]\ \\%
  \label{def:CE}%
  Let $p\in (1,+\infty]$. 
  We say that a pair $(m,\vv)$ belongs to the set 
  $\CE 2pQ$ if 
  \begin{enumerate}[\rm (1)]
  \item $m\in L^1_\weight(Q)\cap L^p(Q)$ is a probability density
    and $\vv\in
  L^2(Q,\tilde\mu;\R^d)$ where $\tilde \mu=m\lambda$.
\item The pair $(m,\vv)$ is a solution to the continuity
  equation \eqref{eq:178}. 
\end{enumerate}
In particular $m\in \rmA_{2}(Q)$ and admits a continuous
representative $\mu\in \ac^2([0,1];\cP_2(\R^d))$ according to Lemma
\ref{le:precise-m}.
Whenever $\mu_0,\mu_1\in \cP_2(\R^d)$ are given, we will also set
\begin{equation}
  \label{eq:197}
  \CE2p{Q;\mu_0,\mu_1}:=
  \Big\{(m,\vv)\in \CE 2pQ:
  \mu\restr{t=0}=\mu_0,\ \mu\restr{t=1}=\mu_1\Big\}.
\end{equation}
\end{definition}
The class $\CE 2pQ$ naturally induces a dynamic transport cost between
probability measures, which results from the interaction of the 
Kantorovich action and the $L^p$-penalization of the densities.
\begin{definition}[The Kantorovich-Lebesgue $\KL_{2,p}$ cost]
  \label{def:KL}
  Let $p\in (1,+\infty)$.
  For every $\mu_0,\mu_1\in \cP_2(\R^d)$ and every parameter $a>0$ we set
  \begin{equation}
    \label{eq:184}
    \begin{aligned}
      \KL_{2,p}^{(a)}(\mu_0,\mu_1):= \inf\Big\{&
      \int_Q \Big(\frac a2|\vv|^2m+\frac 1{2a} (m+m^p)\Big)\,\d\lambda:
      (m,\vv)\in \CE 2p{Q;\mu_0,\mu_1} \Big\},
    \end{aligned}
  \end{equation}
  with the usual convention $\KL_{2,p}^{(a)}(\mu_0,\mu_1)=+\infty$ if
  $\CE 2p{Q;\mu_0,\mu_1}$ is empty.
  When $a=1$ we will just write $\KL_{2,p}(\mu_0,\mu_1)$.
  The induced Kantorovich-Lebesgue distance can be defined as
  \begin{equation}
    \label{eq:187}
    d_{\KL_{2,p}}(\mu_0,\mu_1):=\inf_{a>0} \KL_{2,p}^{(a)}(\mu_0,\mu_1).
  \end{equation}
\end{definition}
It is obvious that 
\begin{equation}
  \label{eq:188}
  \KL^{(a)}_{2,p}(\mu_0,\mu_1)\le 
  \max\Big(\frac ab,\frac ba\Big)\KL_{2,p}^{(b)}(\mu_0,\mu_1)
  \quad\text{for every }a,b>0.
\end{equation}
The rescaled costs $\KL^{(a)}_{2,p}$ just correspond to 
the cost $\KL_{2,p}$ but for a continuity equation in the
dilated cylinder 
$(0,a)\times \R^d$:
\begin{equation}
  \label{eq:227}
  \KL_{2,p}^{(a)}(\mu_0,\mu_1):= \inf\Big\{
  \int_0^a\int_{\R^d} \Big(\frac 12|\vv|^2m+\frac 1{2} (m+m^p)\Big)\,\d\lambda:
  (m,\vv)\in \CE 2p{(0,a)\times \R^d;\mu_0,\mu_1} \Big\}   .   
\end{equation}
By a standard rescaling argument
(see e.g.~\cite[Lemma 2.2 and A.5]{RSSU18}) we obtain 
  \begin{equation}
    \label{eq:184bis}
    \begin{aligned}
      d_{\KL_{2,p}}(\mu_0,\mu_1)= \inf\Big\{&
      \int_0^1 \Big(1+\|m_t\|^p_{L^p(\R^d)}\Big)^{1/2}
      \Big(\int_{\R^d}|\vv_t|^2m_t\,\d x\Big)^{1/2}\,\d t:
      (m,\vv)\in \CE 2p{Q;\mu_0,\mu_1}\Big\},
    \end{aligned}
  \end{equation}
so that $d_{KL_{2,p}}\ge W_2$; 
it is not difficult to show that the cost $\KL_{2,p}$ is finite if and only
if $d_{\KL_{2,p}}$ is finite.

Here we do not aim at characterizing the class of measures for which
$\KL_{2,p}$ is finite, we will just point out two important cases.
The first one consists in probability measures with $L^p$ densities 
$\mu\in \cP_{2,p}^r(\R^d)$,
which can also be identified with the subset of
probability densities in $L^p\cap L^1_\weight(\R^d)$.
\begin{lemma}
  \label{le:basic-Lp}
  If $\mu_i=m_i\LL^d\in \cP^r_{2,p}(\R^d)$, $i =1,2$, 
  then 
  $\CE2p{Q;\mu_0,\mu_1}$ is not empty and
  \begin{equation}
    \label{eq:185}
    \KL_{2,p}^{(a)}(\mu_0,\mu_1)\le 
    \frac 1{2a}+\int_{\R^d}\Big(a|x|^2(m_0+m_1)+\frac 1{4a}
    (m_0^p+m_1^p)\Big)\,\d x.
  \end{equation}
\end{lemma}
\begin{proof}
  It is sufficient to choose the McCann's displacement interpolation
  \eqref{eq:38}, 
  yielding
  \begin{displaymath}
    \int_Q |\vv|^2m\,\d\lambda=
    W_2^2(\mu_0,\mu_1)\le 2\int_{\R^d}|x|^2\,\d\mu_0+2\int_{\R^d}|x|^2\,\d\mu_1
  \end{displaymath}
  and 
  \begin{displaymath}
    \int_Q m^p\,\d\lambda\le \frac 12 \big(\|m_0\|_{L^p(\R^d)}^p+\|m_1\|_{L^p(\R^d)}^p\big).
  \end{displaymath}
\end{proof}
The above Lemma shows that the $d_{\KL_{2,p}}$ is a distance on the
set
$\cP^r_{2,p}(\R^d)$. We can denote by $\cP_{2,p}(\R^d)$ its completion w.r.t.~$d_{\KL_{2,p}}$.
The next Lemma will show that this set can be considerably larger than $\cP_{2,p}^r(\R^d)$.
\begin{lemma}
  $\cP_{2,p}(\R^d)$ can be identified with the subset of $\cP_2(\R^d)$
  of measures at finite distance from $\cP_{2,p}^r(\R^d)$:
  \begin{equation}
    \label{eq:226}
    \cP_{2,p}(\R^d)=\Big\{\mu\in \cP_{2}(\R^d):
    \CE2p{Q;\mu,\mu'} \text{ is not empty for some }\mu'\in \cP_{2,p}^r(\R^d)\Big\}.
  \end{equation}
  If $p<1+2/d$ then 
  $\cP_{2,p}(\R^d)=\cP_2(\R^d)$ so that 
  $\KL_{2,p}(\mu_0,\mu_1)<\infty$ 
  for every pair $\mu_0,\mu_1\in \cP_2(\R^d)$.
  
  If $p\ge 1+2/d$ then $
  \cP_{2,p}(\R^d)\supset \cP_{2,p_\star}^r(\R^d)$ with $p_\star>p/(1+2/d)$, 
  so that if 
  $\mu_i=m_i\LL^d\in \cP^r_2(\R^d)$ for some $m_i\in
  L^{p_\star}(\R^d)$ then
  $\KL_{2,p}(\mu_0,\mu_1)<\infty$.
\end{lemma}
\begin{proof}
  Since $d_{\KL_{2,p}}\ge W_2$ and $\cP_2(\R^d)$ is complete, 
  it is easy to check that $\cP_{2,p}(\R^d)\subset \cP_2(\R^d)$.
  It is obvious that a measure $\mu$ in the
  completion $\cP_{2,p}(\R^d)$ can be connected to measures in
  $\cP_{2,p}^r(\R^d)$ with finite cost; conversely, if there exists
  $\mu'=m'\LL^d\in \cP_{2,p}^r(\R^d)$ with
  $(m,\vv)\in \CE2p{Q;\mu,\mu'}$ 
  then $m_t\in \cP_{2,p}^r(\R^d)$ for all $t\in D_p[\mu]$, 
  in particular there exists a decreasing sequence $t_n\downarrow0$
  such that $m_{t_n}\in \cP_{2,p}^r(\R^d)$ and we have
  \begin{align*}
    d_{\KL_{2,p}}(\mu,m_{t_n})
    \le \KL^{t_n}_{2,p}(\mu,m_{t_n})
    \le 
    \frac 12 t_n+\int_0^{t_n}\int_{\R^d} \Big(\frac 12
    |\vv|^2m+m^p\Big)\,\d x\,\d t
    \to 0\text{ as }n\to\infty,
  \end{align*}
  so that $\mu$ belongs to the closure of $\cP_{2,p}^r(\R^d)$ with
  respect to $d_{\KL_{2,p}}$.
  
  Let us now check the last two statements, first considering  
  the case $p<1+2/d$. By the above argument, it is sufficient to show
  that any $\mu_0\in \cP_2(\R^d)$ can be connected to
  a measure $\mu_1\in \cP_{2,p}^r(\R^d)$ by a path $(\mu,\vv)\in \CE2p{Q;\mu_0,\mu_1}$.
  
  In order to find $\mu$ we introduce the Heat semigroup
  $(S_t)_{t\ge0}$ in $\R^d$
  \begin{equation}
    \label{eq:190}
    S_t \mu=\mu\ast g_t,\quad
    g_t(x):=\frac1{(4\pi t)^{d/2}}\mathrm e^{-|x|^2/4t},
  \end{equation}
  and set $\mu_t=m_t\LL^d:=S_t\mu_0$. 
  It is well known that the $L^p$ norm of $m_t$
  obeys the estimate
  \begin{equation}
    \label{eq:194}
    \|m_t\|_{L^p(\R^d)}\le C_p \frac 1{t^{(1-1/p)d/2}};
  \end{equation}
  morever, since $m$ satisfies the Heat equation
  \begin{displaymath}
    \partial_t m-\Delta m=0\quad\text{in }(0,\infty)\times \R^d
  \end{displaymath}
  we see that 
  \begin{equation}
    \label{eq:191}
    \partial_t m+\nabla\cdot (m\vv)=0\quad
    \text{with}
    \quad
    \vv=-Dm/m
  \end{equation}
  and the metric velocity of $m$ w.r.t.~the Wasserstein distance at
  time $t>0$
  coincides with the Fisher information
  \begin{equation}
    \label{eq:192}
    |\dot\mu_t|_{W_2}^2=\int_{\R^d}|\vv(t,x)|^2m(t,x)\,\d x=
    \int_{\R^d}\frac{|D m(t,x)|^2}{m(t,x)}\,\d x.
  \end{equation}
  The convexity of the integrand $(x,\yy)\mapsto |\yy|^2/x$ in
  $(0,\infty)\times \R^d$ and Jensen's inequality yield (see
  \cite[Lemma 8.1.10]{ambrosio2008gradient})
  \begin{equation}
    \label{eq:193}
    \int_{\R^d}\frac{|D m(t,x)|^2}{m(t,x)}\,\d x\le 
    \int_{\R^d}\frac{|D g_t(x)|^2}{g_t(x)}\,\d x= \frac{d}{8\pi t}.
  \end{equation}
  We deduce the upper bound 
  \begin{equation}
    \label{eq:196}
    (1+\|m_t\|^{p}_{L^p(\R^d)})^{1/2}|\dot\mu_t|_{W_2}\le
    Ct^{-\gamma},\quad t\in(0,1),\ 
    \gamma:=\frac 12+\frac d4(p-1),
  \end{equation}
  which is integrable in $(0,1)$ if $p<1+d/2$.

  When $\mu_0=m_0\LL^d$ with $m_0\in L^{p_\star}(\R^d)$, we use
  the same argument replacing \eqref{eq:194} with 
  \begin{equation}
    \label{eq:194bis}
    \|m_t\|_{L^p(\R^d)}\le C_{p_\star,p} \frac 1{t^{(1/p_\star-1/p)d/2}},
  \end{equation}
  and thus obtaining an estimate analogous to \eqref{eq:196} with the exponent
  $\gamma=\frac 12+\frac {pd}4(1/p_\star-1/p)$. The integrability
  condition near $0$ then yields the condition $p_\star>p/(1+2/d)$.
\end{proof}
\section{Convergence in measure, increasing
functions, weighted $L^p$ spaces and anisotropic convolution}
\label{sec:4}
\subsection{The space $L^0(\Omega;\frm)$ and the convergence in measure}
\label{subsec:conv-measure}
 
Let $\Omega$ be a Polish topological space
with its Borel $\sigma$-algebra $\BB$ and 
a $\sigma$-finite Borel measure $\frm$. 
Since $\frm$ is $\sigma$-finite, we can 
 find a 
 \begin{equation}
   \text{l.s.c.~density function $\rho:\Omega\to (0,1]$
 such that $\varrho:=\rho\frm\in \cP(\Omega)$.} 
\label{eq:228} 
\end{equation}
We denoted by $L^0(\Omega,\frm)$ the space of (equivalent classes of)
$\frm$-measurable functions $u: \Omega \to \R$.  
 $L^0(\Omega,\frm)$ is endowed with the topology of the convergence in
 measure
 (on every
 measurable set of finite
 measure):
 recall that a sequence $(u_n)_{n\in \N}$ in $L^0(\Omega,\frm)$ converges to $u\in
 L^0(\Omega,\frm)$ in measure if 
 \begin{equation}
   \label{eq:24}
   \text{for every $\varepsilon > 0, F \subset \BB$ with $\frm(F)
     <\infty$}:
   \quad
   \lim_{n \uparrow +\infty} \frm\left( \lbrace x \in F: |u_n(x) - u(x)| \geq \varepsilon \rbrace  \right) = 0.
 \end{equation} 
It is well known (see also Lemma \ref{l:conv_meas}) that 
this convergence is
equivalent to the convergence in measure w.r.t.~$\varrho$ and it is metrizable,
e.g.~by the distance 
\begin{equation}
\label{eq:78}
  d(u,v): = \int_{\Omega} \left( |u(x)-v(x)| \wedge 1 \right) \,\d\varrho(x).
\end{equation}
If we want to include also functions 
in $L^0(\Omega,\frm;\overline\R)$ with values in the 
extended real line $\overline \R=[-\infty,+\infty]$, 
we can observe that 
every increasing homeomorphism $\zeta:\overline\R\to[-1/2,1/2]$
(e.g.~$\zeta(x):= \frac{x}{2\sqrt{( 1+x^2)}}$ with $\zeta(\pm \infty):= \pm 1/2 $) 
induces a bijection with the set 
\begin{displaymath}
L^0(\Omega,\frm;[-1/2,1/2]):=\{f\in L^0(\Omega,\frm):f(x)\in [-1/2,1/2]\
\text{for $\frm$-a.e.~$x\in \Omega$}\},
\end{displaymath}
via the composition map
$f\mapsto \zeta\circ f$. This correspondence and the last statement 
of Lemma \ref{l:conv_meas} justifies the following definition.
\begin{definition}
We say that a sequence $f_n \in L^0(\Omega,\frm; \overline\R)$, $n\in \N$, converges in
measure to $f \in L^0(\Omega,\frm; \overline \R)$ as $n\to\infty$ 
if 
$\zeta \circ f_n\to \zeta \circ f$ in measure in $L^0(\Omega,\frm)$,
or, equivalently, in any $L^p(\Omega;\varrho)$, $1<p<\infty$. 
\end{definition}

Let us notice that the above definition is independent on the choice
of the map $\zeta$, of $\varrho$ and of $p$.
 We also notice that 
when $\Omega$ is an open subset of $\R^d$, convergence in $L^1_{\loc}(\Omega)$ 
implies convergence in measure, thanks to 
characterization (b) of Lemma \ref{l:conv_meas}.

\subsection{Properties of increasing functions}
\label{subsec:increasing}
 
\GGG
In this section we will deal with (class of) measurable functions
$u\in L^0(I\times \Omega,\tilde\frm)$, $\tilde\frm:=\LL^1\otimes \frm$, which are increasing
w.r.t.~time. In principle, they could be characterized by three
different properties:
\begin{enumerate}[(i)]
\item an integral inequality against $C^1$ function w.r.t.~time 
  (a distributional inequality $-\partial_t u\le 0$ in $\DD'(Q)$ when
  $\Omega=\R^d$):
  this is the most natural way to write a condition 
  invariant w.r.t.~modifications of $u$ in a $\tilde\frm$-negligible set.  
\item $u$ admits a Borel representative $\sfu:I\times \Omega\to \R$ such
  that $\sfu(s,\cdot)\le\sfu(t,\cdot)$ $\frm$-a.e.~in $\Omega$ for every
  $0<s<t<1$.
\item there exists a Borel representative $\sfu:I\times \Omega\to\R$ and 
  a $\frm$-negligible subset $N\subset \Omega$ such 
  that $\sfu(s,x)\le \sfu(t,x)$ for every $x\in \Omega\setminus N$ and
  $0<s<t<1$
  (differently from (ii), the exceptional set $N$ does not depend on
  $s$ and $t$).
\end{enumerate}
We will discuss the (well known) equivalence between (i), (ii), (iii), by
supplementing the analysis with the enucleation of a (suitable) right continuous representative,
the existence of the traces at $t=0$ and $t=1$ in $L^0(\R^d;\bar \R)$,
the approximation by time convolution, 
and the natural variant corresponding to the differential
inequality
$-\partial_t u\le \beta$.

Let us first discuss the case (ii): 
notice that $L^0(\Omega,\frm)$ can be endowed with a partial order by
\begin{equation}
  \label{eq:25}
  u\le v\quad\Leftrightarrow\quad
  u(x)\le v(x)\quad\text{for $\frm$-a.e.~$x\in \Omega$},
\end{equation}
so that we can consider \emph{increasing} maps $t\mapsto u_t\in L^0(\Omega,\frm)$
defined in some subset $J$ of the real line: they satisfy
$u_s\le u_t$ for every $s,t\in J$ with $s<t$. 
We will say that $u$ is \emph{right continuous} if for every right
accumulation point $t$ of $J$ $\lim_{s\downarrow t}u_s=u_t$ in $L^0(\Omega,\frm)$.\nc
\begin{lemma}\label{l:f_increasing}
\
\begin{enumerate}[\rm (1)]
\item Let $(u_n)_{n \in \N}$ be an increasing sequence in $L^0(\Omega,\frm)$.
Then there exists a unique $u \in L^0(\Omega,\frm; \bar \R)$ such that $u_n
\to u$ in measure as $n\to\infty$.
\item Let $t \mapsto u_t$ be an increasing  map from $D$ to
  $L^0(\Omega,\frm)$,
  where $D\subset I$ is a dense set.
Then there exist unique maps $u^-:I\to L^0(\Omega,\frm)$ (resp.~$u^+$) 
and $u_1^- \in
L^0(\Omega,\frm; \R\cup \{+\infty\})$ (resp.~$u_0^+\in L^0(\Omega,\frm;\R\cup\{-\infty\})$)
such that $\lim_{s\in D,\,s \uparrow t}u_s = u_t^-$ in measure for every $t \in
(0,1]$
(resp.~$\lim_{s\in D,\, \downarrow t}u_s = u_t^+$ for every $t\in [0,1)$).
\item If $D\ni t\mapsto u_t$, $D\subset I$ dense,
  is increasing then for every $t\in D$ it holds $u_t^- \leq u_t \leq u_t^+$ and there exists an at most countable set $J_u \subset D$ such that $u_t^- = u_t =
u_t^+$ for every $t \in D \setminus J_u$.
\item If $t\mapsto u_t$ is increasing then 
$t\mapsto u_t^+$ is increasing and right continuous with values in
$L^0(\Omega,\frm)$. $u_t=u_t^+$ for every $t\in I$ if and only if $u$ is
right continuous.
\end{enumerate}
\end{lemma}
\begin{proof}
We will use the notation of Section \ref{subsec:conv-measure}, in
particular the probability measure $\varrho$ and the homemorphism
$\zeta$.

\noindent
(1) We can select a common $\frm$-negligible set $N \subset \Omega$ such that $u_n$ are defined in $\Omega \setminus N$ and 
\begin{equation}
u_m(x) \leq u_n(x) \qquad \forevery x \in \Omega\setminus N \quad \text{ if } m< n.
\end{equation} 
Now define $u(x):= \lim_{n\uparrow \infty} u_n(x) = \sup_{n \in\N}u_n(x)$ for every $x \in \Omega \setminus N$.
If we apply the Lebesgue dominated convergence theorem 
in $L^1(\Omega;\varrho)$ to the sequence $\zeta \circ u_n$, which is
pointwise a.e.~converging towards $\zeta \circ u$ as $n \to +\infty$,
we easily obtain that $u_n \to u$ in measure. 

\noindent
(2) It is not restrictive to consider the case $u_t^-$ and fix $t=1$.
We first choose a sequence $D\ni t_n \uparrow 1$ and define $u_1^-$ applying claim (1).
By monotonicity, 
if $t\ge t_{n}$ it holds \nc
\begin{displaymath}
\int_{\Omega} \left| \zeta(u^-_1(x)) - \zeta(u_t(x)) \right|\,\d
\varrho(x) \GGG =\nc 
\int_{\Omega} \left[ \zeta(u^-_1(x)) - \zeta(u_t(x)) \right]\, \d \varrho(x) 
\leq \int_{\Omega} \left[ \zeta(u^-_1(x)) - \zeta(u_{\GGG t_{n}}(x)) \right]\, \d \varrho(x),
\end{displaymath} 
so that $\lim_{t \uparrow 1} \int_{\Omega} \left| \zeta(u^-_1(x)) -
  \zeta(u_t(x)) \right| d \varrho(x)  =0$, yielding the 
convergence in measure.

\noindent
(3) By monotonicity it is immediate to check that for every $r<s<t$ in $D$
\begin{equation}
u_r \leq u_s^- \leq u_s \leq u_s^+ \leq u_t.
\end{equation}
The function $w(t):= \int_{\Omega} \zeta(u_t(x))\,\d\varrho(x)$ is increasing with $w^\pm (t) = \int_{\Omega} \zeta(u^\pm_t(x))\,\d\varrho(x)$. 
We know that the set $J_w = \lbrace t: w^+(t) > w^-(t) \rbrace$ is at most countable. On the other hand
\begin{equation}
w^+(t) = \int_{\Omega} \zeta(u^+_t(x))d\varrho(x) = \int_{\Omega} \zeta(u^-_t(x))d\varrho(x) = w^-(t), \qquad \forall t \notin J_w.
\end{equation}
Since $u_t^- \leq u_t^+$ and $\zeta$ is strictly monotone, we deduce
$u^-_t = u^+_t$
for every 
$t \in I \setminus J_w $
and we conclude.

\noindent
(4) It follows immediately by the previous claims.
\end{proof}

\GGG
Let us now consider the case of functions defined in $\Omega_I:=I\times \Omega$
endowed with 
the product measure
$\tilde\frm:=\LL^1\otimes \frm$.
We will call $B_{f}(\Omega)$ the space
of bounded Borel maps $\varphi:\Omega\to\R$ vanishing
outside a set of finite measure, i.e.~$\frm\{\varphi\neq 0\}<\infty$.
If $\beta\in L^0(I\times\Omega,\tilde\frm)$ satisfies the property
\begin{equation}
  \label{eq:237}
  \int_{(a,b)\times F}|\beta(t,x)|\,\d\frm(t,x)<\infty\quad
  \text{for every }0<a<b<1,\quad
  F\in \BB,\ \frm(F)<\infty,
\end{equation}
we may select a Borel representative $\tilde \beta$ of $\beta$ such that
\begin{equation}
  \label{eq:238}
  \text{$\tilde \beta(\cdot,x)\in L^1(a,b)$ for every $x\in \Omega$,
    $0<a<b<1$, and }\quad
  B(t,x):=
    \int_{1/2}^t \tilde \beta(s,x)\,\d s,
\end{equation}
is absolutely continuous w.r.t.~$t\in I$ for every $x\in
\Omega$.

We will make extensively use of the standard
regularization technique by convolution. Thus we 
fix a family $h_\tau$, $\tau>0$, of convolution kernels 
\begin{gather}
h
\in C_c^\infty(\R),\quad
h\ge0,\quad 
\supp h\subset (-1,0),\quad
\int_\R h\,\d t=1,
\quad 
h_\tau(t):= \tau^{-1} h({t}/{\tau}).
\label{eq:51}
\end{gather}
If $u$ is a Borel everywhere defined
representative, the set $\Omega_1(u):=\{x\in \Omega:u(\cdot,x)\in
L^1(a,b)\quad\text{for every }0<a<b<1\}$ 
is Borel and we can define the function
\begin{displaymath}
  u_\tau(t,x):=
  \begin{cases}
    \int_\R h_\tau(t-s)u(s,x)\,\d s&\text{if }x\in \Omega_1(u)\\
    0&\text{otherwise,}
  \end{cases}
\end{displaymath}
which is Borel, it is $C^\infty(0,1-\tau)$ w.r.t.~$t$ for every
$x$, and satisfies $u_\tau=u_\tau'$ $\tilde\frm$-a.e.~whenever 
$u=u'$ $\tilde\frm$-a.e.
Finally, we will consider pairs $(\zeta,Z)$ in the sets
\begin{equation}
  \label{eq:239}
  \ZZ:=\Big\{(\zeta,Z):\zeta\in C^0_b(\R),\ \zeta\ge0,\ Z(r)=\int_0^r
  \zeta(s)\,\d s\Big\},\quad
  \ZZ_c:=\Big\{(\zeta,Z)\in \ZZ:\zeta\in C^0_c(\R)\Big\}.
\end{equation}
\begin{lemma}
  \label{le:increasing-main}
  Let $u,\beta\in L^0(I\times \Omega)$ with $\tilde \beta$ and
  $B$ as in \eqref{eq:237} and \eqref{eq:238}.
  The following properties are
  equivalent:
  \begin{enumerate}[(i)]
  \item For every nonnegative $\eta\in C^1_c(I)$ and 
    $\varphi\in B_f(\Omega)$, and
    every pair $(\zeta,Z)\in \ZZ_c$
    \begin{equation}
      \label{eq:221}
      \int_{\Omega_I}\eta'(t)\varphi(x)Z(u(t,x))\,\d\tilde\frm(t,x)\le
      \int_{\Omega_I}\eta(t)\varphi(x)\zeta (u(t,x))\beta(t,x)\,\d\tilde\frm(t,x).
    \end{equation}
    \item There exists a Borel representative $\tilde u$ 
      and a Borel set $D$ of full $\LL^1$-measure in $(0,1)$ such that
      \begin{equation}
        \label{eq:222}
        \tilde u(s,\cdot)-B(s,\cdot)\le \tilde u(t,\cdot)-B(t,\cdot)\quad\text{$\frm$-a.e.~in $\Omega$}
        \text{ for every }s,t\in D,\ s<t.
      \end{equation}
      \item There exists a Borel representative $\sfu$ such that
        $\sfu(s,x)-B(s,x)\le \sfu(t,x)-B(t,x)$ for every $x\in \Omega$ and $s,t\in
        I$, $s<t$.
  \end{enumerate}
  Moreover, if one of the above conditions holds, then
  \begin{enumerate}[{\rm ({I}.1)}]
  \item $\sfu$ can be chosen so that $t\mapsto \sfu(t,x)$ is right
    continuous for every $x\in \Omega$ 
    (this will be called a \emph{precise representative of $u$});
    if $\sfu_1,\sfu_2$ are two precise representatives of $u$ then
    for $\frm$-a.e.~$x\in \Omega$ 
    they satisfy $\sfu_1(\cdot,x)=\sfu_2(\cdot,x)$ 
    everywhere on $I$. In particular $\sfu(\cdot,x)$ do not depend 
    on $\beta$, up to $\frm$-negligible sets.
  \item If $\sfu^-(t,x):= \lim_{s\up t}\sfu(s,x)$, $x\in \Omega$, we have 
    \begin{equation}
      \label{eq:230}
      \sfu^-(t,\cdot)\le \tilde u(t,\cdot)\le \sfu(t,\cdot)\quad
      \text{$\frm$-a.e.~in $\Omega$, for every }t\in D,
    \end{equation}
    and equality holds in \eqref{eq:230} with at most countable exceptions.
  \item $\Omega_1(\sfu)=\Omega$ and 
    $\sfu(t,x)=\lim_{\tau\down0}\sfu_\tau(t,x)$ for every $t\in I$,
    $x\in \Omega$.
  \item
    If $u\in L^1((a,b)\times\Omega,\tilde\frm)$ for every $0<a<b<1$,
    then \eqref{eq:221} holds if and only if 
    it is satisfied with the choice $Z(u)\equiv u$, $\zeta(u)\equiv1$.
  \item
    If $u,b \in L^r((a,b)\times\Omega,\tilde\frm)$ for every $0<a<b<1$
    and some $r\in [1,\infty)$
    then
    \begin{equation}
      \label{eq:231}
      \sfu(t,\cdot)\in L^r(\Omega),\quad
      \lim_{\tau\down0}\|\sfu_\tau(t,\cdot)-\sfu(t,\cdot)\|_{L^r(\Omega,\frm)}=0
      \quad\text{for every }t\in
      (0,1).
    \end{equation}
  \end{enumerate}
\end{lemma}
\begin{proof}
  Let us first consider the equivalence of properties $(i),\, (ii)$, and
  $(iii). $
  It is easy to check that $(iii)\Rightarrow
  (ii)\Rightarrow (i)$. 
  Let us show the implication $(i)\Rightarrow (iii)$.
    
  We consider a pair $(\zeta,Z)\in \ZZ$ so that $Z\circ u$ satisfies
  the integrability condition 
  \begin{equation}
  \label{eq:237u}
  \int_{(a,b)\times F}|Z(u(t,x))|\,\d\frm(t,x)<\infty\quad
  \text{for every }0<a<b<1,\quad
  F\in \BB,\ \frm(F)<\infty,
\end{equation}
which is always satisfied if $(\zeta,Z)\in \ZZ_c$.
We select a Borel representative $u$ such that $Z\circ u\in L^1_{\rm
  loc}(I)$ for every $x\in \Omega$,
  we call 
    $  B_\zeta(t,x):=\int_{1/2}^t \zeta(u(s,x))\tilde
  \beta(s,x)\,\d s$,
  and $w_\zeta(t,x):=Z\circ u-B_\zeta$. 
  We fix $\zeta$ and we initially 
  omit the dependence on $\zeta$ by writing
  $w:=w_\zeta$.
    It is easy to check that 
  for every nonnegative $\eta\in C^1_c(I)$ and $\varphi\in B_f(\Omega)$
  \begin{equation}
    \label{eq:221w}
    \int_{\Omega_I}\eta'(t)\varphi(x)w(t,x)\,\d\tilde\frm(t,x)\le0.
  \end{equation}
  Since $w$ is Borel and
  $\Omega_1(w)=\Omega$,
  Fubini's theorem shows that for every nonnegative $\eta\in C^1_c(I)$
  and $\varphi\in B_f(\Omega)$
  \begin{align*}
    &-\int_{\Omega_I}\eta'(t)\varphi(x)w_\tau(t,x)\,\d\tilde\frm(t,x)=
    -\int_{\Omega}\Big(\int_0^1 \eta'(t)w_\tau(t,x)\,\d t\Big)
    \varphi(x)\,\d\frm(x)\\
    &=\int_{\Omega}\Big(\int_0^1 \eta(t)\partial_t w_\tau(t,x)\,\d t\Big)
    \varphi(x)\,\d\frm(x)=
    \int_{\Omega_I} \eta(t)\varphi(x)\partial_t w_\tau(t,x)\,
    \,\d\tilde\frm(t,x)
    \ge0
  \end{align*}
  We deduce that $\partial_t w_\tau\ge0$ $\tilde\frm$-a.e.~in
  $\Omega_I$ and therefore there exists a $\frm$-negligible set
  $N(\tau)\subset \Omega$ such that $\partial_t w_\tau(\cdot,x)\ge0$
  $\LL^1$-a.e.~in $\R$ for every $x\in \Omega\setminus N(\tau)$.

  Since $\partial_t w_\tau(\cdot,x)$ is a continuous function, 
  we conclude that $\partial_t w_\tau(\cdot,x)\ge0$
  in $I$ for every $x\in \Omega\setminus N(\tau)$.
  Setting $z(r):=r\,h(r)$, a simple calculation shows that for $0<\tau'<\tau''$
  \begin{displaymath}
    h_{\tau''}(r)-h_{\tau'}(r)=
    -
    \int_{\tau'}^{\tau''} \frac 1{\tau^2}z(r/\tau)\,\d \tau
  \end{displaymath}
  so that 
  \begin{displaymath}
    w_{\tau''}(t,x)-w_{\tau'}(t,x)=
     -
     \int_{\tau'}^{\tau''} \frac 1{\tau^2}
     \int_\R z((t-s)/\tau)w(s,x)\,\d s\,\d\tau
  \end{displaymath}
  and a further integration w.r.t.~$\frm$ with weight $\varphi$
  yields
  \begin{displaymath}
    \int_{\Omega}\big(w_{\tau''}(t,x)-w_{\tau'}(t,x)\big)\,\varphi(x)\,\d\frm(x)=
    -
    \int_{\tau'}^{\tau''} \frac 1{\tau^2}
    \Big(\int_{\Omega_I} z((t-s)/\tau)\varphi(x)w(s,x)\,\d \tilde\frm\Big)\,\d\tau\ge0
  \end{displaymath}
  since $z\le 0$. 
  It follows that there exists a $\frm$-negligible set 
  $N(\tau',\tau'')\supset N(\tau')\cup N(\tau'')$ such that
  $w_{\tau'}(t,x)\le w_{\tau''}(t,x)$ 
  for every rational $t\in I\cap \Q$ and $x\in \Omega\setminus
  N(\tau',\tau'')$. Since $w_{\tau'}$ and $w_{\tau''}$ are also
  continuous
  w.r.t.~$t$ if $x\in \Omega\setminus
  N(\tau',\tau'')$ we conclude that 
  $w_{\tau'}(t,x)\le w_{\tau''}(t,x)$ for every $t\in (0,1)$ and $x\in
  \Omega
  \setminus
  N(\tau',\tau'')$.

  By an induction argument, 
  we can therefore select a decreasing sequence $\tau_n\downarrow0$
  and $\frm$-negligible set $N$ such that 
  $t\mapsto w_{\tau_n}(t,x)$ is continuous and increasing
  w.r.t.~$t$ in $I$ and
  $n\mapsto w_{\tau_n}(t,x)$ is decreasing  
  for every $x\in \Omega\setminus N$.
  We can thus define
  \begin{displaymath}
    \sfw(t,x):=\lim_{n\to\infty}w_{\tau_n}(t,x)\quad\text{if }x\in
    \Omega\setminus N; \quad
    \sfw(t,x)=0\text{ otherwise.}
  \end{displaymath}
  Clearly $\sfw$ is increasing for every $x$; moreover, 
  Lebesgue theorem shows that $\sfw(t,x)=w(t,x)$ $\LL^1$-a.e.~in $(0,1)$
  for every $x\in \Omega\setminus N$, so that 
  $\sfw$ is a Borel representative of $w$.
  Finally, 
  $\sfw\ast h_\tau(\cdot,x)=w\ast h_\tau(\cdot,x)=w_\tau(\cdot,x)$
  in $I$ for every $x\in \Omega\setminus N$ so that
  $\sfw$ is right continuous. 

  Let us now turn back to write $w=w_\zeta$, $\sfw=\sfw_\zeta$ and 
  try to remove the dependence on $\zeta$. We consider an increasing sequence 
  $\zeta_n(r)=\nchi(|2^{-n} r|)$ where $\nchi:[0,\infty)\to [0,1]$ is 
  a smooth decreasing function such that $\nchi\restr{[0,1]}\equiv 1$
  and $\nchi\restr{[2,\infty)}\equiv 0$.
  Clearly $Z_n(r)\equiv r$ if $r\in [-2^n,2^n]$ and $Z_n\circ
  Z_{n+1}=Z_n$.
  We set $w_n:=w_{\zeta_n}$, $B_n:=B_{\zeta_n}$, and we call $N(n)$
  the exceptional set $N(\zeta_n)$ we have previously found with $\bar
  N:=\cup_{n\in \N}N(n)$.
  It is easy to check that $B_n(\cdot,x)\to B(\cdot,x)$ locally uniformly in
  $I$ for every $x\in \Omega$ so that 
  $S_n(x,a,b):=\sup_{x\in (a,b)}|B_n(\cdot,x)-B(\cdot,x)|\to 0$ as $n\to\infty$. Setting $\sfu_n:=\sfw_n+B_n$ 
  we clearly have
  $\sfu_n(\cdot,x)=Z_n(u(\cdot,x))=Z_n(Z_{n+1}(u(\cdot,x)))$ 
  $\LL^1$-a.e.~in $I$ for every $x\in \Omega\setminus \bar N$,
  so that the right continuity of $\sfu_n$ yields
  $\sfu_n(\cdot,x)=Z_n(\sfu_{n+1}(\cdot,x))$ everywhere in $I$ for
  every $x\in \Omega\setminus \bar N$.
  For every $0<a<b<1$ and $x\in \Omega\setminus \bar N$
  there exist $a'\in (0,a)$, $b'\in (b,1)$ and
  $\bar n$ sufficiently big so that 
  $\sfu_n(a',x)=Z_n(u(a',x)),\ \sfu_n(b',x)=Z_n(u(b,x)),\ S_n(x,a,b)\le 1$
  for every $n\ge\bar n$.
  We obtain the estimate
  \begin{displaymath}
    (u(a',x)\land 0)-B(a',x)-1\le \sfu_n(t,x)-B(t,x)\le
    (u(b',x)\lor 0)-B(b',x)+1
    \quad\text{if }t\in (a,b),\ 
    n\ge\bar n.
  \end{displaymath}
  We deduce that the sequence $n\mapsto \sfu_n(\cdot,x)$ is definitely constant 
  in every interval $(a,b)$ and therefore the limits
  $\sfu(t,x)=\lim_{n\to\infty}\sfu_n(t,x)$,
  $\sfw(t,x)=\lim_{n\to\infty}\sfw_n(t,x)$ exist for every $t\in I$
  and define right continuous functions for every $x\in
  \Omega\setminus\bar N$
  which satisfy $\sfu(t,x)=\sfw(t,x)+B(t,x)$. Since $\sfw$
  is increasing w.r.t.~$t$ we conclude.
  
  (I.1) and (I.3) follow by the above construction. 
    If now $\tilde u$ satisfies \eqref{eq:222}, since $\sfu=\tilde u$
  $\tilde\frm$-a.e.~in $I\times \Omega$, we may find 
  a subset $D_1\subset D$ of full measure such that $\tilde
  u(t,\cdot)=\sfu(t,\cdot)$ $\frm$-a.e.~in $\Omega$ for every $t\in
  D_1$.
  By Lemma \ref{l:f_increasing} we conclude that $\tilde
  u^+(t,\cdot)=\sfu(t,\cdot)$
  for every $t\in I$, so that \eqref{eq:230} holds.

  Concerning (I.4) it is clear by a limit procedure
  as we did in the previous claim that 
  if \eqref{eq:221} holds for every $(\zeta,Z)\in \ZZ$ 
  and $u\in L^1((a,b)\times \Omega)$ for every $0<a<b<1$, then
  \eqref{eq:221} holds also choosing $\zeta\equiv1$ and the identity
  map as $Z$. To prove the converse implication, we simply observe
  that in the above proof $(i)\Rightarrow (iii)$ 
  we are allowed to choose at the beginning $\zeta=1$ so that 
  the argument directly shows $(iii)$ and therefore $(i)$ for
  arbitrary $(\zeta,Z)\in \ZZ_c$.

  If eventually $u\in L^p((a,b)\times \Omega,\tilde\frm)$ for every $0<a<b<1$ we
  deduce by Fubini's theorem that the set $t\in I: \sfu(t,\cdot)\in
  L^p(\Omega,\frm)$ is of full measure in $I$, in particular there are
  sequences $a_n\downarrow 0$ and $b_n\uparrow 1$ such that
  $\sfu(a_n,\cdot),\sfu(b_n,\cdot)\in L^p(\Omega;\frm)$.
  It is also easy to check that $B(t,\cdot)\in L^p(\R^d)$ for every
  $t\in I$, so that 
  we can then use the obvious bounds
  \begin{equation}
    \label{eq:232}
    \Big(\sfu(a_n,\cdot)-B(a_n,\cdot)\Big)\land0\le \sfu(t,\cdot)
    -B(t,\cdot)\le \Big(\sfu(b_n,\cdot)-B(b_n,\cdot)\Big)\lor
    0\quad
    \text{if }a_n\le t\le b_n
  \end{equation}
  to show that the negative and the positive part of $\sfu(t,\cdot)$
  are uniformly bounded by fixed functions in $L^p(\Omega,\frm)$ 
  in every compact interval of $(0,1)$. \eqref{eq:231} then follows by
  Lebesgue's Dominated Convergence Theorem.
\end{proof}
\begin{remark}
  \label{rem:extension-to-D'}
  \upshape
  It is clear that when $\Omega=\R^d$ and $\frm=\LL^d$ 
  \eqref{eq:221} can be equivalently formulated as
  \begin{equation}
    \label{eq:233z}
    -\partial_t Z(u)\le \zeta(u)\beta\quad\text{in }\DD'(Q)\quad
    \text{for every }(\zeta,Z)\in \ZZ_c.
  \end{equation}
  If moreover $u\in L^1_{\rm loc}(Q)$ then 
  claim (I.4) shows that we can equivalently write
  \begin{equation}
    \label{eq:233}
    -\partial_t u\le \beta\quad\text{in }\DD'(Q).
  \end{equation}
  Notice eventually that if $(\zeta_i,Z_i)\in \ZZ_c$ with 
  $\zeta_1\le \zeta_2\le 1$ then \eqref{eq:233} yields
  \begin{equation}
    \label{eq:103}
    0\le \partial_t Z_1(u)+\zeta_1(u)\beta\le 
    \partial_t Z_2(u)+\zeta_2'(u)\beta\le \partial_t u+\beta\quad\text{in }\DD'(Q).
  \end{equation}
Indeed,  \eqref{eq:103} follows immediately from \eqref{eq:233z} with the
  choices $\zeta(r):=\zeta_1(r)$, $\zeta(r):=\zeta_2(r)-\zeta_1(r)$ and
  $\zeta(r):=r-\zeta_2(r)$.  
\end{remark}
We now study the traces of $u$ at $t=0$ and $t=1$, by assuming that 
$\beta$ satisfies
\begin{equation}
  \label{eq:237strong}
  \int_{I\times F}|\beta(t,x)|\,\d\tilde \frm(t,x)<\infty\quad
  \text{for every }
  F\in \BB,\ \frm(F)<\infty,
\end{equation}
so that we may select a Borel representative $\tilde \beta$ of $\beta$ such that
\begin{equation}
  \label{eq:238strong}
  \text{$\tilde \beta(\cdot,x)\in L^1(I)$},
  \quad
  B(t,x):=
    \int_{1/2}^t\tilde \beta(s,x)\,\d s,\ B(\cdot,x)\in \ac([0,1]) 
    \quad\text{for every $x\in \Omega$}.
\end{equation}

\begin{lemma}[Traces at $t=0$ and $t=1$]
  \label{le:traces}
  Let $u,\beta\in L^0(I\times \Omega,\tilde\frm)$ 
  be satisfying \eqref{eq:237strong} and one of the equivalent conditions of Lemma
  \ref{le:increasing-main} 
  and let $w:=u-B$.
  Let $w_0^+,w_1^-$ be defined as in Lemma \ref{l:f_increasing}
    starting from a Borel representative $\tilde w(t,\cdot)=\tilde u(t,\cdot)-B(t,\cdot)$ satisfying
    \eqref{eq:222} and let
    \begin{equation}
      \label{eq:240}
      u^+_0:=w_0^++B(0,\cdot) \in L^0(\Omega;\R\cup\{-\infty\})\quad
      u_1^-:=w_1^-+B(1,\cdot)\in L^0(\Omega;\R\cup\{+\infty\}),
    \end{equation}
    \begin{equation}
      \label{eq:234}
      \sfu^+(0,x):=\lim_{t\down0}\sfu(t,x)\in [-\infty,+\infty),\quad
      \sfu^-(1,x):=\lim_{t\up1}\sfu(t,x)\in (-\infty,+\infty]
      \quad\text{for every }x\in \Omega.
    \end{equation}
    \begin{enumerate}[(i)]
    \item 
      We have
      \begin{equation}
        \label{eq:189}
        u_0^+=\sfu^+(0,\cdot),\quad
        u_1^-=\sfu^-(1,\cdot)\quad
        \text{$\frm$-a.e.~in }\Omega,
      \end{equation}
    \item
      If there exists $u_0,u_1\in L^0(\Omega,\frm;\bar
  \R)$ such that
  for every nonnegative $\eta\in C^1_c(\R)$ and $\varphi\in
  B_f(\Omega)$ and $(\zeta,Z)\in \ZZ_c$
    \begin{equation}
      \label{eq:221bis}
      \begin{aligned}
        \eta(0)\int_\Omega Z(u_0)\varphi\,\d\frm&-
        \eta(1)\int_\Omega Z(u_1)\varphi\,\d\frm\\&\le
        \int_{\Omega_I}\Big(\eta(t)\varphi(x)\zeta(u(t,x))\beta(t,x)-
        \eta'(t)\varphi(x)Z(u(t,x))\Big)\,\d\tilde\frm(t,x),
      \end{aligned}
    \end{equation}
    then
      \begin{equation}
        \label{eq:189bis}
        u_0\le u_0^+=\sfu^+(0,\cdot),\quad
        u_1\ge u_1^-=\sfu^-(1,\cdot)\quad
        \text{$\frm$-a.e.~in }\Omega.
      \end{equation}
      \item 
        \eqref{eq:221bis} always holds with $u_0,u_1$ replaced by
      $u_0^+$, $u_1^-$ respectively.  
      \item If \eqref{eq:221bis} holds with
      $u_0\land 0,\ u_1\lor 0\in L^1(\Omega,\frm)$ then
      $u\in L^1(I\times \Omega;\tilde\frm)$ and
      \begin{equation}
        \label{eq:221tris}
        \eta(0)\int_\Omega u_0\varphi\,\d\frm-
        \eta(1)\int_\Omega u_1\varphi\,\d\frm\le 
        \int_{\Omega_I}
        \Big(\eta(t)\varphi(x)\beta(t,x)-\eta'(t)\varphi(x) u(t,x)\Big)\,\d\tilde\frm(t,x).
      \end{equation}
    \end{enumerate}
\end{lemma}
\begin{proof}
  \eqref{eq:189} is an immediate consequence of \eqref{eq:230} and of
  the definition of convergence in measure.

  In order to check $(ii)$, 
  we take a smooth function $\nchi\in C^\infty_c(\R)$ taking values in
  $[0,1]$, supported in $(-1,1)$, satisfying 
  $\nchi\equiv 1$ in $[-1/2,1/2]$, and decreasing in $(0,1)$.
  We set $\nchi_\tau(t):=\nchi(t/\tau)$ and observe that as $\tau\down0$
  \begin{align*}
    -\int_0^1\nchi_\tau'(t)Z(\sfu(t,x))\,\d t
    &
      =
      -\int_0^1 Z(\sfu(\tau s,x))
         \nchi'(s)\,\d s\to Z(\sfu^+(0,x))\int_0^1\nchi'(s)\,\d
      s=Z(\sfu^+(0,x)),\\
    \int_0^1\nchi_\tau(t) \zeta(\sfu(t,x))\tilde\beta(t,x)\,\d t&\to 0
  \end{align*}
  a further integration w.r.t.~$x$ after a multiplication by $\varphi$
  and the Lebesgue's Dominated Convergence Theorem
  yield
  \begin{align*}
    \int_\Omega Z(u_0)\varphi\,\d\frm
    \le - 
    \int_\Omega\Big( \int_0^1\nchi_\tau'(t)Z(\sfu(t,x))\,\d
    t-
    \int_0^1\nchi_\tau(t) \zeta(\sfu(t,x))\tilde\beta(t,x)\,\d t\Big)\varphi\,\d\frm(x)
    \to
    \int_{\Omega}Z(\sfu^+(0,\cdot))\varphi\,\d\frm
  \end{align*}
  Since $\varphi$ is arbitrary, we get \eqref{eq:189bis}.

  In order to prove $(iii) $
  it is not restrictive to consider the case
  when $\eta(0)=1$ and $\supp(\eta)\subset (-\infty,1)$. If $\nchi$ is
  as in the previous step, we can define
  $\eta^0_\tau:=\eta\nchi_\tau$, $\eta^1_\tau:=\eta(1-\nchi_\tau)$, 
  observing that $\eta^1_\tau\in C^1_c(I)$ and therefore
  \begin{align*}
    &\int_0^1\Big(\eta(t)\zeta(\sfu(t,x))\tilde\beta(t,x)-\eta'(t)Z(\sfu(t,x))\Big)\,\d
    t
    \ge
    \int_0^1\Big(\eta^0_\tau(t)\zeta(\sfu(t,x))\tilde\beta(t,x)
      -(\eta^0_\tau)'(t)Z(\sfu(t,x))\Big)\,\d
    t\\
    &=
      \int_0^\tau\nchi(t/\tau)\Big(\eta(t)\zeta(\sfu(t,x))\tilde\beta(u(t,x))-\eta'(t)Z(\sfu(t,x))\Big)\,\d t-
      \int_0^1 \eta(\tau s)Z(\sfu(\tau s,x))
         \nchi'(s)\,\d s\to 
             \eta(0)Z(\sfu^+(0,x))
  \end{align*}
  An integration w.r.t.~$\varphi\frm$ and another application
  of Lebesgue's Dominated Convergence Theorem yield
  \begin{align*}
    \int_{I\times\Omega}\Big(\eta(t)\zeta(u(t,x))\beta(t,x)-
    \eta'(t)Z(\sfu(t,x))\Big)\varphi(x)\,\d \tilde\frm(t,x)
    \ge \eta(0)\int_\Omega 
             Z(\sfu^+(0,x))\varphi(x)\,\d\frm(x).
  \end{align*}
  Finally, for what concerns (iv), arguing as in the proof of the last statement of 
  Lemma \ref{le:increasing-main},
  if $u_0 \wedge 0,u_1\vee 0$ are integrable w.r.t.~$\frm$ we deduce 
  the integrability of $u$ in $I\times \Omega$.
  We can then write \eqref{eq:221bis}
  for a sequence $Z_n(t)=\int_0^t \zeta_n(r)\,\d r$ with
  $\zeta_n\uparrow1 $ as $n\to\infty$; passing to the limit as
  $n\to\infty$ we get $Z_n(u)\to u$ pointwise $\frm$-almost everywhere and 
  \eqref{eq:221tris}.
\end{proof}

\subsection{Convolution by anisotropic kernels}
\label{subsec:kernels}
\newcommand{\wu}u
When we deal with functions defined in the cylinder $Q$ 
we will also use space convolutions, induced by 
\begin{gather}
k \in C_c^\infty(\R^d),\quad
k\ge0,\quad \supp k\subset B_1(0),\quad
\int_{\R^d} k\,\d x=1,
\label{eq:50}
\intertext{setting for $\eps,\tau>0$ (recall \eqref{eq:51})}
k_\varepsilon (x):= \varepsilon^{-d}k({x}/{\varepsilon}),\quad
\eta_{\tau,\varepsilon}(t,x) := h_\tau(t)  k_{\varepsilon}(x).
\label{eq:51bis}
\end{gather}
If $\wu \in L^1_{\loc}(Q)$ then the space-time convolution in $\R^{d+1}$
\begin{equation}
  \label{eq:56}
  \wu _{\tau,\eps}:=\wu \ast \eta_{\tau,\eps}\quad\text{is well defined and
    smooth in }(0,1-\tau)\times \R^d.
\end{equation}
We already defined 
the partial convolution $\wu _\tau:=\wu \ast_t\eta_\tau$ by using an
everywhere defined Borel representative $\tilde \wu $ of $\wu $; 
the next Lemma shows that we can select an even better representative
of $\wu _\tau$
which behaves nicely w.r.t.~$\wu _{\tau,\eps}$.
\begin{lemma}
  \label{le:convolution}
  Let $\wu \in L^1_{\loc}(Q)$. 
  There exists a $\LL^d$-negligible set $N\subset \R^d$ and, for every $\tau>0$,
  a measurable map $\sfu_\tau:(0,1-\tau)\times
  \R^d\to \R$ which coincides with $\wu _\tau$ $\lambda$-a.e.~in $Q$ and 
  satisfies the following properties:
  \begin{enumerate}[\rm (1)]
    \item
      The maps $t\mapsto \sfu_\tau(t,x)$ belong to
      $C^\infty(0,1-\tau)$ for every $x\in \R^d\setminus N$ 
with
\begin{equation}
  \label{eq:59}
  \int_F \|\sfu_\tau(\cdot,x)\|_{L^\infty(a,b)}\,\d x<\infty
  \quad\text{for every bounded $F\subset
    \R^d$ and for every $0<a<b<1-\tau$,}
\end{equation}
and the function 
\begin{equation}
\text{$t\mapsto \sfu_\tau(t,\cdot)$ defined in
$(0,1-\tau)$ with values in $L^1_{\loc}(\R^d)$ is
continuous.}
\label{eq:61}
\end{equation}
    \item
      For every bounded measurable $F\in \R^d$ and every 
      $0<a<b<1-\tau$ 
      \begin{equation}
       \label{eq:62}
       \lim_{\eps\down0} \int_F 
       \Big(\sup_{s\in [a,b]}
       |\wu _{\tau,\eps}(s,x)-\sfu_\tau(s,x)|\Big)\,\d x=0.
     \end{equation}
     In particular 
     \begin{equation}
     \lim_{\eps\down0} \sup_{s\in [a,b]}\int_F 
       \Big(
       |\wu _{\tau,\eps}(s,x)-\sfu_\tau(s,x)|\Big)\,\d x=0,\quad
       \lim_{\eps\down0} \int_{a}^{b} \int_F 
       |\wu _{\tau,\eps}(s,x)-\sfu_\tau(s,x)|\,\d s\d x=0,
       \label{eq:63}
     \end{equation}
     and
     \begin{equation}
     \text{$\wu _{\tau,\eps}(t,\cdot)\to
     \sfu_\tau(t,\cdot)$ in $L^1_{\loc}(\R^d)$ as $\eps\down0$ 
     for every $t\in (0,1-\tau)$.}
   \label{eq:64}  
   \end{equation}
   \item
       $\lim_{\eps\down0}\wu _{\tau,\eps}(\cdot,x)=\sfu_\tau(\cdot,x)$ locally uniformly in $C(0,1-\tau)$
       for every $x\in \R^d\setminus N$.
     \item $\lim_{\tau\down0}\sfu_\tau=\wu $ in $L^1_{\loc}(Q)$. 
   \end{enumerate}
   If moreover $\wu $ satisfies one of 
   the equivalent properties of Lemma \ref{le:increasing-main}
   then we may assume that $\sfu_\tau(\cdot, x)$ 
   coincides with the precise representative of Lemma
   \ref{le:increasing-main} for every $x\in \R^d\setminus N.$
\end{lemma}
\begin{proof}
  Let us consider an increasing sequence of intervals $[a_n,b_n]\uparrow
  (0,1)$ as $n\to\infty$
  and an increasing sequence of bounded open balls $B_k$
  centered at $0$ and of radious $k$, so that $B_k\uparrow \R^d$
  as $k\to\infty$. By selecting a Borel representative $\tilde
  w:Q\to\R$ of $\wu $ and applying Fubini's theorem 
  to the restrictions $\wu _{n,k}$ of $\tilde \wu $ to $(a_n,b_n)\times
  B_{k+1}$, 
  we obtain a countable family of strongly measurable maps
  $\wu _{n,k}:B_{k+1}\to L^1(a_n,b_n)$ which belong to the Lebesgue-Bochner space
  $L^1(B_{k+1};L^1(a_n,b_n))$. We can thus consider the set of Lebesgue
  points $L_{n,k}$ of $\wu _{n,k}$ in $B_k$: we know that 
  $\LL^d(B_k\setminus L_{n,k})=0$ and
  each $\bar x\in L_{n,k}$ satisfies
  \begin{equation}
    \label{eq:182}
    \wu _n(\cdot,\bar x)\in L^1(a_n,b_n),\quad
    \lim_{r\down0}\Big\|\int_{B_r(\bar x)}\wu _n(\cdot,x)\,\d x-\wu _n(\cdot,\bar
    x)\Big\|_{L^1(a_n,b_n)}=0.
  \end{equation}
  We thus set $N:=\bigcup_{k,n\in \N}(B_k\setminus L_{n,k})$ and
  \begin{equation}
  \sfu_\tau(t,x):=\int_\R \tilde \wu (s,x)h_\tau(t-s)\,\d s\quad
  t\in (0,1-\tau),\quad
  x\in \R^d\setminus N;\quad
  \sfu_\tau(t,x)=0\text{ if }x\in N.
\end{equation}
Let us now fix $\tau>0$, an interval $(a,b)$ with $0<a<b<1-\tau$, 
a bounded measurable set $F\subset \R^d$, and integers $n,k$
sufficiently big so that $(a,b+\tau)\subset (a_n,b_n)$ and $F\subset
B_k$ and therefore the restriction of $\tilde \wu $ to
$(a,b)\times F$ coincides with the restriction of $\bar \wu :=\wu _{n,k}$;
in particular
 \begin{equation}
  \sfu_\tau(t,x)=\bar \wu _\tau(t,x)=\int_{a_n}^{b_n}\bar \wu (s,x)h_\tau(t-s)\,\d s\quad
  \text{if }t\in (a,b),\quad
  x\in F\setminus N.
\end{equation}
By Fubini's theorem we know that 
$\int_F \| \bar \wu \|_{L^1(a,b)}\,\d x<\infty$,
so that the maps $t\mapsto \sfu_\tau(t,x)$ belong to
$C^\infty(a,b)$ for every $x\in F\setminus N$ 
and satisfies \eqref{eq:59}.
In order to check \eqref{eq:61} 
we can use the Lebesgue Dominated Convergence theorem
and the estimate \eqref{eq:59}, since
for every $x\in F\setminus N$ and $t\in
(a,b)$ we have
$\lim_{s\to t}\sfu_\tau(s,x)=\sfu_\tau(t,x)$ and if $s\in (a,b)$
$|\sfu_\tau(s,x)|\le \|\sfu(\cdot,x)\|_{L^\infty(a,b)}$.
Being $F$ and $(a,b)$ arbitrary we conclude the proof of the first
claim.

Let us now consider the claim (2).
Since the linear map $\zeta\mapsto \zeta\ast h_\tau$ is
well defined (and continuous) from $L^1(a_n,b_n)$ to
$C([a,b])$, the map 
$\bar \wu ^x_\tau:=\bar \wu (\cdot,x)\ast h_\tau$ 
is strongly measurable from $B_{k+1}$ to $C([a,b])$ 
and the set of its Lebesgue points surely contains $F\setminus N$.
Performing a convolution
w.r.t.~$x\in \R^d$ with the kernel $k_\varepsilon$, $\eps<1$, and values in the
Banach space $C([a,b])$ we thus obtain a continuous map
$x\mapsto \bar \wu ^x_{\tau,\eps}$ from $B_k$ to $C([a,b])$;
$\bar \wu ^x_{\tau,\eps}$ is defined by the Bochner integral
$\bar \wu ^x_{\tau,\eps}(t):=\int_{\R^d} \bar \wu^y_\tau(t) \, k_\eps(x-y)\,\d
y$
and it
is not difficult to check that
$\bar \wu ^x_{\tau,\eps}(t)=\wu _{\tau,\eps}(t,x)$ for every $x\in B_k$ and
$t\in [a,b]$.

\eqref{eq:62} then follows by general results on convolutions for
Banach-valued functions. \eqref{eq:63} are obvious consequences of
\eqref{eq:62}
and \eqref{eq:64} follows by the first limit of \eqref{eq:63}.

In order to check the third claim, we apply Lebesgue theorem for the Bochner integral in the Banach
space $C([a,b])$: at every Lebesgue point $x\in B_k\setminus N$ 
  \begin{displaymath}
    \lim_{\eps\down 0}\bar \wu ^x_{\tau,\eps}=\bar \wu ^x_\tau\quad
    \text{uniformly in }C([a,b]).
  \end{displaymath}
  
  Claim (4) can be proved by adapting the well known result for
  convolution with compactly supported kernels.

  Finally, if $\sfu_\tau'$ is the precise representative of Lemma
  \ref{le:increasing-main},
  up to modifying $N$ by a further $\LL^d$-negligible set, 
  we know that $\sfu_\tau(\cdot,x)=\sfu_\tau'(\cdot,x)$ 
  $\LL^1$-a.e.~in $(0,1)$ for every $x\in N$. Since both the functions
  are continuous w.r.t.~$t\in (0,1)$ they should coincide.
\end{proof}

\nc
\subsection{The space $L^q +L^\infty_{1/\weight}(\Omega)$}
\label{subsec:weight}
\GGG Let us consider an open set $\Omega\subset \R^h$ 
and a measurable weight $\omega:\Omega\to [1,\infty)$. 
As we said in Section \ref{subsec:notation} 
our main examples will be $\Omega=\R^d$ or $\Omega=Q$ with the weight
$\kappa(x):=1+|x|^2$ of \eqref{eq:97}.

The space $L^p\cap L^1_\omega(\Omega)$ can be endowed with the Banach norm
\begin{equation}
  \label{eq:15}
  \|g\|_{L^p\cap L^1_\omega(\Omega)}:=\max\Big(\|g\|_{L^1_\omega(\Omega)},\|g\|_{L^p(\Omega)}\Big).
\end{equation}
Its dual admits a sum representation.
\begin{definition}
\label{def:spaceX}
Let $\omega:\Omega\to [1,\infty)$ a measurable weight. 
We call  $\cX_\omega^q(\Omega):=L^q(\Omega)+L^\infty_{1/\omega}(\Omega)$ the space
of functions 
$v \in L^0(\Omega)$ admitting a decomposition 
\begin{equation}
  \label{eq:18}
  v=w+z,\quad w \in L^\infty_{1/\omega}(\Omega), z \in L^q(\Omega),\quad
\end{equation}
with norm $\| v \|_{\cX_\omega^q(\Omega)} := 
\inf \lbrace \| w \|_{L^\infty_{1/\omega}(\Omega)} + \| z\|_{L^q(\Omega)} :  v = w + z \rbrace$.

We simply denote by  $\cX^q(\Omega)$ the space defined when $\omega$ is the   standard weight $\kappa= 1+ |x|^2$. 
\end{definition}

We can equivalently characterize the norm of $\cX_\omega^q(\Omega)$
as the inf-convolution
\begin{equation}
  \label{eq:16}
  \| v \|_{\cX_\omega^q(\Omega)}=\inf_{z\in L^q(\Omega)}\|v-z\|_{L^\infty_{1/\omega}(\Omega)}+\|z\|_{L^q(\Omega)};
\end{equation}
\eqref{eq:18} is also equivalent to 
\begin{equation}
  \label{eq:19}
  v = \omega\,\tilde w + z, \qquad \tilde w \in L^\infty(\Omega),\ z \in L^q(\Omega).
\end{equation}

We collect in the next Lemma a list of useful properties of
$\cX_\omega^q(\Omega)$.
 
\begin{lemma}
\label{le:Xq}
$\cX_\omega^q(\Omega)$ is a Banach space which can be (isometrically)
identified with the dual of $L^p\cap L^1_\omega(\Omega)$;
in particular it holds
\begin{equation}
  \label{eq:17}
  \|v\|_{\cX_\omega^q(\Omega)}=\sup\Big\{\int_{\Omega} vg\,\d x:g\in L^p\cap
  L^1_\omega(\Omega),\ \|g\|_{L^p\cap L^1_\omega(\Omega)}\le 1\Big\}.
\end{equation}
Moreover
\begin{itemize}
\item[(a)] The infimum in the defnition of the norm of $\cX_\omega^q(\Omega)$
  given by Definition \ref{def:spaceX} is attained and the minimizer is unique.
\item[(b)] If $v \geq 0$ we can restrict the infimum to nonegative pairs $w,z$.
\item[(c)] Similarly, it is not restrictive to assume that $w$ and $z$ share the same sign of $v$, i.e. $w,z \geq 0$ in $\lbrace x \in \Omega: v \geq 0 \rbrace$ and $w,z \leq 0$ in $\lbrace x \in \Omega: v \leq 0 \rbrace$.
\item[(d)] For every $v\in \cX_\omega^q(\Omega)$ the function $|v|$ belongs to 
  $\cX_\omega^q(\Omega)$ and $\| v \|_{\cX_\omega^q} = \| |v| \|_{\cX_\omega^q}$.
\item[(e)] if $0 \leq v_1 \leq v_2$ pointwise a.e.~in $\Omega$ and
  $v_2\in \cX_\omega^q(\Omega)$ then 
  $v_1\in \cX_\omega^q(\Omega)$ and 
  $\| v_1 \|_{\cX_\omega^q} \leq \| v_2 \|_{\cX_\omega^q}$. 
  In particular $ \| v \wedge 0\|_{\cX_\omega^q}, \| v \vee 0
  \|_{\cX_\omega^q} \leq \| v\|_{\cX_\omega^q}
  \leq \|v \vee 0 \|_{\cX_\omega^q} + \| v \wedge 0 \|_{\cX_\omega^q}$.
\item[(f)]
  If $v \geq 0$ and there exists a constant $C\ge0$ such that 
  \begin{equation}
    \label{eq:20}
    \int_{\Omega} vf\,\d x \leq C \max\left( \int_{\Omega}\omega 
      f \,\d x, \| f\|_{L^p(\Omega)} \right)
    ,
\end{equation}
for every nonnegative $f\in L^p\cap L^1_\omega(\Omega)$
      with
      bounded support, 
      then $v \in \cX_\omega^q(\Omega)$ and $\| v \|_{\cX_\omega^q} \leq C$. 
\end{itemize}	
\end{lemma}
\begin{proof}
  The duality with $L^p\cap L^1_\omega(\Omega)$ and \eqref{eq:17}
  follows by a general result on the dual of the intersection of
  Banach spaces, see e.g.~\cite[Theorem 2.7.1]{Bergh-Lofstrom76}.
  It is also easy to check that the infimum is attained, since bounded
  sets in $L^q$ (resp.~$L^\infty_{1/\omega}$) are weakly
  (resp.~weakly$^*$) relatively compact and the norm is weakly
  (resp.~weakly$^*$) lower semicontinuous. Since the $L^q$ norm is
  strictly convex, the minimizer is unique.
  
  In order to check (b) it is sufficient to notice that if $z,w$
  satisfy \eqref{eq:18} and $v\ge 0$, then 
  $z_1:=0\lor z\land v$ and $w_1 := v-z_1$ provide a pair of
  functions, still satisfying \eqref{eq:18} with $0\le
  z_1\le z\vee 0$ and $0\le w_1\le w \vee 0$ so that 
  $\|z_1\|_{L^q(\Omega)}\le \|z\|_{L^q(\Omega)}$ and 
  $\|w_1\|_{L^\infty_{1/\omega}(\Omega)}\le 
  \|w\|_{L^\infty_{1/\omega}(\Omega)}.$
  A similar argument, localized to the sets 
  $\lbrace x \in \Omega: v \geq 0 \rbrace $ and $\lbrace x \in \Omega: v
  \leq 0 \rbrace $, yields (c).
  
  Let us now consider statement (d); first of all, by the previous
  claim  (c),
  if $v = w+z$ is the optimal decomposition of $v$ we have
  $|v| = |w| + |z|$ so that $\| |v| \|_{\cX_\omega^q} \leq \|  v \|_{\cX_\omega^q}$. On the other hand, if $|v| = w+z$
  is the optimal decomposition of $|v|$, we have 
  $ v = w \sign(v)+z\sign(v)$ so that $\| v \|_{\cX_\omega^q} \leq \| |v| \|_{\cX_\omega^q}$.
  
  Claim (e) follows by the following remark: 
  if $v_2 = w_2 + z_2$ is the optimal decomposition of $v_2$
  (so that $z_2,w_2$ are nonnegative by (b))
  we
  may set $z_1:=z_2\land v_1$, $w_1:=v_1-z_1\le w_2$, obtaining an admissible
  decomposition for $v_1$ with $\|z_1\|_{L^q(\Omega)}\le
  \|z_2\|_{L^q(\Omega)}$ and $\|w_1\|_{L^\infty_{1/\omega}(\Omega)}\le
  \|w_2\|_{L^\infty_{1/\omega}(\Omega)}$ which shows
  $\|v_1\|_{\cX_\omega^q}\le \|v_2\|_{\cX_\omega^q}$.
  
  Finally, since functions with bounded support are dense in $L^p\cap
  L^1_\omega(\Omega)$
  it is easy to see that \eqref{eq:20} yields $v\in \big(L^p\cap
  L^1_\omega(\Omega)\big)'=\cX_\omega^q(\Omega)$; we
  can then apply \eqref{eq:17}.
\end{proof}

\section{Weak subsolutions to Hamilton-Jacobi equations}
\label{sec:weak-subsolutions}

In this section we study some properties of weak subsolutions to
Hamilton-Jacobi equations in $Q=(0,1)\times \R^d$. 
\subsection{Weak subsolutions, precise representative and truncations}
\label{subsec:weak-subsol}

We start with the following
\begin{definition}
Given an Hamiltonian $H$ satisfying assumptions \ref{h.1} and
a function $\alpha\in L^1_{\loc}(Q)$, 
a \emph{weak subsolution} 
to the equation
\begin{equation}\label{HJ_alpha}
-\partial_t u + H(x,Du) \leq \alpha,
\end{equation}
is a function  $u \in L^1_{\loc}(Q)$ with distributional gradient
$Du\in L^2_{\loc}(Q;\R^d)$ 
satisfying the inequality \eqref{HJ_alpha}
in the sense of distributions, i.e.
\begin{equation}\label{wHJ_sub}
\int_{Q}\Big( u  \, \partial_t \xi + \big( H(x, Du) - \alpha \big)
\xi\Big)\,\d x\,\d t   \leq 0
\end{equation}
for every non-negative $\xi \in C_c^\infty(Q)$.
\end{definition}

It is easy to check that any weak subsolution, according to the above
definition, actually satisfies  \eqref{wHJ_sub}   for every $\xi \in
C_c^0(Q)$ with $\partial_t \xi \in C_c^0(Q)$.

\begin{remark}\label{r:Hto2}
Thanks to the growth condition {\em (H3)} on the Hamiltonian
(see Assumptions \ref{h.1}), many estimates can be derived by 
looking at the model case 
\GGG associated to the Hamiltonian $H(x,\pp):=\frac 1{2c_H}|\pp|^2$
and to the equation 
\begin{equation}\label{HJ_alpha_2}
-\partial_t u + \frac1{2c_H} |Du|^2 \leq \beta.
\end{equation}  
Notice indeed that if $u$ is a weak subsolution to \eqref{HJ_alpha}
then 
$
u $ is a subsolution to \eqref{HJ_alpha_2}, with 
$\beta= 
\alpha + \gamma^-_H 
$. 
In particular, if $\alpha\in \cX^q(Q)$  then also $\beta\in \cX^q(Q)$. 
\end{remark}
Weak subsolutions exhibit a nice behaviour with respect to truncations.
\GGG
\begin{lemma}[Truncation]\label{l:truncation1}
Let $u_i$, $i=1,2$, be subsolutions of \eqref{HJ_alpha}
with respect to $\alpha_i\in L^1_{\loc}(Q)$, and let $\nchi_i$ be the
characteristic functions
\begin{equation}
  \label{eq:29}
  \nchi_1:=\nchi_{\{u_1\ge u_2\}},\quad
  \nchi_2:=\nchi_{\{u_2>u_1\}}.
\end{equation}
Then 
\begin{align}
  \label{eq:30}
  u_1\lor u_2\quad&\text{is a subsolution of \eqref{HJ_alpha} w.r.t.}\quad
                    \nchi_1\alpha_1+\nchi_2\alpha_2,
                    \\
\label{eq:32}  u_1\land u_2\quad&\text{is a subsolution of
                                  \eqref{HJ_alpha} w.r.t.}\quad 
                                  \nchi_2\alpha_1+\nchi_1\alpha_2.
\end{align}
\end{lemma}
\begin{proof}
\GGG
We prove only \eqref{eq:30}, by adapting the classical Stampacchia's
truncation argument to distributional inequalities. The proof of
\eqref{eq:32} is completely analogous. 

\noindent \emph{Step 1: if $u_i,\beta_i\in L^1_{\loc}(Q)$ satisfy
  $-\partial_t u_i\le \beta_i$ in $\DD'(Q)$ then 
  $-\partial_t(u_1\lor u_2)\le \beta$ where
  $\beta:=\nchi_1\beta_1+\nchi_2\beta_2
  $
  and $\nchi_i$ are defined as in \eqref{eq:29}.}

Recalling that for every $r_1,r_2\in \R$ $r_1\lor r_2=r_1+(r_2-r_1) \vee 0$,
we may consider the regularized truncations depending on $\eps>0$
\begin{displaymath}
  S_\eps(r_1,r_2):=r_1+T_\eps(r_2-r_1),\quad
  T_\eps(r):=
  \begin{cases}
    0&\text{if }r\le 0,\\
    \sqrt{\eps^2+r^2}-\eps&\text{if }r>0,
  \end{cases}
\end{displaymath}
whose derivatives $S_{\eps,i}:=\frac\partial{\partial r_i}S_\eps$ satisfy
\begin{displaymath}
  S_{\eps,1}(r_1,r_2)=
  \begin{cases}
    1&\text{if }r_1\ge r_2,\\
    1-\frac{r_2-r_1}{\sqrt{\eps^2+(r_2-r_1)^2}}&\text{if }r_2>r_1;
  \end{cases}
  \qquad
  S_{\eps,2}(r_1,r_2)=
  \begin{cases}
    0&\text{if }r_1\ge r_2,\\
    \frac{r_2-r_1}{\sqrt{\eps^2+(r_2-r_1)^2}}&\text{if }r_2>r_1.
  \end{cases}
\end{displaymath}
Notice that $S_\eps$ are of class $C^1$, Lipschitz, their derivatives satisfy 
$0\le S_{\eps,i}\le 1$, and 
\begin{equation}
  \label{eq:33}
  S_\eps(r_1,r_2)\up r_1\lor r_2,
  \quad
  S_{\eps,1}(r_1,r_2)\downarrow \nchi_{\{r_1\ge
    r_2\}},\quad
  S_{\eps,1}(r_1,r_2)\uparrow \nchi_{\{r_1<
    r_2\}}
  \qquad\text{as $\eps\down0$}.
\end{equation}
Let $\eta^\delta$, $\delta>0$, be a usual family of nonnegative, smooth
regularization kernels with compact support in $\R^{d+1}$. 

Let us choose a nonnegative test function $\zeta\in C^\infty_c(Q)$ and an open
subset $G$ with compact support in $Q$ such that $\supp(\zeta)\subset
G$;
choosing
$\delta$ sufficiently small, $u_i^\delta:=u_i\ast \eta^\delta$ and
$\beta_i^\delta:=\beta_i\ast \eta^\delta$ are well defined and smooth on
$G$ and satisfy $-\partial_t u_i^\delta\le \beta_i^\delta$ in the
classical sense. Since $S_{\eps,i}\ge0$ we thus get
\begin{equation}
  \label{eq:34}
  -\partial_t(S_\eps(u_1^\delta,u_2^\delta))=
  -S_{\eps,1}(u_1^\delta,u_2^\delta)\partial_tu_1^\delta-
  S_{\eps,2}(u_1^\delta,u_2^\delta)\partial_tu_2^\delta
  \le 
  S_{\eps,1}(u_1^\delta,u_2^\delta)\beta^\delta_1+
  S_{\eps,2}(u_1^\delta,u_2^\delta)\beta^\delta_2
\end{equation}
pointwise in $G$, so that 
\begin{displaymath}
  \int_Q S_\eps(u_1^\delta,u_2^\delta)\,\partial_t\zeta \,\d t \d
  x\le 
  \int_Q \Big(S_{\eps,1}(u_1^\delta,u_2^\delta)\beta^\delta_1+
  S_{\eps,2}(u_1^\delta,u_2^\delta)\beta^\delta_2\Big)\zeta\,\dt \d x.
\end{displaymath}
We can first pass to the limit as $\delta\down0$, using the fact that
$u_i^\delta$ (resp.~$\beta_i^\delta$) converges to $u_i$
(resp.~$\beta_i$)
pointwise a.e.~and strongly in $L^1(G)$, obtaining
\begin{displaymath}
  \int_Q S_\eps(u_1,u_2)\,\partial_t\zeta \,\d t \d
  x\le 
  \int_Q \Big(S_{\eps,1}(u_1,u_2)\beta_1+
  S_{\eps,2}(u_1,u_2)\beta_2\Big)\zeta\,\dt \d x.
\end{displaymath}
We can eventually pass to the limit as $\eps\down0$ using
\eqref{eq:33} and the uniform bounds $0\le S_{\eps,i}\le 1$.

\noindent{\em Step 2: if $u_i,v_i\in L^1_{\loc}(Q)$ satisfy
  $\partial_{x_k}u_i=v_i$ in $\DD'(Q)$ then
  $\partial_{x_k}(u_1\lor u_2)=\nchi_1v_1+\nchi_2 v_2$ 
  and $v_1=v_2$ a.e.~on the set $\{u_1=u_2\}.$}

This is  a well-known property and, in any case, it follows from  the previous claim
(for $\partial_{x_k}$ instead of $\partial_t$), observing that a
corresponding statement
holds also for the inequalities $-\partial_{x_k}u_i\ge v_i$.
The fact that $v_1=v_2$ a.e.~on the set $\{u_1=u_2\}$ follows by 
interchanging the order of $u_1$ and $u_2$ in the formula.

\noindent{\em Conclusion.}
We first apply claim 1, choosing $\beta_i:=\alpha_i-H(x,Du_i)$
obtaining
\begin{displaymath}
  -\partial_t(u_1\lor u_2)\le
  \nchi_1\alpha_1+\nchi_2\alpha_2-\Big(\nchi_1H(x,Du_1)+\nchi_2H(x,Du_2)\Big);
\end{displaymath}
we eventually observe that $\nchi_1$ and $\nchi_2$ are characteristic
functions of a partition, so that 
$\nchi_1 H(x,Du_1)+\nchi_2H(x,Du_2)
=H(x,D(u_1\lor u_2))$ by the second claim.
\end{proof}
\begin{corollary}
  \label{cor:truncation}
  Let $u$ be a subsolution of \eqref{HJ_alpha}, $\alpha,u\in
  L^1_{\loc}(Q)$.
  Then for every choice of $\sigma_-<\sigma_+$ in $\R$ 
  the function 
  $u_\sigma:= \sigma_- \vee u \wedge \sigma_+,$ 
  is a weak subsolution to \eqref{HJ_alpha} with respect to the right
  hand side $\alpha_\sigma$ given by
  \begin{equation}
    \label{eq:31}
    \alpha_\sigma:=
    \begin{cases}
      \alpha&\text{if }\sigma_-\le u\le \sigma_+,\\
      \gamma^+_H&\text{if }u<\sigma_-\text{ or }u>\sigma_+
    \end{cases}
  \end{equation}
\end{corollary}
\begin{proof}
  It is sufficient to observe that every constant function is a
  subsolution to \eqref{HJ_alpha} with right hand side $\gamma_H^+$.
\end{proof}

The next lemma shows that
partial time integration of functions   $\beta\in\cX^q(Q)$
naturally yields
functions $B$ satisfying the assumptions 
used in Lemma \ref{le:increasing-main} and \ref{le:traces}.
\begin{lemma}\label{l:APAC}
  If \GGG $\beta \in \cX^q(Q)$ then $\beta$ satisfies 
  \eqref{eq:237strong} and \eqref{eq:238} holds.
  Moreover, the map $t\mapsto B(t,\cdot)$ is $1/p$-H\"older 
  continuous with values in $\cX^q(\R^d)$.
\end{lemma}
\begin{proof}
\GGG
Let us decompose $\beta=\kappa \beta_1+\beta_2$ with
$\beta_1\in L^\infty(Q)$ and $\beta_2\in L^q(Q)$.
Since $\kappa$ does not depend on $t$, 
by Fubini's theorem we can find a Borel representative $\tilde\beta_i$
such that the maps $t \to \beta_i(t,x)$ belong  to $L^q(I
)$ for every $x \in \R^d$ and 
\begin{equation}
  \sup_{(t,x)\in Q}|\tilde\beta_1(t,x)|<+\infty,\quad
  \int_{\R^d} \left(\int_I |\tilde \beta_2(t,x)|^q \,\d t \right)\,\d x < +\infty.
\end{equation}
The map $B$ is then well defined, it belongs to $L^q(I\times F)$
for every $F$ with finite Lebesgue measure, and it is easy to check that $\partial_t B = \beta $.

Finally we have
\begin{displaymath}
  B(t,x)-B(s,x)=\weight(x)B_1(s,t,x)+B_2(s,t,x),\quad
  B_1(s,t,x):=\int_s^t \tilde\beta_1(r,x)\,\d r,\
  B_2(s,t,x):=\int_s^t\tilde\beta_2(r,x)\,\d r
\end{displaymath}
with
$\|B_1(s,t,\cdot)\|_{L^\infty(\R^d)}\le (t-s)\|\beta_1\|_{L^\infty(Q)}$ and
$\|B_2(s,t,\cdot)\|_{L^q(\R^d)}\le (t-s)^{1/p} \|\beta_2\|_{L^q(Q)}$,
which shows that 
\begin{equation}
  \label{eq:72}
  \|B(t,\cdot)-B(s,\cdot)\|_{\cX^q(\R^d)}\le (t-s)^{1/p} \|\beta\|_{\cX^q(Q)}.
\end{equation}
\end{proof}
Whenever 
\GGG $\alpha\in \cX^q(Q)$, 
\nc
functions satisfying \eqref{HJ_alpha} enjoy the properties analyzed in the previous section.

\begin{proposition}[Regularity and precise representatives of weak subsolutions]
\label{hj2_piacq}
Let $u$ be a weak subsolution to \eqref{HJ_alpha} with $\alpha \in \cX^q(Q)$ and let
$\beta:=\alpha+\gamma_H^-$,
$B$ as \eqref{eq:238strong}. Then $u,\beta$ satisfy the assumptions of Lemma
\ref{le:increasing-main}
and Lemma \ref{le:traces} (with $\Omega:=\R^d$), 
in particular $u$ admits a \GGG precise representative 
$\sfu$ such that $t\mapsto \sfu(t,x)$ is right continuous for every $x\in
\R^d$, 
$u=\sfu$
$\lambda$-a.e.~in $Q$ and \eqref{eq:222} holds.
Moreover, $u$ has traces $u^+_0\in
L^0(\R^d;\R\cup\{-\infty\})$ and $u^-_1\in
L^0(\R^d;\R\cup\{+\infty\})$ at $t=0$ and $t=1$ respectively, 
which are the pointwise limit of $\sfu(t,\cdot)$ as $t\down0$ and $t\up1$ in
$L^0(\R^d;\bar \R)$ and satisfy the properties of Lemma \ref{le:traces}.
\end{proposition}
\begin{proof}
  Defining $\beta:= \alpha+\gamma^-_H$ 
  we immediately see that $-\partial_t u\le \beta$ in $\DD'(Q)$  and we
  can apply 
  Lemma \ref{le:increasing-main} and \ref{le:traces},
  see also Remark \ref{rem:extension-to-D'}.
\end{proof}
Motivated by the previous result we introduce the set 
$\hj_q(Q,H)$ of pairs $(u,\alpha)$ solving \eqref{HJ_alpha}, where
the right-hand side belongs to $\cX^q(Q)$.
\begin{definition}[The convex set $\hj_q(Q,H)$]
\label{def:subsolutions}
We will denote by $\hj_q(Q,H)$ the collection of pairs 
$(u,\alpha)\in 
L^1_{\rm loc}(Q)\times
\cX^q(Q)$ solving \eqref{HJ_alpha}. 
We will always use the symbol $\sfu$ to denote a precise
representative of $u$ according to Proposition \ref{hj2_piacq}.
\end{definition}    
Since the Hamiltonian $H$ is convex with respect to its second
variable, it is easy to check that the set $\hj_q(Q,H)$ is a convex
subset of $L^1_{\rm loc}(Q)\times \cX^q(Q)$.
\subsection{A priori estimates and stability for weak subsolutions}
\label{subsec:stability}
We derive now some regularity properties and a priori estimates for
weak subsolutions of Hamilton Jacobi equations, by studying their
duality with suitable classes of solutions to the continuity
equation, and in particular with Wasserstein geodesics. 
We will always denote by 
$\sfu$ a precise representative associated to a pair
$(u,\alpha)\in \hj_q(Q,H)$ and 
by $\mu$ the precise representative associated to a pair $(m,\vv)\in
\CE 2pQ$. We start with an estimate in the case that the measure $m$ is bounded and with compact support.

\begin{proposition}\label{p:crucial_estimate}
Let $(u,\alpha)\in \hj_q(Q,H)$ 
and let $(m,\vv)\in \CE2\infty Q$. We also suppose that 
there exists a compact set $K\subset \R^d$ such that 
$\supp(\mu_t) \subset K$ for every $t\in [0,1]$.

Setting $\beta:=\alpha+\gamma^-_H$, for every $0 < s< t<1$ we have
\newcommand{\pmn}{{}}
\begin{equation}\label{basic}
  \int_{\R^d} \sfu\pmn_t \,\d\mu_t -
  \int_{\R^d}\sfu\pmn_s\, \d\mu_s \ge
-\frac{c_H}2 \int_s^t \int_{\R^d} |\vv(r,x)|^2 \,m(r,x)\,\d x \,\d r -
\int_s^t \int_{\R^d} \beta(r,x)m(r,x)\, \d x \,\d r.
\end{equation}
\end{proposition}

\begin{proof}
\GGG 
We introduce a pair of convolution kernels $h,k$ as in
\eqref{eq:51}, \eqref{eq:50}, \eqref{eq:51bis}, and we keep the notation of Lemma \ref{le:convolution}, setting
$u_{\tau,\eps} = u \star \eta_{\tau,\varepsilon}$ and choosing $\tau < 1-t$.

 Since by Remark \ref{r:Hto2} $u$ is a weak subsolution to 
\eqref{HJ_alpha_2} with $\beta:=\alpha+\gamma^-_H$, 
then $u_{\tau,\eps}$ is a classical subsolution to 
\begin{equation}
-\partial_t u_{\tau,\eps} +\frac1{2 c_H} |Du_{\tau,\eps}|^2 \leq \beta_{\tau,\eps},
\end{equation}
where $\beta_{\tau,\eps} := \beta \star \eta_{\tau,\eps}$.
Whence
\begin{align}
\label{eq:55}
\int_{\R^d} u_{\tau,\eps} (t,x) \,\d \mu_t(x) 
&- \int_{\R^d} u_{\tau,\eps}(s,x) \,\d \mu_s(x) = 
\int_s^t \int_{\R^d} \left( \partial_t u_{\tau,\eps} +
  Du_{\tau,\eps} \cdot \vv \right) \,\d\mu_r \,\d r \\
&\notag\geq  \int_s^t \left( \int_{\R^d} \left(\frac
    1{2c_H}|Du_{\tau,\eps}|^2  + 
    Du_{\tau,\eps} \vv -  \beta_{\tau,\eps} \right)\, \d\mu_r \right)\,\d r \\
&
\label{eq:55bis}\geq  
  - \frac{c_H}2 \int_s^t \int_{\R^d} |\vv|^2 m\,\d x \,\d r 
  -\int_s^t  \int_{\R^d}
  \beta_{\tau,\eps} 
  m\,\d x\,\d r.
\end{align}
We first pass to the limit as $\eps\down0$ and then as $\tau\down0$.
Standard results for convolutions (see also Lemma
\ref{le:convolution})
imply that 
$\beta_{\tau,\eps}\to \beta$ in $L^1_{\loc}(Q)$, so that we can easily
pass to the limit in \eqref{eq:55bis}.

Concerning the first term of \eqref{eq:55}, thanks to the compact
support and the bounded density of $\mu_t$ and $\mu_s$, 
we can first use \eqref{eq:64} and
then \eqref{eq:231} of Lemma \ref{le:increasing-main}
to pass to the limit, 
and deduce \eqref{basic} thanks to the choice of the precise right
continuous representative.
\end{proof}

\begin{corollary}\label{c:stima_wasserstein}
\GGG
Let us suppose that $(u,\alpha)\in \hj_q(Q,H)$.
For every pair of probability measures $\mu', \mu'' \in \cP(\R^d)$
with compact support and $L^\infty$-densities
$m',m''$, and for every $0<s<t<1$ we have 
\begin{equation}\label{stimawas}
  \begin{aligned}
    \int_{\R^d} \sfu_t \,\d\mu'' - \int_{\R^d}\sfu_s \,\d\mu' \ge
    &- \frac{c_H}{2(t-s)}W^2_2(\mu',\mu'') 
    \\&-
    (t-s)^{1/p}\|\beta\|_{\cX^q(Q)} \max \left(\|m'\|_{L^p\cap
        L^1_\weight(\R^d)}, \| m'' \|_{L^p\cap L^1_\weight(\R^d)}
    \right),
  \end{aligned}
\end{equation}
where $ \beta:=\alpha+\gamma^-_H.$
\end{corollary}
\begin{proof}
\GGG
Let $(\mu_r)_{r\in [s,t]}$ be 
the (rescaled) Wasserstein geodesic satisfying $\mu_s:=\mu'$ and
$\mu_t:=\mu''$. Then $\mu$ satisfies the continuity equation with
minimal velocity field $\vv$ satisfying 
$$
\int_s^t \int_{\R^d} |\vv|^2 \,\d\mu_r \d r  
=\frac{1}{(t-s)}W^2_2(\mu',\mu'').
$$
By writing the decomposition 
$\beta=\weight \beta_0+\beta_1$ with $\beta_0\in L^\infty(Q)$ and
$\beta_1\in L^q(Q)$ we get
\begin{align*}
\int_{(s,t)\times \R^d} \beta \,\d\tilde\mu&=
\int_{(s,t)\times \R^d} \beta_0\weight \,\d\tilde\mu+
\int_{(s,t)\times \R^d} \beta_1 \,\d\tilde\mu\\&\le
(t-s)\|\beta_0\|_{L^\infty(Q)}\max_{r\in[s,t]}\int_{\R^d}\weight\,\d\mu_r+
(t-s)^{1/p}\|\beta_1\|_{L^q(Q)}\max_{r\in[s,t]}\|m_r\|_{L^p(\R^d)}
\\&\le 
(t-s)^{1/p}\Big(\|\beta_0\|_{L^\infty(Q)}+\|\beta_1\|_{L^q(Q)}\Big)
\max\Big(\|m'\|_{L^p\cap
        L^1_\weight(\R^d)}, \| m'' \|_{L^p\cap L^1_\weight(\R^d)}
    \Big).
\end{align*}
Therefore, the conclusion follows from Proposition \ref{p:crucial_estimate}.
\nc
\end{proof}

\begin{theorem}[Interior regularity and traces of precise representatives]\label{thm:precise_repr}
\GGG 
If $(u,\alpha)\in \hj_q(Q,H)$ 
\nc
then the precise representatives $\sfu_t = \sfu(t, \cdot)$ belong
to $\cX^q(\R^d)$ for every $t \in (0,1)$.
For every $\mu=m\LL^d  \in \cP_{2,p}^r(\R^d)$ 
 the map $t \mapsto \int_{\R^d} \sfu(t,x) \,\d\mu$ (respectively $t\mapsto
 \int_{\R^d} \sfu^-(t,x) \,\d\mu$) is right (respectively left)
 continuous in $I$,
$u^-(1,\cdot)$, $u^+(0,\cdot)$ are semiintegrable with respect to
$\mu$   and 
\begin{equation}\label{exlimits}
\begin{split}
\exists \, \lim_{t \uparrow 1} \int_{\R^d} \sfu(t,x) \,\d\mu &= \int_{\R^d} u^-(1,x) \,\d\mu \quad \in (-\infty, + \infty]; \\
\exists \, \lim_{t \downarrow 0} \int_{\R^d} \sfu(t,x) \,\d\mu &= \int_{\R^d} u^+(0,x) \,\d\mu \quad \in [-\infty, + \infty).
\end{split}
\end{equation}
Moreover, for every $0\le a<r<b\le 1$ and every pair of nonnegative measures
$\mu_{a,b}=m_{a,b}\LL^d\in \cP_{2,p}^r(\R^d)$ 
such that $\sfu_a\in L^1(\mu_a),\ \sfu_b\in L^1(\mu_b)$ 
with $I_a=\int_{\R^d}\sfu_a\,\d\mu_a$, $I_b:=\int_{\R^d}\sfu_b\,\d\mu_b$
we have 
\begin{equation}
  \label{eq:75}
  \|\sfu_r\|_{\cX^q(\R^d)}\le (I_a)_- +(I_b)_+ +\Big(
  \frac{c_H(b-a)}{(b-r)(r-a)}+2\|\beta\|_{\cX^q(Q)}\Big)\max\big(\|m_a\|_{L^p\cap
    L^1_\weight(\R^d)}
  ,\|m_b\|_{L^p\cap L^1_\weight(\R^d)}\big).
\end{equation} 
\end{theorem}
\begin{proof}
Let us 
first choose $0<a<b<1$, $r\in (a,b)$, 
$\mu_{a,b}=m_{a,b}\LL^d\in \cP_2(\R^d)$,
$m_{a,b}\in L^p\cap L^1_\weight(\R^d)$ such that
$\sfu_a\in L^1(\mu_a),\ \sfu_b\in L^1(\mu_b)$; since $\sfu_a,\sfu_b$
belong
to $L^1_{\loc}(\R^d)$ the collection of such measures is surely not empty.

For $R,K>0$ sufficiently big we consider the probability measures
$\mu_{b,R,K}=m_{b,R,K}\LL^d\in \cP^r_2(\R^d)$
\begin{equation}
  \label{eq:73}
  m_{b,R,K}:=c_{R,K}(m_b\land K)\nchi_{B_R(0)}\LL^d,\quad
  c_{R,K}:=\Big(\int_{B_R(0)}(m_b\land K)\,\d x\Big)^{-1}\in [1,+\infty),
\end{equation}
where the truncation is needed in order to apply Corollary \ref{c:stima_wasserstein}.
We choose a bounded measurable
$\varphi:\R^d\to [0,\infty)$ with compact support such that $\|\varphi\|_{L^1(\R^d)}>0$;
we can apply \eqref{stimawas} with the choices $s:=r$, $t:=b$, 
$m'= \frac{\varphi(x) \LL^d}{\|\varphi\|_{L^1(\R^d)}}, $
$m''= m_{b,R,K}$. 
The triangle inequality for $W_2$ and the fact that
$W_2^2(\mu,\delta_0)=\int_{\R^d}|x|^2\,\d\mu$ yield
\begin{displaymath}
  W_2^2(\mu',\mu'')\le 2\Big(\int_{\R^d}|x|^2\,\d\mu'(x)+c_{R,K}\int_{\R^d}|x|^2\,\d\mu_b\Big);
\end{displaymath}
on the other hand
\begin{displaymath}
   \|m'\|_{L^p\cap L^1_{\weight}(\R^d)}\le 
   \frac{1}{\|\varphi\|_{L^1(\R^d)}}\|\varphi\|_{L^p\cap L^1_\weight(\R^d)}
     ,\quad
     \|m''\|_{L^p\cap L^1_{\weight}(\R^d)}\le c_{R,K}\|m_b\|_{L^p\cap L^1_\weight(\R^d)}
\end{displaymath}
so that \eqref{stimawas} yields
\begin{equation}\label{prec1}
\begin{split}
\int_{\R^d} \sfu(r,x)\varphi(x) \,\d x 
&\leq \|\varphi \|_{L^1(\R^d)} \int_{\R^d} \sfu_b\,\d\mu_{b,R,K}+ \frac{c_H}{(b
  - r)} c_{R,K}\|\varphi\|_{L^1_\weight(\R^d)}
  \|m_b\|_{L^1_\weight(\R^d)}
  \\
& \quad + c_{R,K}\|\beta\|_{\cX^q(Q)}  \| \varphi \|_{L^p\cap
    L^1_\weight(\R^d)} \| m_b\|_{L^p\cap
    L^1_\weight(\R^d)}.
\end{split}
\end{equation}
We can now pass to the limit as $R,K\up\infty$ observing that
$c_{R,K}\to 1$, obtaining
\begin{equation}
  \label{eq:74}
  \int_{\R^d} \sfu(r,x)\varphi(x) \,\d x \le {C_b}
  \, \|\varphi\|_{L^p\cap L^1_\weight(\R^d)},\quad
  C_b:=(I_b)_++\Big(\frac{c_H}{(b
  - r)}+\|\beta\|_{\cX^q(Q)}\Big) \| m_b\|_{L^p\cap
    L^1_\weight(\R^d)}.
\end{equation}
Replacing $\varphi$ with $\varphi\nchi_{\{u(r,\cdot)>0\}}$ we see that 
the same estimate holds for the positive part of $u$; \eqref{eq:20}
then yields
\begin{equation}
  \label{eq:68}
  \|\sfu(r,\cdot)\lor 0\|_{\cX^q(\R^d)}\le C_b.
\end{equation}
A similar argument, replacing now $b$ with $a$ 
yields
\begin{equation}\label{prec2}
\begin{split}
  \int_{\R^d} \sfu(r,x)\varphi(x) \,\d x \geq  -C_a\| \varphi \|_{L^p\cap
    L^1_\weight(\R^d)},\quad
    C_a:=(I_a)_-+\Big(\frac{c_H}{(r-a)}+\|\beta\|_{\cX^q(Q)}\Big) \| m_a\|_{L^p\cap
    L^1_\weight(\R^d)}.
\end{split}
\end{equation}
We deduce that $\|\sfu(r,\cdot)\land 0\|_{\cX^q(Q)}\le C_a$
and 
\begin{equation}
  \label{eq:70}
  \|\sfu(r,\cdot)\|_{\cX^q(\R^d)}\le C_a+C_b,
\end{equation}
which yields \eqref{eq:75}.

Since $\cX^q(\R^d)$ is the dual space of $L^p\cap L^1_\weight(\R^d)$
(see Lemma \ref{le:Xq}),
the right continuity of the curves $t\mapsto
\int_{\R^d}\sfu_t\,\d\mu$ for every $\mu=m\LL^d\in \cP_2(\R^d)$ with
$m\in L^p(\R^d)$ is equivalent to the right continuity
of $t\mapsto \sfu_t$ with respect to the weak$^*$ topology of $\cX^q(\R^d)$.
Thanks to the uniform bound \eqref{eq:75} and the strong density of 
bounded functions with compact support in $L^p\cap L^1_\weight(\R^d)$,
this property follows by the right continuity in $L^0(\R^d)$. 

Finally, \eqref{exlimits} follows by decomposing $\sfu$ as $\sfw-B$ as in
Proposition \ref{hj2_piacq}; by Lemma \ref{l:APAC}
it is sufficient to study the behaviour of $w$. 
By monotonicity, since the right trace 
$w^-(1,\cdot)\ge \sfw(r,\cdot)$ for every $r<1$
and 
$\sfw(r,\cdot)\in \cX^q(\R^d)$, we immediately get that the negative part
of $w^-(1,\cdot)$ is $\mu$-integrable. 
The first property of \eqref{exlimits} then follows by Beppo Levi's Monotone
Convergence Theorem. A similar argument works for $u^+(0,\cdot)$.

Eventually passing to the limit in \eqref{eq:75} as $b\up 1$ and
$a\down0$, we extend its validity to the case $a=0$ and $b=1$.
\end{proof}

We conclude this section with  a stability argument for the 
weak subsolutions of \eqref{HJ_alpha}.

\begin{theorem}[Stability and semicontinuity]\label{t:stability}
Suppose that the sequence $(u_n,\alpha_n)\in \hj_q(Q,H)$, $n\in \N$, 
satisfies the \GGG local uniform bound
\begin{equation}
  \label{eq:76}
  \sup_{n\in \N}\sup_{a\le t\le
    b}\|\sfu_n(t,\cdot)\|_{\cX^q(\R^d)}<\infty\quad
  \text{for every }0<a<b<1,
\end{equation}
\nc
and that  
$\alpha_n \weakto^*
\alpha$  weakly$^*$  
in $\cX^q(Q)$. 
Then there exist  a subsequence $k \mapsto n(k)$ and a limit function $u\in L^1_{\rm loc}(Q)$ such that
\begin{enumerate}[(a)]
\item $(u,\alpha)\in \hj_q(Q,H)$ (thus in particular $Du\in L^2_{\loc}(Q;\R^d)$);
\item 
\begin{align}
\label{eq:77}
  \sfu_{n(k)}(t,\cdot) \weakto^* \tilde u\quad  &\text{in }
\cX^q(\R^d) && \text{with }\sfu^-(t,\cdot)\le \tilde \sfu_t\le \sfu(t,\cdot)
\text{ for \GGG every }\, \, t \in \GGG
                                               (0,1),\\
\label{eq:36}
u_{n(k)} \weakto u \quad&\text{in } L^q(K) &&\hbox{for all compact sets $K\subset Q$}
,\\
\label{eq:58}
Du_{n(k)} \weakto Du \quad &\text{in } L^2(K;\R^d)
&&\hbox{for all compact sets $K\subset Q$}.
\end{align}
\item 
  \GGG For every $\mu_i=m_i\LL^d$ with $m_i\in L^p\cap
  L^1_\weight(\R^d)$ the quantities 
  \begin{equation}
    \label{eq:67}
    \cA(u_n,\alpha_n):=\int_{\R^d}
  u_n^+(0,\cdot)\,\d\mu_0  - 
  \int_{\R^d} u_n^-(1,\cdot) \,\d\mu_1  - \int_{Q} F^*(x,\alpha_n)\, \d
  x \d t 
  \end{equation}
  and the corresponding one $\cA(u,\alpha)$ defined on the limit
  pair $(u,\alpha)$ are well defined in $[-\infty,+\infty)$ and
  satisfy
  \nc
\begin{equation}\label{eq:semicontinuity}
\begin{split}
\limsup_{n \uparrow + \infty} \cA(u_n,\alpha_n)\le \cA(u,\alpha).
\end{split}
\end{equation}
\end{enumerate}
\end{theorem} 

\begin{proof}

\GGG
(b)
Let $\beta_n:=\alpha_n+\gamma^-_H$, let $B_n$ be correspondly defined as
in Lemma \ref{l:APAC}, and let $\sfw_n:=\sfu_n+B_n$,
so that $\sfw_n$ is increasing w.r.t.~time and right continuous. 

It is easy to see that $B_n$ satisfies \eqref{eq:77}, so that it is
sufficient to prove this convergence property for $\sfw_n$, which 
still obeys to 
\begin{equation}
  \label{eq:76bis}
  \sup_{a\le t\le
    b}\|\sfw_n(t,\cdot)\|_{\cX^q(\R^d)}\le C(a,b)<\infty
  \quad
  \text{for every }0<a<b<1,\ n\in \N.
\end{equation}

We fix an interval $[a,b]\subset
(0,1)$ and we consider 
the set $K:=\{z\in \cX^q(\R^d):\|z\|_{\cX^q(\R^d)}\le C(a,b)\}$
endowed 
with the 
weak$^*$-topology $\sigma$ of $\cX^q(\R^d)$.
$K$ is compact w.r.t.~$\sigma$. We introduce the distance
\begin{equation}
  \label{eq:79}
  \delta(z_1,z_2):=\int_{\R^d}|z_1(x)-z_2(x)|\,\rho(x)\,\d x\quad
  \text{for every }z_1,z_2\in K,
\end{equation}
where $\rho(x)$ is the density of the measure $\varrho$ we used to
define the distance $d$ in $L^0(\R^d)$, see \eqref{eq:78}.
Clearly $\delta\ge d$ and it is also $\sigma$ lower semicontinuous,
since the weak$^*$ topology of $\cX^q(\R^d)$ is clearly stronger 
than the weak topology in $L^1(\R^d,\varrho)$.

For every $a\le s\le t\le b$ we have
\begin{align*}
  \delta(\sfw_n(s),\sfw_n(t)) &
                            \le \int_{\R^d} |\sfw_n(t,x) -  \sfw_n(s,x)| \,\d\varrho(x) 
                            = \int_{\R^d} \Big(\sfw_n(t,x)-  \sfw_n(s,x) \Big)\,\d\varrho(x)
  \\&=
      \int_{\R^d} \sfw_n(t,x) \,\d\varrho(x) - \int_{\R^d} \sfw_n(s,x) d\varrho(x)
\end{align*}
so that the $\delta$-total variation of $\sfw_n$ in the interval $[a,b]$ can be
estimated by 
\begin{equation}
\operatorname{Var}_\delta\left( \sfw_n; {[a,b]} \right) \leq 
\int_{\R^d} \sfw_n(b,x) \,\d\varrho(x) - \int_{\R^d} \sfw_n(a,x)
\,\d\varrho(x) \leq 
C(a,b)\|\rho\|_{L^p\cap L^1_\weight(\R^d)},\quad
\varrho=\rho\LL^d.
\end{equation}
Since the sequence of functions
$\sfw_n$ takes value in the $\sigma$-compact set $K$ 
and their $\delta$-total variation is uniformly bounded,
a refined version of Helly's Theorem 
see e.g. \citep[Prop.~3.3.1]{ambrosio2008gradient},   
shows that $\sfw_n$  weakly$^*$ converges in $\cX^q(\R^d)$, for every $t
\in [a,b]$, 
to a function $\tilde \sfw$. A standard diagonal argument yields the same
convergence for every $t\in I$.
Since $t\mapsto \tilde \sfw_t$ is increasing, we can set
$\sfw_t:=\lim_{s\downarrow t}\tilde \sfw_s$ thus obtaining \eqref{eq:77}.
Notice that if $J$ denotes the (at most) countable set
$J:=\{t\in (0,1):\sfw^-_t\neq \sfw_t\}$, we have $\tilde \sfw_t=\sfw_t$ for every
$t\in (0,1)\setminus J$.

If $\zeta\in L^p(Q)$ has compact support contained in $[a,b]\times F$,
where $F$ is a compact subset of $\R^d$, we get
\begin{align*}
  \lim_{n\to\infty}\int_{Q} \sfw_n \zeta\,\d t\,\d x
  &=
  \lim_{n\to\infty}\int_a^b \Big(\int_F \sfw_n(t,x)\zeta(t,x)\,\d
  x\Big)\,\d t  
  =
    \int_a^b\Big(\lim_{n\to\infty}
\int_F \sfw_n(t,x)\zeta(t,x)\,\d
  x\Big)\,\d t  
  \\&=\int_a^b \Big(\int_F \sfw(t,x)\zeta(t,x)\,\d
  x\Big)\,\d t  
\end{align*}
where we have applied 
the above convergence result and the Lebesgue Dominated Convergence 
Theorem; notice that for a.e.~$t\in [a,b]$
\begin{displaymath}
  \Big|\int_F \sfw_n(t,x)\zeta(t,x)\,\d
  x\Big|\le C(a,b)g(t),\quad
  g(t):=\|\zeta(t,\cdot)\|_{L^p\cap L^1_\weight(\R^d)}
  \le \|\zeta(t,\cdot)\|_{L^p(\R^d)}+\|\zeta(t,\cdot)\|_{L^1_\weight(\R^d)}
\end{displaymath}
so that $\int_a^b g(t)\,\d t\le \|\zeta\|_{L^p(Q)}+\|\zeta\|_{L^1_\weight(Q)}$.

Concerning \eqref{eq:58}, it is sufficient to show that for every
compact set $K\subset Q$ there exists a constant $C(K)$
such that
\begin{equation}
  \label{eq:60}
  \int_{K}|Du_n|^2\,\d x\,\d t\le C(K)\quad
  \text{for every }n\in \N.
\end{equation}
We select a nonnegative function $\xi \in C_c^\infty(Q)$ 
such that $\xi \equiv 1$ on $K$; in particular the support of $\xi$ 
will be contained in $[a,b]\times \R^d$ for some interval
$[a,b]\subset (0,1)$. 
The distributional
formulation \eqref{wHJ_sub} of \eqref{HJ_alpha} yields
\begin{align*}
\int_{K}  |Du_n|^2 \,\d x\,\d t&\leq
\int_Q  |Du_n|^2 \xi\,\d x\,\d t\leq
 \int_Q \big(\alpha_n \xi+u_n\partial_t \xi\big)\,\d x\,\d t
 \\& \le \big(\sup_{n\in \N}\|\alpha_n\|_{\cX^q}\big)\|\xi\|_{L^p\cap L^1_\weight(Q)}+
    \big(\sup_{n\in \N}\sup_{t\in [a,b]}
     \|u_n(t,\cdot)\|_{\cX^q}\big) \|\partial_t \xi\|_{L^p\cap L^1_\weight(Q)},
\end{align*}
which yields \eqref{eq:60}.

\noindent
(a) 
\GGG For every nonnegative test function $\xi\in C^\infty_c(Q)$
we have to pass to the limit in the inequality \eqref{wHJ_sub} written
for $u_n,\alpha_n$. By the previous claim, it is sufficient to prove
that
\begin{displaymath}
  \liminf_{n\to\infty}\int_Q H(x,Du_n)\xi\,\d x\, \d t\ge 
  \int_Q H(x,Du)\xi\,\d x.
\end{displaymath}
This property follows by the weak $L^2$ lower semicontinuity of the integral
functional associated to $H$ (and weighted by $\xi$),
relying on the structural properties 
(convexity and lower bound) of
$H$ stated in Assumptions \ref{h.1}(H3), see e.g.
\cite[Theorem 6.54]{Fonseca-Leoni07}.

\noindent
(c) \eqref{eq:67} is well defined (possibly taking the value
$-\infty$)
thanks to 
Theorem \ref{thm:precise_repr} and the fact that all the three terms
belong to $[-\infty,+\infty)$.

Since $F^*$ is nonnegative, measurable, convex and lower
semicontinuous w.r.t.~its second argument, 
the last integral functional
$\int_Q F^*(x,\alpha)\,\d x\,\d t$ is lower semicontinuous with
respect to weak $L^1$ convergence on compact subsets, so that 
\begin{displaymath}
  \limsup_{n\to\infty}-\int_QF^*(x,\alpha_n)\,\d x\,\d t\le -\int_QF^*(x,\alpha)\,\d x\,\d t.
\end{displaymath}
Let us now consider the behaviour of the first integral defining $\cA(u_n,\alpha_n)$.
\nc In order to prove inequality \eqref{eq:semicontinuity} 
\GGG we observe that for $\eps>0$ 
$$u_n^+(0,x)=w_n^+(0,x)\le
\sfw_n(\eps,x)=\sfu_n(\eps,x)+B_n(\eps,x)$$
and, by \eqref{eq:72},
\begin{displaymath}
  \Big|\int_{\R^d}B_n(\eps,x)\,\d\mu_0(x)\Big|\le
  (t-s)\|\beta_n\|_{\cX^q(Q)}
  \|m_0\|_{L^p\cap L^1_\weight(\R^d)},
\end{displaymath}
so that \nc
\begin{equation}\label{eq:u_n:o_piccolo}
\int_{\R^d} u_n^+(0,x) \,\d\mu_0 \leq \int_{\R^d} \sfu_n(\varepsilon,
x) \,\d\mu_0 
+ o(1), \qquad \text{ as } \varepsilon \downarrow 0.
\end{equation}
where $o(1)$ is uniform with respect to $n$.
Sending $n \to \infty$, thanks to the weak* convergence of
$\sfu_n(\varepsilon, \cdot)$ in
$\cX^q(\R^d)$,  we get
\begin{equation}
\limsup_{n \to\infty} \int_{\R^d} u_n^+(0,x) \,\d\mu_0 \leq \limsup_{n \to\infty} \int_{\R^d} \sfu_n(\varepsilon, x) \,\d\mu_0 + o(1) 
= \int_{\R^d} \sfu(\varepsilon, x) \,\d\mu_0 + o(1).
\end{equation}
Thanks to Theorem \ref{thm:precise_repr}, passing to the limit as $\varepsilon \downarrow 0$ we get the first part of \eqref{eq:semicontinuity}. A similar argument yields
\begin{equation}
\liminf_{n \to\infty} \int_{\R^d} u_n^-(1,x) \,\d\mu_1 \geq \int_{\R^d} u^-(1,x) \,\d\mu_1\,.
\end{equation}
\end{proof}

\subsection{A general duality-transport result}
\label{subsec:duality-transport}

The following result is fundamental for the development of a variational approach to MFPPs. 
We derive a \GGG transport and a \nc 
duality relation between weak subsolutions to 
\eqref{HJ_alpha} with $\alpha\in \cX^q(Q)$ and $L^p$-solutions to the
continuity equation.

\GGG In the next Theorem we will use an arbitrary test function $\zeta$ and
its primitive $Z$ in $\ZZ_c$, see \eqref{eq:239}.

\begin{theorem}\label{l:duality_hj_cont}
  Let $(u,\alpha)\in \hj_q(Q,H)$ be a weak subsolution to 
  \eqref{HJ_alpha} according to Definition \ref{def:subsolutions} 
and let $(m,\vv)\in \CE 2p Q$ be a distributional solution to
the continuity equation according to Definition \ref{def:CE}.
Then  we have: 
\begin{enumerate}[\rm (1)]
\item for every $0<s<t<1$
\begin{equation}
  \label{eq:71}
  \int_s^t \int_{\R^d} |Du(r,x)|^2 m(r,x)\,\d x\,\d r < +\infty;
\end{equation}
\item 
\GGG for every pair $(\zeta,Z)\in \ZZ_c$
  \begin{equation}
    \label{eq:104}
    \partial_t\big(Z(u)m\big)+
    \nabla\cdot\big(Z(u)m\,
    \vv\big)+\zeta(u)\Big(\alpha-H(x,Du)-Du\cdot\vv\Big)m\ge0
    \quad \text{in }\DD'(Q),
  \end{equation}
  also in duality with nonnegative functions $\varphi\in
  C^1_c(\overline Q)$ if $\mu_0,\mu_1\ll\LL^d$:
  \begin{equation}
    \label{eq:109a}
    \begin{aligned}
      -\int_Q Z(u)m&\Big(\partial_t\varphi+D\varphi\cdot\vv\Big)\,\d
      x\d t + \int_Q \zeta(u)\Big(\alpha-H(x,Du)-Du\cdot\vv\Big)
      m\,\varphi\,\d x\d t
      \\&\ge\int_{\R^d}Z(u^+(0,x))\varphi(0,x)\,\d\mu_0(x)-
      \int_{\R^d}Z(u^-(1,x))\varphi(1,x)\,\d\mu_1(x),
    \end{aligned}
\end{equation}
\item 
\GGG for 
every $s,t\in D_p[\mu]$ with $0<s<t<1$ \nc it holds
\begin{align}
  \label{duality_formula}
  \int_{\R^d} \sfu_s \,\d\mu_s-
  \int_{\R^d} \sfu_t 
  \,\d\mu_t 
  &\le \int_s^t\int_{\R^d}
  \Big(\alpha-  H(x,Du) - Du \cdot \vv \Big)m \,\d x\,\d r.
\end{align}
\end{enumerate}
\GGG
Moreover, 
\eqref{eq:71}
and \eqref{duality_formula} 
also hold for $s=0$ 
(respectively, $t=1$) provided 
$0\in D_p[\mu]$ and $\int_{\R^d} u_0^{+} \,\d\mu_0>-\infty$ 
(respectively, $1\in D_p[\mu]$, $\int_{\R^d} u_1^{-}
\,\d\mu_1<+\infty$).
\end{theorem}

\begin{proof}

{\it Step 1.} Let us first suppose that $u$ is bounded.
\GGG By Theorem \ref{thm:precise_repr} the traces $u^-_1$ and $u^+_0$
are well defined and belong to $L^\infty(\R^d)$.

Operating a regularization by convolution according to 
the notation of Section \ref{subsec:kernels}, \nc
we know that $u_{\tau,\eps}$ is a classical subsolution (with
uniformly bounded derivatives) to 
\begin{equation}
-\partial_t u_{\tau,\eps} + \tilde H_{\tau,\eps} \leq
\alpha_{\tau,\eps}\quad\text{in }[0,1-\tau]\times \R^d,
\end{equation}
where $\tilde H_{\tau,\eps}(x):=  \big(H(x, Du)\big)
\ast\eta_{\tau,\eps}$. 
\GGG 
Similarly, by multiplying the previous
inequality by $\zeta(u_{\tau,\eps})$, we also obtain
\begin{equation}
  \partial_t Z(u_{\tau,\eps})+
  \zeta(u_{\tau,\eps})\big(\alpha_{\tau,\eps}- \tilde
  H_{\tau,\eps}\big)\ge 0\quad\text{in }[0,1-\tau]\times \R^d.
\end{equation}

\GGG
By selecting a nonnegative function $\varphi\in
C^\infty_c([0,1-\tau)\times \R^d)$ 
and testing the continuity equation for $(m,\vv)$ with $\varphi\cdot
Z(u_{\eps,\tau})$, we get for $\theta_{\tau,\eps}:=
\alpha_{\tau,\eps}- \tilde
  H_{\tau,\eps}-Du_{\tau,\eps}\cdot\vv$
\begin{align}
  \int_{\R^d} &Z(u_{\eps,\tau}(0,x))\varphi(0,x)m_0(x)\,\d x=
  -\int_Q \Big(\partial_t(Z(u_{\eps,\tau})\varphi)+D(Z(u_{\eps,\tau})\varphi)\cdot\vv\Big)m\,\d
    x\d t\notag
     \\&=
         -\int_Q Z(u_{\eps,\tau})m\Big(\partial_t\varphi+D\varphi\cdot\vv\Big)\,\d
  x\d t
   -
    \int_Q
    \Big(\partial_tZ(u_{\eps,\tau})+DZ(u_{\eps,\tau})\cdot\vv\Big)\varphi
         m\,\d
    x\d t\notag
         \\&\le -\int_Q Z(u_{\eps,\tau})m\Big(\partial_t\varphi+D\varphi\cdot\vv\Big)\,\d
  x\d t
   +
    \int_Q
    \zeta(u_{\eps,\tau})\theta_{\tau,\eps}
    m\,\varphi\,\d
    x\d t\label{eq:105}
\end{align}
Similarly, using  the distributional formulation of the continuity equation we
get
\GGG for every $0\le s<t\le 1-\tau$
\begin{align}
0&=\int_{\R^d}  u_{\tau,\eps}(s,\cdot) \,\d\mu_s - \int_{\R^d} u_{\tau,\eps}(t,\cdot) \,\d\mu_t 
+\int_s^t \int_{\R^d} \left( \partial_t u_{\tau,\eps} +
  Du_{\tau,\eps} \cdot \vv \right) \,\d\mu_r \,\d r 
\notag\\\label{eq:duality_epsilon_tau}
&\ge 
\int_{\R^d}  u_{\tau,\eps}(s,\cdot) \,\d\mu_s - \int_{\R^d}
u_{\tau,\eps}(t,\cdot) \,\d\mu_t 
+\int_s^t \int_{\R^d} \left( \tilde H_{\tau,\eps} +
  Du_{\tau,\eps} \cdot \vv -\alpha_{\tau,\eps}\right)m \,d x\,\d r
\end{align}
Due to the growth condition from below which is assumed for $H$  (see \eqref{H_least_quadratic2}), we have
$$
\GGG
 \tilde H_{\tau,\eps}\ge
 \big(\frac1{2c_H}|Du|^2-\gamma^-_H\Big)\ast\eta_{\tau,\eps} 
 \geq \frac1{2c_H}|Du_{\tau,\eps}|^2 - \gamma^-_{H,\eps},\quad
 \gamma^-_{H,\eps}(x) :=\big(\gamma^-_H\ast\eta_{\tau,\eps}\big)(x)=
 \big(\gamma^-_H\ast k_\eps\big)(x).
$$
Hence we estimate
\begin{equation}\label{dabbasso}
\GGG
\tilde H_{\tau,\eps} +Du_{\tau,\eps} \cdot \vv \geq  \frac1{4c_H}
|Du_{\tau,\eps}|^2 -  \gamma^-_{H,\eps}- c_H|\vv|^2
\ge -\gamma^-_{H,\eps}- c_H|\vv|^2.
\end{equation}
If we choose $s,t\in D_p[\mu]$ we can pass to the limit 
\GGG first as $\eps\down0$ and then as  $\tau \downarrow 0$
in \eqref{eq:duality_epsilon_tau}, by invoking Fatou's Lemma and
the lower bound \eqref{dabbasso} for the integral involving 
$\tilde H_{\tau,\eps} +Du_{\tau,\eps} \cdot \vv$, standard convolution
estimates for 
$\alpha_{\tau,\eps}$ in $\cX^q(Q)$, and
Lemma \ref{le:convolution} 
for the
first two integrals; we thus obtain
\begin{equation}\label{set}
\int_{\R^d} \sfu(s,x) \,\d\mu_s(x) - \int_{\R^d} \sfu(t,x) \,\d\mu_t (x)
\leq \int_s^t \int_{\R^d} \left( \alpha -H(x,Du) - Du  \cdot \vv
  \right)m\,\d x\,\d t
\end{equation}
for every $s,t\in [0,1)\cap D_p[\mu]$, $s<t$.
The very same argument, using the convolution kernel $\hat
h(t):=h(-t)$ and choosing a point $s_1>s$ where $u^-_{s_1}=u_{s_1}$,
yields the inequality \eqref{set} for $t=1$ with $s=s_1$. Adding the
same inequality between $s$ and $s_1$ we conclude the proof of \eqref{duality_formula}.

\eqref{set} and the lower bound $H(x,Du)-Du\cdot\vv\ge \frac
1{4c_H}|Du|^2-\gamma^-_H-
c_H|\vv|^2$ yields
\begin{equation}
  \label{eq:108}
  \frac 1{4 c_H}\int_s^t \int_{\R^d} |Du(r,x)|^2\, \d\mu_r
  \,\d r \le \int_{\R^d}  \sfu(t,\cdot) \,\d\mu_t - \int_{\R^d}
  \sfu(s,\cdot) \,\d\mu_s
      +C
    \end{equation}
for the constant $C=(\|\alpha\|_{\cX^q(Q)}+\|\gamma^-_H\|_{\cX^q(\R^d)})\|m\|_{L^p\cap L^1_\weight(Q)}+
c_H\int_Q |\vv|^2\,\d\tilde\mu$.

Finaly, using the fact that $Z,\zeta$ are bounded and continuous, 
we can pass to the limit in \eqref{eq:105} obtaining
the weak formulation of \eqref{eq:104} supplemented with boundary terms:
\begin{equation}
  \label{eq:109}
  \begin{aligned}
    -\int_Q Z(u)m\Big(\partial_t\varphi+D\varphi\cdot\vv\Big)\,\d
    x\d t &+ \int_Q \zeta(u)\Big(\alpha-H(x,Du)-Du\cdot\vv\Big)
    m\,\varphi\,\d x\d t
    \\&\ge\int_{\R^d}Z(u^+(0,x))\varphi(0,x)\,\d\mu_0(x)-
    \int_{\R^d}Z(u^-(1,x))\varphi(1,x)\,\d\mu_1(x),
  \end{aligned}
\end{equation}
for every nonnegative $\varphi\in C^1_c(\overline Q).$

\noindent
{\it Step 2.} 
Let us now deal with the case of a general subsolution $u$. 
If we define $u_k:= -k \vee u \wedge k$, from Corollary 
\ref{cor:truncation} we know that $u_k$ is a bounded weak subsolution to 
\begin{equation}
-\partial_t u_k + H(x, Du_k) \leq \alpha_k,
\end{equation}
where $\alpha_k \to \alpha$ weakly$^*$ in $\cX^q(Q)$
and pointwise a.e., with the uniform domination $|\alpha_k|\le |\alpha|+\gamma^-_H$.
In particular \eqref{set}, \eqref{eq:108} and \eqref{eq:109}
hold with $u_k,\alpha_k$ in place of $u,\alpha$ respectively.

\eqref{eq:104} and \eqref{eq:109a} 
can be directly obtained by
choosing $k$ sufficiently big, so that $\supp(\zeta)\subset (-k,k)$,
since in this case
\begin{displaymath}
  Z(u)=Z(u_k),\quad \zeta(u)=\zeta(u_k),\quad 
  \zeta(u)\alpha=\zeta(u_k)\alpha_k.
\end{displaymath}
In order to prove \eqref{eq:71} and \eqref{duality_formula}, 
we recall that  $\sfu(t, \cdot)$ belongs to
$\cX^q(\R^d)$ for every $t \in (0,1)$ (with a uniform estimate in any
compact subset $[a,b]$).
\GGG If  we choose $s,t\in D_p[\mu]$ we thus have $\sfu(t,\cdot)\in
L^1_{\mu_t}(\R^d)$, $\sfu(s,\cdot)\in L^1_{\mu_s}(\R^d)$,
so that the right hand side of \eqref{eq:108} is uniformly bounded 
and \eqref{eq:71} follows from Fatou's lemma. 

\eqref{duality_formula} is a consequence of a similar limit
starting from
\begin{equation}\label{setk}
\begin{split}
\int_{\R^d} \sfu_k(s,x)\,\d\mu_s(x) - \int_{\R^d} \sfu_k(t,x)\,\d\mu_t (x)
\leq \int_s^t \int_{\R^d} \left(\alpha_k -H(x,Du_k) - Du_k  \cdot \vv \right)m
\,\dx \,\d r
\end{split} 
\end{equation}
and applying Lebesgue Dominated Convergence Theorem (and Fatou's lemma in the right-hand side if $t=1$ or $s=0$). 
\end{proof}
For our next purposes, we deduce the following consequence  
of inequality \eqref{duality_formula}, which is a natural extension of
\eqref{stimawas}.
The proof follows the same argument of Corollary
\ref{c:stima_wasserstein}, starting from \eqref{duality_formula}.
\begin{corollary}\label{c:duality:u_Lq}
Let $(u,\alpha)\in \hj_q(Q,H)$ and $H$ satisfying \eqref{H_least_quadratic2}. 
Let $\mu_i=m_i\LL^d\in \cP_{2,p}^r(\R^d)$. 
Then we have \GGG
\begin{align*}  
\int_{\R^d} u_0^{+} 
  \,\d\mu_0 
  -\int_{\R^d} u_1^{-} \,\d\mu_1 \le \frac{c_H}2W_2^2(\mu_0,\mu_1)
  +\|\beta\|_{\cX^q(Q)}\max \big(\|m_0\|_{L^p\cap L^1_\weight(\R^d)},
  \|m_1\|_{L^p\cap L^1_\weight(\R^d)}\big), 
\end{align*}
where $ \beta:=\alpha+\gamma^-_H.$
\end{corollary}
\subsection{Contact-defect measures associated to weak subsolutions}
\label{subsec:contact-defect}
\GGG
In this last section we rewrite \eqref{duality_formula}, \eqref{eq:109a}, and 
\eqref{eq:104} in a
more expressive way. 

Let us fix a pair $(u,\alpha)\in \hj_q(Q,H)$ and 
a solution $(m,\vv)\in \CE 2p{Q;\mu_0,\mu_1}$ 
with $\mu_i\ll\LL^d$, $i=0,1$.
First of all, for every $\zeta\in C^0_c(\R)$ with $Z(r):=\int_0^r
\zeta(s)\,\d s$ we consider the
linear functional $T_\zeta\in \DD'(\R^{d+1})$
\begin{equation}
  \label{eq:154}
  \begin{aligned}
    T_\zeta(\varphi):=-\int_Q Z(u)m&\Big(\partial_t\varphi+D\varphi\cdot\vv\Big)\,\d
      x\d t + \int_Q \zeta(u)\Big(\alpha-H(x,Du)-Du\cdot\vv\Big)
      m\,\varphi\,\d x\d t\\&
      -\bigg(\int_{\R^d}Z(u^+(0,x))\varphi(0,x)\,\d\mu_0(x)-
      \int_{\R^d}Z(u^-(1,x))\varphi(1,x)\,\d\mu_1(x)\bigg)
    \end{aligned}
\end{equation}
for every $\varphi\in C^\infty_c(\R^{d+1})$. 
Since $\varphi,\zeta\ge0$ yield
$T_\zeta(\varphi)\ge0$ and $\supp(\varphi)\subset \R^{d+1}\setminus Q
\ \Rightarrow\ T_\zeta(\varphi)=0$, we know that there exists a
unique nonnegative Radon measure
$\vartheta_\zeta$ such that 
\begin{equation}
  \label{eq:163}
  T_\zeta(\varphi)=\int_{\R\times
    \R^d}\varphi(t,x)\,\d\vartheta_\zeta(t,x)
  \quad \text{for every }\varphi\in C^\infty_c(\R^{d+1}),\quad
  \supp(\vartheta_\zeta)=\overline Q.
\end{equation}
In this way, $T_\zeta$ can be extended to a linear positive functional
on $C^0_c(\R^{d+1})$.
In particular, choosing $\varphi$ with compact support in $Q$
\begin{equation}
  \label{eq:172}
  \partial_t\big(Z(u)m\big)+
    \nabla\cdot\big(Z(u)m\,
    \vv\big)+\zeta(u)\Big(\alpha-H(x,Du)-Du\cdot\vv\Big)m=\vartheta_\zeta
    \quad \text{in }\DD'(Q).
\end{equation}
We want now to associate a nonnegative Radon measure $\vartheta$ 
on $\R\times \R^{d+1}$ to $T$  so that 
\begin{equation}
  \label{eq:164}
  T_\zeta(\varphi)=\int_{\R\times \overline Q}\zeta(r)\varphi(t,x)\,\d
  \vartheta(r,t,x)
  \quad
  \text{for every }\zeta\in C^0_c(\R),\ \varphi\in C^{0}_c(\R^{d+1}).
\end{equation}
\begin{proposition}
  For every $(u,\alpha)\in \hj_q(Q,H)$ and 
  $(m,\vv)\in \CE 2p{Q;\mu_0,\mu_1}$ with $\mu_i\ll\LL^d$,
  there exists a unique positive \emph{contact-defect} Radon measure
  $\vartheta$ on $\R\times \R^{d+1}$ satisfying \eqref{eq:164}.
  Moreover,  
  $\vartheta$ is supported on $\R\times \overline Q$.
\end{proposition}
\begin{proof}
  The functional $(\zeta,\varphi)\mapsto T_\zeta(\varphi)$ 
  is bilinear on $C^0_c(\R)\times C^0_c(\R^{d+1})$ and it is positive
  on positive functions, so the representation formula \eqref{eq:164}
  follows
  by the extension of Riesz representation Theorem to positive
  bilinear functionals, see Theorem \ref{thm:bilinear-Riesz} in the Appendix.
\end{proof}
Let us now write the distributional definition of $T_\zeta$ in a
slightly different way.
We introduce the nonnegative functions $Y_H,Y_F$ related to Fenchel duality
between the pairs $H,L$ and $F,F^*$:
\begin{align}
  \label{eq:81}
  Y_H(x,\pp,\vv)&:=H(x,\pp)+\pp\cdot \vv+L(x,\vv),\quad x,\pp,\vv\in
                  \R^d;&&\\
  \label{eq:81bis}
  Y_F(x,m,\alpha)&:=F(x,m)-\alpha m+F^*(x,\alpha),\quad
                   x\in \R^d,\ \alpha\in \R,\ m\ge 0.
\end{align}
\begin{corollary}
  \label{cor:second_duality}
  Let $(u,\alpha)\in \hj_q(Q,H)$ 
  and let $(m,\vv)\in \CE 2p {Q;\mu_0,\mu_1}$,
  and let us consider the measurable functions
  \begin{equation}
    \label{eq:141}
    \sfY_H:=Y_H(\cdot,Du,\vv),\quad
    \sfY_F:=Y_F(\cdot,m,\alpha),\quad
    \sfL:=L(\cdot,\vv),\quad
    \sfF:=F(\cdot,m),\quad
    \sfF^*:=F^*(\cdot,\alpha).
  \end{equation}
  For every $\zeta\in C^0_c(\R)$ and $Z\in C^1_b(\R)$ with
    $Z'=\zeta$ 
    we have
    \begin{equation}
      \label{eq:111}
      \begin{aligned}
        \zeta(u)\big(\sfY_H\,m+\sfY_F\big)
        +\vartheta_\zeta
      &= 
        \zeta(u)\big(\sfL\,m+\sfF+\sfF^*\big)+
        \partial_t(Z(u)m)+\nabla\cdot(Z(u)m\vv)
      \end{aligned}\end{equation}
    in the sense of distributions of $\DD'(Q)$, and
      also in duality with functions $\varphi\in
  C^1_c(\overline Q)$ if $\mu_0,\mu_1\ll\LL^d$:
  \begin{equation}
    \label{eq:109b}
    \begin{aligned}
      &\int_Q 
      \zeta(u)\big(\sfY_H \, m+\sfY_F\big)
      \varphi\,\d
      x\,\d t+\int_{\R\times
        \R^{d+1}}\zeta(u)\varphi(t,x)\,\d\vartheta(u,t,x)
      \\&= \int_Q \zeta(u)\big(\sfL\,m+\sfF\big) 
      \varphi\,\d x\d
      t
      -\int_QZ(u)m\big(\partial_t\varphi+D\varphi\cdot\vv\big)\,\d
      x\d t
      \\&\qquad-
      \bigg(\int_{\R^d}Z(u^+(0,x))\varphi(0,x)\,\d\mu_0(x)-
      \int_{\R^d}Z(u^-(1,x))\varphi(1,x)\,\d\mu_1(x)-
      \int_Q \zeta(u)\sfF^* 
      \varphi\,\d x\,\d t\bigg).
    \end{aligned}
\end{equation}
\end{corollary}
\begin{proof}
  The proof is a simple manipulation of \eqref{eq:172} and
  \eqref{eq:154}.
  Let us check e.g.~\eqref{eq:111}, starting directly from \eqref{eq:172}:
  \begin{align*}
     \zeta(u)
    &\big(\sfL\,m+\sfF+\sfF^*\big)+
        \partial_t(Z(u)m)+\nabla\cdot(Z(u)m\vv)
      \\&=\zeta(u)
    \big(\sfL\,m+\sfF+\sfF^*\big)
          -
      \zeta(u)\big(\alpha-\sfH-Du\cdot\vv\big)m+\vartheta_\zeta
    \\&=\zeta(u)\big(\sfL\,m +\sfH+Du\cdot\vv\big)
        +\zeta(u)\big(\sfF-\alpha m+\sfF^*\big)+\vartheta_\zeta
        =\zeta(u)\big(\sfY_H\,m+\sfY_F\big)+\vartheta_\zeta.
  \end{align*}
\end{proof}
We conclude this section by deriving
an important formula concerning the total mass of the measures 
$\vartheta_\zeta$ and $\vartheta$. 
\begin{theorem}
  \label{thm:duality-with-defect}
  Let $(u,\alpha)\in \hj_q(Q,H)$ 
  and let $(m,\vv)\in \CE 2p {Q;\mu_0,\mu_1}$ be a distributional solution to
  the continuity equation according to Definition \ref{def:CE},
  and let us keep the notation of \eqref{eq:81}, \eqref{eq:81bis},
  \eqref{eq:141}.
  If $\mu_0,\mu_1\ll \LL^d$, for every nonnegative $\zeta\in
  C^0_c(\R)$ $\vartheta_\zeta$ has finite total mass and 
  for every $\eta\in C^1_c(\R),\ \zeta\in C^0_c(\R)$ we have
  \begin{align}
    \label{eq:109cc}
    &\begin{aligned}
      &\int_Q 
      \zeta(u)\big(\sfY_H \, m+\sfY_F\big)\eta(t)
      \,\d
      x\,\d t+\int_{\overline Q}\eta\,\d\vartheta_\zeta
      +\int_Q Z(u)m\eta'(t)\,\d x\,\d t
      \\&\qquad= \int_Q \zeta(u)\big(\sfL\,m+\sfF\big) \eta(t)
      \,\d x\d
      t
      -
      \bigg(\int_{\R^d}Z(u^+_0)\eta_0\,\d\mu_0-
      \int_{\R^d}Z(u^-_1)\eta_1\,\d\mu_1-
      \int_Q \eta(t)\zeta(u)\sfF^* 
      \,\d x\,\d t\bigg),
    \end{aligned}\\
    \label{eq:109cold}
           &\begin{aligned}
      &\int_Q 
      \zeta(u)\big(\sfY_H \, m+\sfY_F\big)
      \,\d
      x\,\d t+\vartheta_\zeta(\overline Q)
      \\&\qquad= \int_Q \zeta(u)\big(\sfL\,m+\sfF\big) 
      \,\d x\d
      t
      -
      \bigg(\int_{\R^d}Z(u^+_0)\,\d\mu_0-
      \int_{\R^d}Z(u^-_1)\,\d\mu_1-
      \int_Q \zeta(u)\sfF^* 
      \,\d x\,\d t\bigg).
    \end{aligned}
\end{align}
Moreover, for every $0<s<t<1$ 
we have $\vartheta(\R\times \overline{Q_{s,t}})<\infty$ and there exists
an at most countable set $J\subset (0,1)$ such that
for every $s,t\in D_p[\mu] \setminus J$ we have
$\vartheta(\R\times\{s\}\times \R^d)=\vartheta(\R\times\{t\}\times\R^d)=0$ and
\begin{equation}
    \label{eq:109d}
    \begin{aligned}
      &\int_{Q_{s,t}}
      \big(\sfY_H \, m+\sfY_F\big)\,\d\lambda
      +\vartheta(\R\times\overline{Q_{s,t}})
      = \int_{Q_{s,t} } \big(\sfL\,m+\sfF\big) 
      \,\d x\,\d t 
      -
      \Big(\int_{\R^d}\sfu_s\,\d\mu_s-
      \int_{\R^d} \sfu_t\,\d\mu_t-
      \int_{Q_{s,t}} \sfF^* 
      \,\d x\,\d t
      \Big).
    \end{aligned}
\end{equation}
\eqref{eq:109d} also holds at $s=0$ (resp.~$t=1$) 
if 
$\int_\R (u_0^+\land 0)\,\d\mu_0>-\infty$ (resp.~$\int_\R (u_1^-\lor 0)\,\d\mu_1<+\infty$).
If both these integrals are finite then $\vartheta$ has finite total mass and we have
\begin{equation}
    \label{eq:109dbis}
    \begin{aligned}
      &\int_Q 
      \big(\sfY_H \, m+\sfY_F\big)
      \,\d
      x\,\d t+\vartheta(\R\times\overline Q)
      = \int_Q \big(\sfL\,m+\sfF\big) 
      \,\d x\d
      t
      -
      \Big(\int_{\R^d}u^+_0\,\d\mu_0-
      \int_{\R^d} u^-_1\,\d\mu_1-
      \int_Q \sfF^* 
      \,\d x\,\d t\Big).
    \end{aligned}
\end{equation}
\end{theorem}

\begin{proof}
  Let us fix a nonnegative radial test function
  $\xi:\R^d\to[0,1]$ of class $C^1$ 
  such that 
  \begin{displaymath}
    \xi(x)\equiv 1\ \text{if }|x|\le 1,\quad
    \xi(x)\equiv 0\ \text{if }|x|\ge 2;\quad
    \xi_k(x):=\xi(2^{-k}x).
  \end{displaymath}
  Notice that $\xi_k$ is increasing w.r.t.~$k$; we also set
  $\varphi_k(t,x):=\eta(t)\xi_k(x)$
  and we write \eqref{eq:109b}
  for a fixed nonnegative $\zeta$ and $\varphi:=\varphi_k$.
  We notice that
  \begin{displaymath}
    \text{$\partial_t\varphi_k\equiv \eta'\xi_k$ on $Q$},
  \end{displaymath}
  \begin{displaymath}
    \int_Q Z(u)m\eta|D\xi_k\cdot \vv|\,\d t\,\d x\le 
    2^{-k}\big(\sup_{\R^d}|D\xi|\big)\big(\sup_\R Z \big)
    \big(\sup_\R\eta)\int_Q|\vv|^2m\,\d x\,\d t\to
    0\quad\text{as }k\to\infty;
  \end{displaymath}
  since $\varphi_k\uparrow 1$, 
  applying Lebesgue Dominated Convergence Theorem we easily get
  \begin{displaymath}
    \lim_{k\to\infty}\int_Q \zeta(u)\big(\sfY_H \, m+\sfY_F\big)
      \varphi_k\,\d
      x\,\d t=
      \int_Q \zeta(u)\big(\sfY_H \, m+\sfY_F\big)\eta
      \,\d
      x\,\d t,
  \end{displaymath}
  \begin{displaymath}
         \lim_{k\to\infty} \int_Q Z(u)m\eta'(t)  \varphi_k\,\d x\,\d t=
     \int_Q Z(u)m\eta'(t) \,\d x\,\d t,
  \end{displaymath}
  \begin{displaymath}
    \lim_{k\to\infty}\int_Q \zeta(u)\big(\sfL\,m+\sfF\big) 
      \varphi_k\,\d x\d
      t=
      \int_Q \zeta(u)\big(\sfL\,m+\sfF\big) \eta
      \,\d x\d
      t,\quad
      \lim_{k\to\infty}
       \int_Q \zeta(u)\sfF^* 
      \varphi_k\,\d x\,\d t
      =
       \int_Q \zeta(u)\eta\sfF^* 
      \,\d x\,\d t
  \end{displaymath}
  \begin{displaymath}
    \lim_{k\to\infty}\int_{\R^d}Z(u^+_0)\varphi_k(0,\cdot)\,\d\mu_0=
    \int_{\R^d}Z(u^+_0)\eta_0\,\d\mu_0,\quad
    \lim_{k\to\infty}
    \int_{\R^d}Z(u^-_1)\varphi_k(1,\cdot)\,\d\mu_1
    =\int_{\R^d}Z(u^-_1)\eta_1\,\d\mu_1.
  \end{displaymath}
  Eventually, choosing $\zeta\ge0$ and 
  $\eta\ge 0$,
  Beppo Levi monotone convergence theorem yields
  \begin{displaymath}
    \lim_{k\to\infty}\int_{\R\times
        \R^{d+1}}\zeta(u)\varphi_k(t,x)\,\d\vartheta(u,t,x)
    = \int_{\R\times
        \R^{d+1}}\eta(t)\zeta(u)\,\d\vartheta(u,t,x)
    \end{displaymath}
  so that 
  \begin{equation}
    \label{eq:109c}
    \begin{aligned}
      &\int_Q 
      \zeta(u)\big(\sfY_H \, m+\sfY_F\big)\eta(t)
      \,\d
      x\,\d t+\int_{\R\times
        \R^{d+1}}\eta(t)\zeta(u)\,\d\vartheta(u,t,x)
      +\int_Q Z(u)\eta'(t)m\,\d x\,\d t
      \\&\qquad=
      \int_Q \zeta(u)\big(\sfL\,m+\sfF\big) \eta(t)
      \,\d x\d
      t
      -
      \bigg(\int_{\R^d}Z(u^+_0)\,\d\mu_0-
      \int_{\R^d}Z(u^-_1)\,\d\mu_1-
      \int_Q \zeta(u)\eta(t)\sfF^* 
      \,\d x\,\d t\bigg).
    \end{aligned}
\end{equation}
Writing the above identity for the positive and the negative part of
$\zeta$,
we obtain that it holds for every $\zeta\in C^0_c(\R)$,
together with its $\eta$-localized version \eqref{eq:109cc}.

We can now write  \eqref{eq:109c} by choosing 
$\eta\ge0$ with compact support in $(0,1)$ and $\eta\equiv 1$ on 
the interval $[s,t]$ and choosing an increasing
sequence
of nonnegative functions $\zeta_k\in C^0_c(\R)$ converging to $1$, so that
$(Z_k(u)\lor 0):=\int_0^{u\lor 0}\zeta_k(v)\,\d v$ is monotonically
increasing to $u\lor 0$ and 
$(Z_k(u)\land 0):=-\int_{u\land 0}^0\zeta_k(v)\,\d v$ is monotonically
decreasing to $u\land 0$. The application of Beppo Levi's and Lebesgue
dominated convergence theorem yields
  \begin{equation}
    \label{eq:109cbis}
    \begin{aligned}
      &\int_Q 
      \big(\sfY_H \, m+\sfY_F\big)\eta(t)
      \,\d
      x\,\d t+\int_{\R\times
        \R^{d+1}}\eta(t)\,\d\vartheta(u,t,x)
      +\int_Q u\eta'(t)m\,\d x\,\d t
      \\&\qquad=
      \int_Q \big(\sfL\,m+\sfF\big) \eta(t)
      \,\d x\d
      t+
      \int_Q \zeta(u)\eta(t)\sfF^* 
      \,\d x\,\d t
    \end{aligned}
\end{equation}
which in particular shows that $\vartheta(\R\times Q_{s,t})<\infty$ for every
$0<s<t<1$. The set $J_1:=\big\{s\in (0,1):\vartheta(\R\times\{s\}\times
\R^d)>0\big\}$ is therefore at most countable and $D_p[\mu]\setminus
J_1$ has full measure in $(0,1)$.
On the other hand, \eqref{duality_formula} shows that the map
\begin{equation}
  \label{eq:195}
  t\mapsto \int_{\R^d}\sfu_t\,\d\mu_t+\int_{1/2}^t
  \int_{\R^d}\big(\alpha-H(x,Du)-Du\cdot\vv\big)m
  \,\d x\,\d r,\quad
  t\in D_p[\mu]
\end{equation}
is finite and increasing in $D_p[\mu]$, so that it has an at most
countable
jump set $J_2$. Setting $J:=J_1\cup J_2$ and choosing
a decreasing sequence $\eta_h$ 
such that $\eta_h\equiv 1$ on $[s,t]$, $\eta_h\in [0,1]$, and
$\eta_h\downarrow 0$ in $\R\setminus [s,t]$, \eqref{eq:109cbis} yields
\eqref{eq:109d}
at every $s,t\in D_p[\mu]\setminus J$.

A similar argument holds for $s=0$ (resp.~$t=1$) by 
assuming
$\int_\R (u_0^+\land0)\,\d\mu_0>-\infty$
(resp.~$\int_\R (u_1^-\lor 0)\,\d\mu_1<+\infty$) and choosing 
a function $\eta$ identically $1$ in the interval $[-1,t]$ (resp.~$[s,2]$).
\end{proof}
We deduce from the previous Theorem a precise characterization of the
optimality conditions. In order to express in a convenient way the
trace conditions, we will use approximate limits: recall that 
for a measurable map $g:A\to \R$ (a.e.) defined in a measurable subset $A\subset \R$ 
and a point $t_0\in \R$ of positive density for $A$ we have
\begin{equation}
  \label{eq:142}
  \begin{aligned}
    \apliminf_{t\to t_0}g(t)&:= \sup\Big\{\ell\in \R:
    \lim_{r\down0}r^{-1}{\LL^1(\{t\in A:
      |t-t_0|<r, \ g(t)<\ell\})}=0\Big\},\\
    \aplimsup_{t\to t_0}g(t)&:= \inf\Big\{\ell\in \R:
    \lim_{r\down0}r^{-1}{\LL^1(\{t\in A:
      |t-t_0|<r, \ g(t)>\ell\})}=0\Big\},
  \end{aligned}
\end{equation}
and we call $\aplim_{t\to t_0}g(t)=\ell$ if $\ell=\apliminf_{t\to
  t_0}g(t)=\aplimsup_{t\to t_0}g(t)$.
\begin{corollary}
  \label{cor:optimality}
  Let $(u,\alpha)\in \hj_q(Q,H)$ 
  and let $(m,\vv)\in \CE 2p {Q;\mu_0,\mu_1}$ with $\mu_0,\mu_1\ll \LL^d$ 
  and $\int_{\R^d}(u_0^+\lor 0)\,\d\mu_0<+\infty$,
  $\int_{\R^d}(u_1^-\land 0)\,\d\mu_1>-\infty$.
  Let us set
  \begin{equation}
    \label{eq:174}
    \begin{aligned}
      \cA(u,\alpha)&:=\int_{\R^d}u^+_0\,\d\mu_0- \int_{\R^d}
      u^-_1\,\d\mu_1- \int_Q F^*(x,\alpha(t,x)) \,\d x\,\d t\in [-\infty,+\infty)\\
      \cB(m,\vv)&:=
      \int_{Q}
  \Big(L(x,\vv(t,x))m(t,x)+F(x,m(t,x))\Big)\,\d x\,\d t,
    \end{aligned}
  \end{equation}
  Then
  \begin{equation}
    \label{eq:170}
    \cA(u,\alpha)\le \cB(m,\vv)
  \end{equation}
  and
  the equality holds in \eqref{eq:170} if and only if
  $u_i\in L^1(\R^d,\mu_i)$ and
  \begin{equation}
    \label{eq:173}
    \begin{aligned}
      Y_H(x,Du,\vv)&=0&&\text{for $\tilde\mu$-a.e.~$(t,x)\in Q$},\\
      Y_F(x,\alpha,m)&=0&&\text{for $\LL^{d+1}$-a.e.~$(t,x)\in Q$},\\
      \vartheta(\R\times\overline Q)&=0.&&
    \end{aligned}
  \end{equation}
  Moreover, $\vartheta(\R\times \overline Q)=0$ if and only if
  for every nonnegative $\zeta\in C^0_c(\R)$ 
  \begin{equation}
    \label{eq:172bis}
    \partial_t\big(Z(u)m\big)+
    \nabla\cdot\big(Z(u)m\,
    \vv\big)+\zeta(u)\Big(\alpha-H(x,Du)-Du\cdot\vv\Big)m=0
    \quad \text{in }\DD'(Q),
  \end{equation}
  and the traces $u_0^+$ and $u_1^-$ are taken in the following weak sense
  \begin{equation}
    \label{eq:198}
    \aplim_{t\down0}
    \int_{\R^d} u(t,x)\,\d\mu_t=
    \int_{\R^d}u^+_0\,\d\mu_0,\quad
    \aplim_{t\up1}
    \int_{\R^d} u(t,x)\,\d\mu_t=
    \int_{\R^d}u^-_1\,\d\mu_1.
  \end{equation}
\end{corollary}
\begin{proof}
  \eqref{eq:170} clearly follows from \eqref{eq:109d} since
  $\vartheta$ is a positive measure
  (notice that if the integral of the negative part of $u_0^+$
  w.r.t.~$\mu_0$ or 
  the integral of the positive part of $u_1^-$
  w.r.t.~$\mu_1$ are not finite, $\cA(u,\alpha)=-\infty$ and
  \eqref{eq:170} is trivially satisfied).
  The optimality conditions \eqref{eq:173} are also an immediate
  consequence of \eqref{eq:109d}, since $\sfY_H$ and $\sfY_F$ are
  nonnegative.

  Let us check the last statement. If $\vartheta(\R\times \overline Q)=0$ then
  \eqref{eq:172bis} is an immediate consequence of \eqref{eq:172}.
  \eqref{eq:109d} shows that the map 
  $t\mapsto \int_{\R^d}\sfu(t,x)\,\d\mu_t(x)$, $t\in D_p[\mu]$, is the restriction
  to $D_p[\mu]$ of an absolutely continuous map
  with limits $\int_{\R^d}u^+_0\,\d\mu_0$ 
  (resp.~$\int_{\R^d}u^-_1\,\d\mu_1$)  as $t\down0$ (resp.~as
  $t\up1$),
  so that \eqref{eq:198} holds.

  Conversely, \eqref{eq:172bis} shows that $\vartheta(\R\times Q)=0$.
  \eqref{eq:109d} then shows that there exists an at most countable
  set $J$ such that 
  \begin{align*}
    \vartheta(\R\times \{0\}\times \R^d)=
    I(t)+\int_{\R^d}\sfu_t\,\d\mu_t-
    \int_{\R^d}u_0^+\,\d\mu_0\quad
    \text{for every }t\in D_p[\mu]\setminus J,
  \end{align*}
  where
  \begin{displaymath}
    I(t):=\int_{Q_{0,t}}\Big(\big(\sfL-\sfY_H\big)m+\big(\sfF-\sfY_F-\sfF^*\big)\Big)\,\d\lambda
  \end{displaymath}
  Since $\lim_{t\down0}I(t)=0$ \eqref{eq:198} yields 
  $\vartheta(\R\times \{0\}\times \R^d)=0$. A similar argument 
  can be used to show $\vartheta(\R\times \{1\}\times \R^d)=0$.
\end{proof}
Notice that it always holds
\begin{equation}
  \label{eq:198bis}
  \apliminf_{t\down0}
  \int_{\R^d} u(t,x)\,\d\mu_t\ge
  \int_{\R^d}u^+_0\,\d\mu_0,\quad
  \aplimsup_{t\up1}
  \int_{\R^d} u(t,x)\,\d\mu_t\le
  \int_{\R^d}u^-_1\,\d\mu_1,
\end{equation}
so that \eqref{eq:198} is in fact equivalent to the property
  \begin{equation}
    \label{eq:198tris}
    \aplimsup_{t\down0}
    \int_{\R^d} u(t,x)\,\d\mu_t\le 
    \int_{\R^d}u^+_0\,\d\mu_0,\quad
    \apliminf_{t\up1}
    \int_{\R^d} u(t,x)\,\d\mu_t\ge
    \int_{\R^d}u^-_1\,\d\mu_1.
  \end{equation}

\nc

\section{A Variational approach to MFPP}
\label{sec:variational-approach}

It is known (see \cite{lasry2007mean}) that solutions to MFG systems can be, at least formally, characterized as minimizers of two optimal control problems.
Here we adopt the same strategy in the context of MFPPs. 
\GGG Following the standard approach of 
the dynamical formulation of Optimal Transport problems, we first
study
a primal formulation, consisting in the minimization of
a convex action-entropy functional among all the solutions 
$(m,\vv)\in \CE 2p{Q;\mu_0,\mu_1}$ to the continuity
equation
connecting two given measures 
$\mu_i\in \cP_2(\R^d)$.

The dual problem naturally arises as a supremum
of a concave cost functional, defined on subsolutions to the
Hamilton-Jacobi equation.

As usual, in all this section we will refer to the main structural
conditions stated in \ref{h.1}

\subsection{The variational structure of the primal problem}
\label{subsec:variational-primal}

We state the primal problem for an arbitrary pair of measures
$\mu_i\in\cP_2(\R^d)$;
at this stage, the minimal assumption requires that $\CE
2p{Q;\mu_0,\mu_1}$ is not empty (or, equivalently, that $\mu_i\in \cP_{2,p}(\R^d)$)
and
we do not need to add any other a priori regularity on $\mu_0,\mu_1$.
\begin{problem}\label{p:control_eq:cont}
Let $\mu_i\in \cP_2(\R^d)$ be given.
We look for the minimizer of the 
functional 
\begin{equation}
  \label{eq:85}
  \cB(m,\vv):=
  \int_{0}^1 \int_{\R^d}
  \Big(L(x,\vv(t,x))m(t,x)+F(x,m(t,x))\Big)\,\d x\,\d t,
\end{equation}
among all solutions $(m,\vv)\in \CE 2p{Q;\mu_0,\mu_1}$ of the continuity
equation connecting $\mu_0$ to $\mu_1$.
\end{problem}
In order to prove the existence of a solution to Problem 
\ref{p:control_eq:cont}, 
we first show that a control of $m$ in $L^1_{1+|x|}(Q)$ 
yields uniform bounds of the sublevels of $\cB$.
Notice that at this first stage we are not assuming that 
$(m,\vv)$ solves the continuity equation
nor that $m$ has integral $1$.
\begin{lemma}[A priori estimates I]
  \label{le:apriori1}
  Let $m\in L^1_\weight(Q)$ be nonnegative and let $\vv:Q\to \R^d$ be a Borel vector field
  satisfying
  \begin{align}
    \label{eq:143}
    \int_Q(1+|x|)\,m(t,x)\,\d x\,\d t&\le M,\\
    \label{eq:143bis}
    \cB(m,\vv)&\le B
  \end{align}
  for suitable constants $M,B>0$. 
  Then $\ww:=m\vv\in L^{2p/(p+1)}(Q)$ and there exists a constant
  $E>0$ only depending on 
  $B,$ $M$,
  and on
  the structural constants
  \begin{equation}
    \label{eq:107}
    c_H,\ c_H^\pm,\ c_f,\  C_f:=\|\gamma_f\|_{L^q(\R^d)},
  \end{equation}
  such that
  \begin{equation}
    \label{eq:144}
    \int_Q \Big(|\vv|^2m+m^p\Big)\,\d x\,\d t\le E,\quad
  \int_Q |\ww|^{\frac{2p}{p+1}}\,\d x\,\d t\le E.
  \end{equation}
\end{lemma}
\begin{proof}
  Let us first observe that \eqref{eq:7} and
  the elementary inequalities
  \begin{displaymath}
    \gamma_f m\le \frac1{p(2c_f)^p}m^p+
    \frac{(2c_f)^q}{q}  \gamma_f^q,\quad
  \end{displaymath}
  yield 
  \begin{equation}
    \label{eq:146}
    \int_Q F(x,m)\,\d x\,\d t\ge 
    \frac1{2pc_f^p}\int_Qm^p\,\dx\,\d t-
    C_{F},\quad
    C_F:=\frac{(2c_f\,C_f)^q}q,
  \end{equation}
  and
  \begin{equation}
    \label{eq:124a}\cB(m,\vv)
   =\int_{Q}L(x,\vv)m+F(x,m)\,\d x\,\d t
      \ge 
    \int_{Q}\Big(\frac 1{2c_H}m|\vv|^2+\frac{1}{2p\,c_f^p}\,m^p\Big)\,\d x \d t -(C_F+c_H^+ M)
\end{equation}
Setting
$c:=2\max(c_H,p\,c_f^p)$, we conclude that 
\begin{equation}
  \label{eq:91}
  \int_Q \Big(
  m|\vv|^2+m^p\Big)\,\d x\,\d r
  \le E:=c(B+C_F+c_H^+ M).
\end{equation}
We eventually obtain a bound of $\ww=m\vv$ in $L^{2p/(p+1)}(Q)$ by
\begin{equation}
\|m\vv\|_{L^{2p/(p+1)}(Q)}\le \|m^{1/2}\|_{L^{2p}(Q)}
\|m^{1/2}\vv\|_{L^2(Q;\R^d)}=\|m\|_{L^p(Q)}^{1/2}\Big(\int_{Q}|\vv|^2m\,\d
x\,\d t\Big)^{1/2}\le E.\label{eq:92}
\end{equation}
\end{proof}
\eqref{eq:144} shows that 
for $(m,\vv)\in \CE 2p{Q;\mu_0,\mu_1}$ an upper bound of $\cB(m,\vv)$ 
provides an
upper bound for $\KL_{2,p}(\mu_0,\mu_1)$.
The converse property also holds.
\begin{lemma}[A priori estimates II]
\label{le:apriori}
Let us consider
\begin{equation}
  \label{eq:199}
  \mu_i\in \cP_2(\R^d)\quad
  \text{with }
  \quad
  \int_{\R^d}|x|^2\,\d\mu_i=M_i,\quad
  \KL_{2,p}(\mu_0,\mu_1)=K<\infty.
\end{equation}
Then the infimum of $\cB$ on $\CE 2p{Q;\mu_0,\mu_1}$ is finite
and there are constants $B$ and $M$ only depending on 
$K$, $M_i$ 
and on the structural constants 
of \eqref{eq:107} such that 
the infimum can be restricted to all the pairs $(m,\vv)\in \CE
2p{Q;\mu_0,\mu_1}$ 
satisfying the apriori bounds \eqref{eq:143} and \eqref{eq:143bis};
in particular, $(m,\vv)$ also satisfy \eqref{eq:144}.
\end{lemma}
\begin{proof}
Let us pick a curve $(\bar m,\bar \vv)\in \CE2p{Q;\mu_0,\mu_1}$ with 
$\int_Q(|\bar\vv|^2\bar m+\bar m^p)\,\d\lambda\le 4K.$ 
We get the bound
\begin{equation}\label{eq:106}
  \int_Q F(x,\bar m)\,\d x\,\d t \le C_1:=4\frac
  {c_f^p}{p}K+
  C_f 
  (4K)^{1/p},
\end{equation}
thanks to the growth conditions \eqref{eq:7}. 
Since
\begin{equation}
  \label{eq:200}
  \Big(\int_{\R^d}|x|^2\,\d\bar\mu_t\Big)^{1/2}
  \le M_0^{1/2}+W_2(\mu_t,\mu_0)\le 
  M_0^{1/2}+(2K)^{1/2},
\end{equation}
we also get
\begin{equation}
  \int_0^1\int_{\R^d}L(x,\bar\vv)\bar m\,\d x\,\d t\le
  C_2:=c_HK+3c^-_H(1+M_0+2K)
\label{eq:117}
\end{equation}
thanks to 
\eqref{eq:128}, \eqref{eq:9}. In particular $\cB(\bar m,\bar \vv)\le B:=C_1+C_2$.

We now fix 
an element $(m,\vv)\in \CE 2p{Q;\mu_0,\mu_1}$ with $\ww=m\vv$ satisfying
$\cB(m,\vv)\le B$, and
we introduce the absolutely continuous function 
\begin{equation}
  \label{eq:119}
  t\mapsto M(t):=\Big(\int_{\R^d}|x|^2\,\d\mu_t\Big)^{1/2}\ge
  \int_{\R^d}|x|\,\d\mu_t,
  \quad
  t\in[0,1],
\end{equation}
whose derivative satisfies
\begin{equation}
  \label{eq:123}
  \left|\frac\d{\d t}M(t)\right|\le 
  \Big(\int_{\R^d}
  |\vv_t|^2 \,\d \mu_t\Big)^{1/2},\quad
  \left|\frac\d{\d t}M^2(t)\right|\le M^2(t)+\int_{\R^d}
  |\vv_t|^2 \,\d \mu_t\quad\text{a.e.~in }(0,1).
\end{equation}
Arguing as for \eqref{eq:124a} we obtain for a.e.~$t\in (0,1)$
\begin{align}
  B(t)&:=\int_{\R^d}L(x,\vv_t)m_t+F(x,m_t)\,\d x
        \ge 
    \int_{\R^d}\Big(\frac 1{2c_H}m_t|\vv_t|^2+\frac{1}{2p\,c_f^p}\,m_t^p\Big)\,\d x-
      C_F-c_H^+
    M(t).
\end{align}
By \eqref{eq:123} we deduce that 
\begin{align*}
  \left|\frac\d{\d t}M^2(t)\right|\le 2c_H (B(t)+C_3)+ 2M^2(t),\quad
  C_3:= C_F+\frac 12(c^+_H)^2
\end{align*}
so that a simple comparison argument yields the two estimates
\begin{equation}
  \label{eq:120}
  M^2(t)\le \Big(M_0+2c_HC_3+2c_H\int_0^t B(r)\,\d r\Big)\mathrm e^{2},
  \quad
  M^2(t)\le \Big(M_1+2c_HC_3+2c_H\int_t^1 B(r)\,\d r\Big)\mathrm e^{2}.
\end{equation}
Summing up the previous inequalities we end up with
\begin{equation}
  \label{eq:121}
  1+M(t)\le M:=1+\mathrm e\Big(M_0+M_1+2c_H(C_3+B)\Big)^{1/2}
  \quad\text{for every }t\in [0,1].
\end{equation}
\end{proof}

\begin{remark}
  \label{re:tedious}
  Notice that Lemma \ref{le:basic-Lp} yields a uniform estimate of the
  constants $B$ and $M$ of the previous Lemma in terms of
  \begin{equation}
    \label{eq:107bis}
    M_{01}:=\max(\|m_0\|_{L^p\cap
      L^1_\weight(\R^d)}, 
    \|m_1\|_{L^p\cap L^1_\weight(\R^d)}).
  \end{equation}
\end{remark}
We will now write Problem \ref{p:control_eq:cont}
in terms of a minimization of a \emph{convex} functional on a
\emph{convex set}. To this aim, 
for every $(m,\vv)\in \CE 2pQ$ we introduce the 
functions
\begin{equation}
  \label{eq:86}
  \ww:=m\vv\in 
  L^{\frac{2p}{p+1}}
  (Q;\R^d),\quad
  \tilde L(x,m,\ww):=
  \begin{cases}
    L(x,\ww/m)m&\text{if }m\neq0,\\
    +\infty&\text{if }m=0,\ \ww\neq 0,\\
    0&\text{if }m=0,\ \ww=0,
  \end{cases}
\end{equation}
and write $\cB$ as a functional of the pair $(m,\ww)$:
\begin{equation}
  \label{eq:87}
  \tilde\cB(m,\ww):=
  \int_{0}^1 \int_{\R^d}
  \Big(\tilde L\big(x,m(t,x),\ww(t,x)\big)+F(x,m(t,x))\Big)\,\d x\,\d t.
\end{equation}
Thanks to the coercivity conditions of 
$L$, if $m\in L^p\cap L^1_\weight(Q)$, $\ww\in L^1(Q;\R^d)$ and 
$\tilde \cB(m,\ww)<\infty$,
then we can write $\ww=m\vv$ for a vector field $\vv\in
L^2(Q,\tilde \mu;\R^d)$.
If moreover $(m,\ww)$ satisfy the conservation law
\begin{equation}
  \label{eq:89}
  \partial_t m+\nabla\cdot \ww=0\quad\text{in }\DD'(Q),
\end{equation}
then $(m,\vv)\in \CE 2pQ$ so that $m$ admits a continuous
representative
$\mu\in \ac^2([0,1];\cP_2(\R^d))$
and the initial and final conditions $\mu_{t=i}=\mu_i$, $i=0,1$,
make sense.

Therefore Problem \ref{p:control_eq:cont} is in fact equivalent to 
the minimization of the \emph{convex} functional $\tilde \cB$ 
on the convex set $\K$ depending on $\mu_0,\mu_1$:
\begin{equation}
  \label{eq:88}
  \begin{aligned}
    \K:=\Big\{&(m,\ww)\in
    \big(L^p\cap L^1_\weight(Q)\big)
    \times L^{2p/(p+1)}(Q;\R^d):
    \partial_t m+\nabla\cdot \ww=0\quad\text{in }\DD'(Q),
    \\
    &
    m\lambda=\tilde \mu\quad\text{for a curve }\mu\in
    \ac^2([0,1];\cP_2(\R^d)),
    \quad
    \mu_{t=0}=\mu_0,\ \mu_{t=1}=\mu_1\Big\}.
  \end{aligned}
\end{equation}
\begin{lemma}[Lower semicontinuity of $\tilde\cB$]
  \label{le:lsc}
  If $(m_n,\ww_n)$ is a sequence weakly converging to $(m,\ww)$ 
  in $L^p(Q)\times L^{2p/(p+1)}(Q;\R^d)$ and satisfyng 
  $m_n\ge 0$, $\int_Q(1+|x|^2)\,m_n\,\d x\,\d t\le M$, then 
  for every choice of $g,\gg,h$ 
  measurable functions 
  $  \gg:Q\to\R^d$, $g:Q\to \R$, $h:Q\to\R$ with
  \begin{equation}
  \label{eq:202}
  |g(t,x)|\le C_g(1+|x|),\quad
  \gg\in L^\infty(Q;\R^d),\quad
  h\in L^q(Q),
\end{equation}
we have
  \begin{equation}
    \label{eq:203}
    \liminf_{n\to\infty}\tilde\cB(m_n,\ww_n)
    +\int_Q \Big((g+h)m_n+\gg\cdot\ww_n\Big)\,\d\lambda
    \ge \tilde\cB(m,\ww)
    +\int_Q \Big((g+h)m+\gg\cdot\ww\Big)\,\d\lambda.
  \end{equation}
  In particular
  \begin{equation}
    \label{eq:149}
    \liminf_{n\to\infty}\tilde\cB(m_n,\ww_n)\ge \tilde\cB(m,\ww).
  \end{equation}  
\end{lemma}
\begin{proof}
  We can decompose $\tilde \cB$ in the sum
\begin{displaymath}
  \tilde \cB(m,\ww)=\tilde\cB_1(m,\ww)-
  \int_Q \big(\gamma_H^+(x)+\gamma_f(x)\big)m\,\d x\,\d t,
\end{displaymath}
where $\tilde \cB_1$ is defined as $\tilde\cB$ starting from 
the integrands
$F_1(x,m):=F(x,m)+\gamma_f(x)m$ 
and $\tilde L_1(x,m,\ww):=\tilde L(x,m,\ww)+(\gamma_H^+(x))m$.
Since $F_1$ and $\tilde L_1$ are nonnegative, measurable, convex and lower semicontinuous with
respect to the variables $m$ and $(m,\ww)$ respectively,
$\tilde\cB_1$ is weakly lower semicontinuous in $L^p(Q)\times
L^{2p/(p+1)}(Q;\R^d)$ by the general result \cite[Theorem
6.54]{Fonseca-Leoni07}.

Since we can incorporate the contribution of $\gamma_H^+$ in $g$ 
and the contribution of $\gamma_f$ in $h$, it remains to check that 
\begin{equation}
  \label{eq:95}
  \liminf_{n\to\infty}
  \int_Q \Big((g+h)m_n+\gg\cdot\ww_n\Big)\,\d\lambda\ge
  \int_Q \Big((g+h)m+\gg\cdot\ww\Big)\,\d\lambda
\end{equation}
where $g\le 0$.
Weak convergence in $L^p(Q)$ yields 
$  \lim_{n\to\infty}\int_Q h\,m_n\,\d\lambda
=
\int_Q
h\,m\,\d \lambda.$
The contribution of $\gg$ can also be controlled since $\ww_n$ is
weakly converging to $\ww$ in $L^1(Q;\R^d)$.

Concerning the third term arising from the integral against $g$, we use the 
the uniform estimate 
$\int_{Q}(1+|x|^2)\,m_n\,\d x\,\d t
\le M$ and for every $\delta\in (0,1/2)$
\begin{equation}
  \label{eq:148}
  -g(t,x)\le C(\delta^{-1}+\delta |x|^2),\quad
  \LL^d\text{-a.e.~in }\R^d.
\end{equation}
For every $\delta>0$ choosing 
and $K_\delta:=\{x\in \R^d: |x|\le\delta^{-1}\}$ we get
\begin{align*}
  \liminf_{n\to\infty}\int_Q g\,m_n\,\d\lambda
  &
    \ge
    \liminf_{n\to\infty}\int_0^1\int_{\R^d\setminus K_\delta}
    g\,m_n\,\d x\,\d t
    +
    \liminf_{n\to\infty}\int_0^1\int_{K_\delta} g\,m_n\,\d
    x\,\d t
  \\&\ge
      -C\liminf_{n\to\infty}\int_0^1\int_{\R^d\setminus K_\delta}
      (\delta^{-1}+\delta|x|^2)\,m_n\,\d x\,\d t
    +
        \int_0^1\int_{K_\delta} g\,m\,\d x\,\d t
      \\&\ge
          -2C\delta M+
          \int_0^1\int_{K_\delta} g\,m\,\d x\,\d t.
\end{align*}
Since $\delta>0$ is arbitrary, we conclude. 
\end{proof}
\begin{theorem}\label{p:min_B}
For every choice of $\mu_i$ as in \eqref{eq:199}
Problem \ref{p:control_eq:cont} admits a solution 
$(m_\star,{\vv_\star})$, and $( m_\star, m_\star \vv_\star)$ is
a minimizer of $\tilde \cB$ on $\K$.

Moreover, if the map $m\mapsto f(x, m)$ is strictly increasing in $\R$ for
  a.e.~$x\in \R^d$, then the minimizer $m$ is unique and 
  the optimal vector field $\vv$ is unique $m\lambda$-a.e.~in $Q$.
\end{theorem}
\begin{proof}
By the previous Lemmas \ref{le:apriori} and \ref{le:apriori1}, we can minimize $\tilde\cB$ among all curves
$(m,\ww)\in \K$, $\ww=m\vv$, satisfying the apriori bounds
\eqref{eq:143}, \eqref{eq:143bis}
and therefore \eqref{eq:144}.
The existence of a minimizer then follows by the direct method of the
Calculus of Variations. 
\GGG If $(m_n,\vv_n)\in \CE 2p{Q;\mu_0,\mu_1}$ is a minimizing sequence for $\cB$, 
we easily see from \eqref{eq:144} that $m_n$ is uniformly bounded in
$L^p(Q)$ and $\ww_n=m_n\vv_n$ is uniformly bounded in
$L^{2p/(p+1)}(Q;\R^d)$.
The curves $t\mapsto \mu_{n,t}$ are equi uniformly continuous, since
\begin{equation}
  \label{eq:93}
  W_2(\mu_{n,t},\mu_{n,s})\le |t-s|^{1/2}\Big(\int_0^T
  |\vv_n|^2m_n\,\d x\,
  \d t\Big)^{1/2};
\end{equation}
up to extracting a suitable subsequence, we can thus assume that 
$m_n\weakto m$ in $L^p(Q)$, $\ww_n\weakto \ww=m\vv$ in
$L^{2p/(p+1)}(Q;\R^d)$ and $\mu_{n,t}\weakto \mu_t$ weakly in
$\cP(\R^d)$ for every $t\in [0,1]$, with $(m,\vv)\in \CE 2p{Q;\mu_0,\mu_1}$;
the pair $(m,\ww)$ thus belong to $\K$ (in particular it
satisfies
the initial and final conditions). 

We can then apply the lower semicontinuity of $\tilde\cB$ 
given by Lemma \ref{le:lsc}.
The uniqueness of minimizer easily follows by the strict convexity of
the map $m \mapsto F(x,m)$. Once $m$ is uniquely determined, 
the strict convexity of $L(x,\cdot)$ (guaranteed by the
$\pp$-differentiability of $H$) ensures the uniqueness of $\vv$.
\end{proof}

\subsection{The variational structure of the dual problem and the
  minimax principle}
\label{subsec:variational-dual}

In this section we will compute and study the dual formulation of
Problem \ref{p:control_eq:cont}.

\begin{problem}\label{p:control_HJ}
  Let $\mu_i=m_i\LL^d\in \cP_{2,p}^r(\R^d)$ be given. 
    We look for the maximizer of the functional
  \begin{equation}
    \label{eq:98}
    \cA(u,\alpha):=
    \int_{\R^d}u^+(0,x)\,\d\mu_0(x)-
    \int_{\R^d}u^-(1,x)\,\d\mu_1(x)-
    \int_0^1\int_{\R^d}F^*(x,\alpha(t,x))\,\d x\,\d t
  \end{equation}
  among all pairs $(u,\alpha)\in \hj_q(Q,H)$ 
  according to Definition \ref{def:subsolutions}.
\end{problem}
Notice that by Theorem \ref{thm:precise_repr} the functional $\cA$ is
well defined and takes values in $[-\infty,+\infty)$. 
Corollary \ref{cor:second_duality} also shows that 
\begin{equation}
  \label{eq:99}
  \cA(u,\alpha)\le \cB(m,\vv)\quad
  \text{for every }(u,\alpha)\in \hj_q(Q,H),\quad
  (m,\vv)\in \CE 2p{Q;\mu_0,\mu_1},
\end{equation}
so that 
\begin{equation}
  \label{eq:100}
  \sup\Big\{\cA(u,\alpha): (u,\alpha)\in \hj_q(Q,H)\Big\}
  \le \min\Big\{\cB(m,\vv):  (m,\vv)\in \CE 2p{Q;\mu_0,\mu_1}\Big\}.
\end{equation}
We will present two main results; 
first of all, we will justify the duality between 
Problem \ref{p:control_HJ} and Problem \ref{p:control_eq:cont}
by a classical min-max argument starting from a suitable 
saddle point formulation: in this way we will show that there is no duality gap in
\eqref{eq:100},
even if we restrict the set of competitors for the dual problem
to smoother functions. 
This weak duality result holds for arbitrary pair of measures
$\mu_0,\mu_1\in \cP_2(\R^d)$.

Our second result will show that if $\mu_i$ have densities in $L^p$ 
then also Problem \ref{p:control_HJ}
admits a solution and the supremum in \eqref{eq:100} is attained.

We first compute the form of the dual problem by a classical min-max
argument.

\subsubsection*{Saddle point formulation of primal and dual problems}
Let us consider the convex set  
\begin{equation}
  \label{eq:136}
  \B:=\Big\{(m,\ww)\in \big(L^p\cap L^1_\weight(Q)\big)\times 
   L^{2p/(p+1)}(Q;\R^d):  m\ge 0,
   \ \tilde\cB(m,\ww)<\infty\Big\}.
\end{equation}
A pair $(m,\ww)\in \mathbb B$ is a weak solution of
the continuity equation 
\begin{equation}
  \label{eq:125}
  \partial_t m+\nabla\cdot \ww=0\quad\text{in }\DD'(Q)
\end{equation}
with initial and final condition $\mu_0,\mu_1\in \cP_2(\R^d)$ 
if and only if 
\begin{equation}
  \label{eq:126}
  \sup_{u\in \A}
  \int_{\R^d}u_0\,\d\mu_0 
  -\int_{\R^d}u_1\,\d\mu_1
  +\int_Q \big(\partial_t u\,m+Du\cdot
  \ww\big)\,\d x\,\d t=0,
\end{equation}
where
\begin{equation}
  \label{eq:129}
  \A:=\Big\{u\in C^1(\overline Q): \sup_{Q}\frac{|u|+|\partial_t u|}{1+|x|}+|Du|<+\infty\Big\}.
\end{equation}
Notice that when $u\in \A$ and $\alpha:=-\partial_t u+H(x,Du)$ 
\eqref{eq:98} still makes sense for arbitrary measures
$\mu_0,\mu_1\in \cP_2(\R^d)$.
By introducing the saddle function
\begin{equation}
\begin{aligned}
  \cL((m,\ww),u)
  &:=
  \int_Q \Big( \tilde L(x,m,\ww)+F(x,m)\Big)\,\d x\,\d t\\
  &\qquad+
  \int_{\R^d}u_0\,\d\mu_0- 
  \int_{\R^d}u_1\,\d\mu_1
  +\int_Q \Big(\partial_t u\,m+Du\cdot
  \ww\Big)\,\d x\,\d t
\end{aligned}\label{eq:135}
\end{equation}
the primal problem can be equivalently obtained as
\begin{equation}
  \label{eq:134}
  \min_{(m,\sww)\in \K}\tilde\cB(m,\ww)
  =
 \adjustlimits \min_{(m,\sww)\in \mathbb B}\sup_{u\in \mathbb A}\cL((m,\ww),u).
\end{equation}
Notice that $\cL$ is convex w.r.t.~$(m,\ww)$ and concave w.r.t.~$(u,\alpha)$.
A nice application of Von Neumann Theorem yields the following result:
\begin{theorem}
  \label{thm:min-max}
  For every $\mu_0,\mu_1\in \cP_2(\R^d)$ we have
  \begin{equation}
    \label{eq:137}
    \begin{aligned}
      \min_{(m,\sww)\in \K}\tilde\cB(m,\ww)
      &= \adjustlimits\min_{(m,\sww)\in \mathbb B}\sup_{u\in \mathbb
        A}\cL((m,\ww),u) 
      = {\adjustlimits\sup_{u\in \mathbb A}
      \inf_{(m,\sww)\in \mathbb B}\cL((m,\ww),u)}
    \\&=
      \sup\Big\{\cA(u,\alpha):u\in \A,\ -\partial_t u+H(x,Du)\le
      \alpha,\ \alpha\in \cX^q(Q)\Big\}.
    \end{aligned}
  \end{equation}
\end{theorem}
\begin{proof}
  Let us first check that we can apply Von Neumann minimax Theorem 
  \ref{thm:VonNeumann}.
   Clearly $\cL$ satisfies \eqref{eq:112} and
  \eqref{eq:113}
  and it is not restrictive to assume that 
  $C_\star:=1+{\adjustlimits\sup_{u\in \mathbb A}
      \inf_{(m,\sww)\in \mathbb B}\cL((m,\ww),u)} $ is finite.
   We endow $\B$ with the product weak topology
  of $L^p(Q)\times L^{2p/(p+1)}(Q)$.
  
  We choose constants $A_1\ge 1+c_H^+$ and $A_0\ge 1+4c_HA_1+c_H^+$
  and the function $u_\star:=2t(A^2_0+A_1^2 |x|^2)^{1/2}\in \A$; 
  we can check that
\begin{align*}
  \partial_t u_\star+Du_\star\cdot \vv
  &\ge 
    2(A_0^2+A^2_1 |x|^2)^{1/2} -2A_1 \,|\vv|
    \ge
    2(A_0^2+A^2_1 |x|^2)^{1/2} -4c_H{A_1}- \frac1{4c_H}\,|\vv|^2,\\
  L(x,\vv)+\partial_t u_\star+Du_\star\cdot \vv
  &\ge
    \frac1{4c_H}|\vv|^2
    +2(A_0^2+A^2_1 |x|^2)^{1/2} -4c_H{A_1}-c_H^+(1+|x|)
    \\&
        \ge
        \frac1{4c_H}|\vv|^2+(A_0-4c_HA_1-c_H^+)+(A_1-c_H^+)|x|
        \ge 
           \frac1{4c_H}|\vv|^2+(1+|x|),
\end{align*}
so that 
\begin{equation}
  \label{eq:145}
  \cL((m,\ww),u_\star)\ge \int_Q\Big(\frac 1{4c_H}m|\vv|^2+
  \frac 1{2pc_f^p}m^p+\big(1+|x|\big)m\Big)\,\d x\,\d t-C_\cL,
\end{equation}
with $
  C_\cL:=  2A_0+2A_1 \int_{\R^d}|x|\,\d \mu_1+C_F$.

Combining \eqref{eq:145} with the estimate \eqref{eq:92},
we can see that the convex set 
$$\B_\star:=\big\{(m,\ww)\in \B:
\cL((m,\ww),u_\star)\le C_\star\}$$ 
is compact in $\B$.
Lemma \ref{le:lsc} also shows that for every $u\in \A$ 
$(m,\ww)\mapsto \cL((m,\ww),u)$ is lower semicontinuous in $\B_\star$
so that we can interchange the order $\inf$ and $\sup$ in \eqref{eq:137}.

Let us now compute
\begin{align*}
  \inf_{(m,\sww)\in \B}&\cL((m,\ww),u)=
  \int_{\R^d}u_0\,\d\mu_0 -
  \int_{\R^d}u_1\,\d\mu_1 
                         -\mathcal C(u),\\
  \mathcal C(u)
  &:=\sup_{(m,\sww)\in \B}
  \int_Q (-\partial_t u\,m-Du\cdot
  \ww)\,\d x\,\d t
  -\int_Q \Big( \tilde L(x,m,\ww)+F(x,m)\Big)\,\d x\,\d t
  \\&=
  \sup_{m\in L^p\cap L^1_\weight(Q),\, m\ge 0}            
      \bigg(\int_Q \Big(-\partial_t u\,m-F(x,m)\Big)\,\d x\,\d t
      \\&\qquad+\sup_{\svv\in L^2(Q,m\lambda;\R^d)}
      \int_Q\Big(-Du\cdot
      \vv- L(x,\vv)\Big)m\,\d x\,\d t\bigg)
\end{align*}
By applying the general duality theorem \cite[Chap. IX,
Prop. 2.1]{Ekeland-Temam} (see also Remark 2.1)
to the space $L^2(Q,m\lambda ;\R^d)$ (notice that $m$ is a finite measure
and $Du\in L^2(Q, m\lambda ;\R^d)$) we obtain 
\begin{equation}
  \label{eq:151}
  \sup_{\svv\in L^2(Q, m\lambda ;\R^d)}
  \int_Q\Big(Du\cdot
  \vv-L(x,\vv)\Big)m\,\d x\,\d t=
  \int_Q H(x,Du)m\,\d x\, \d t.
\end{equation}
Let us now set 
$\alpha:=-\partial_t u+H(x,Du)\in L^\infty_{1/\weight}(Q)$ we obtain
\begin{equation}
  \label{eq:152}
  \mathcal C(u)=
  \sup_{m\in L^p\cap L^1_\weight(Q),\ m\ge 0}       
  \int_Q \Big(\alpha m-F(x,m)\Big)\,\d x\,\d t.
\end{equation}
Since $\alpha m-F(x,m)\le F^*(x,\alpha)$ it is immediate to see that
\begin{equation}
  \mathcal C(u)\le \int_Q F^*(x,\alpha)\,\d x\,\d t.
\end{equation}
On the other hand, restricting the supremum in \eqref{eq:152} to
functions
$m$ vanishing outside a rectangle $R_k=(0,1)\times B_k$, where
$B_k$ is the ball centered at $0$ of radious $k$ in $\R^d$, and applying the 
\cite[Chap. IX,
Prop. 2.1]{Ekeland-Temam} in $L^p(R_k)$ we obtain
\begin{equation}
    \mathcal C(u)\ge \int_{R_k} F^*(x,\alpha)\,\d x\,\d t.
\end{equation}
Since $F^*$ is nonnegative, a limit as $k\to\infty$ by Beppo Levi
monotone convergence theorem yields
\begin{equation}
  \label{eq:155}
  \mathcal C(u)=\int_Q F^*(x,\alpha)\,\d x\,\d t.
\end{equation}
Eventually, since 
$F^*$ is increasing w.r.t.~the second variable,
we obtain the last identity of \eqref{eq:137}. 
\end{proof}
\nc

\subsection{Existence of a solution to the dual problem}
\label{subsec:existence-dual}
Let us now show that in Problem \ref{p:control_HJ} 
the maximum is reached and therefore it coincides with the supremum of
Problem \ref{p:control_HJ}, when the measures $\mu_i$ have $L^p$
densities.
\GGG
Let us first make a preliminary remark concerning a natural lower
bound for $\alpha$.
\begin{remark}[Lower bound on $\alpha$]
  \label{rem:lower_alpha}
  \upshape
  Since $F^*(x,a)\ge0$ and $F^*(x,a)=0$ if $a\le f(x,0)$, 
  it is not restrictive to assume that any competing pair
  $(u,\alpha)\in \hj_q(Q,H)$
  for the maximization of $\cA$ satisfies the lower bound
  \begin{equation}
    \label{eq:253}
    \alpha(t,x)\ge f(x,0)\quad\text{a.e.~in $Q$}.
  \end{equation}
  In fact, if $(u,\alpha)\in \hj_q(Q,H)$ we can always replace $\alpha$ by $\tilde\alpha:=\alpha\lor f(x,0)$ 
  still obtaining a pair $(u,\tilde\alpha)\in \hj_q(Q,H)$ 
  with $\cA(u,\alpha)=\cA(u,\tilde\alpha)$.
\end{remark}

\begin{theorem}[Solution to the dual problem]
  \label{thm:minimum_dual}
  If $\mu_i=m_i\LL^d\in \cP_{2,p}^r(\R^d)$ 
  then there exists an optimal pair 
  $(u^*,\alpha^*) \in \hj_q(Q,H)$ satisfying \eqref{eq:253} such that 
\[\cA(u^*,\alpha^*) = \max_{(u,\alpha) \hj_q(Q,H)}\cA(u,\alpha).\]
\end{theorem}
\nc
\begin{proof}
\GGG
Let $(u_n,\alpha_n) \in \hj_q(Q,H)$ be a maximizing sequence 
such that 
\begin{displaymath}
  \lim_{n\to\infty}\cA(u_n,\alpha_n)=\sup_{(u,\alpha)\in
    \hj_q(Q,H)}\cA(u,\alpha)=A=
  \min_{(m,\sww)\in \K}\tilde\cB(m,\ww).
\end{displaymath}
\nc
We can assume that $\cA(u_n,\alpha_n)$ is bounded below, say by $A-1$.

\GGG 
By remark \ref{rem:lower_alpha} we can suppose that
$\alpha_n\ge f(x,0)\ge-\gamma_f$.
\eqref{eq:6} also yields
\begin{equation}
\label{eq:158}
\frac 1{q c_f^q}\int_{Q} \big(\alpha_n - \gamma_f \big)_+^q\,\d x
\d t  \leq \int_{Q} F^*(x,\alpha_n)\, \d x \d t\leq 
\int_{\R^d}
u_n^+(0,x)m_0\, \d x
-\int_{\R^d}
u_n^-(1,x)m_1\, \d x +1-A.
\end{equation} 
Using  Corollary \ref{c:duality:u_Lq}, we get   that
there exists a constant $C_1>0$ such that
\begin{equation}
\label{eq:159}
\begin{split}
\int_{\R^d}
u_n^+(0,x)m_0\, \d x
-\int_{\R^d}
u_n^-(1,x)m_1\, \d x &\leq C_1\left( 1 + \| \alpha_n \|_{\cX^q(\R^d)}
\right)
\quad \text{for every }n\in N.
\end{split}
\end{equation}
Combining this information with $\alpha_n\ge -\gamma_f$ and
\eqref{eq:158}
this implies that 
\begin{equation}
\| \alpha_n \|_{\cX^q(Q)}^q  \leq C_2 \left( 1 + \| \alpha_n \|_{\cX^q(Q)} \right)
\end{equation}
whence the uniform bound $\| \alpha_n \|_{\cX^q(Q)} \leq C_3$ 
and consequently $\int_{Q} F^*(x,\alpha_n) \le C_4$, 
for constants $C_i$ independent of $n$.

Since $\cA$ is invariant under translations of $u$ by a constant, 
we can tune $(u_n,\alpha_n)$ so that $\int_{\R^d} u_n^-(1,x) m_1\,\d x
= 0$.
If we recall that $\int_{\R^d} u_n^+(0,x) m_0\,\d x \geq \int_{Q}
F^*(x,\alpha_n) +A-1$,  we deduce the uniform estimate 
\begin{equation}
  \label{eq:160}
  A-1\le \int_{\R^d} u_n^+(0,x) m_0\,\d x\le A+C_4
\end{equation}
\nc
Now, thanks to the uniform bound on $\| \alpha_n \|_{cX^q(Q)}$ there
exists a subsequence $k \mapsto n(k)$ such that $\alpha_{n(k)}
\weakto^* \alpha^*$ in $\cX^q(Q)$.
\GGG
By Theorem \ref{thm:precise_repr} 
we also get for every $[a,b] \subset I$ 
\begin{equation}
\sup_{n \in \N,\ r\in [a,b]} \|\sfu_n(r,\cdot)\|_{\cX^q(\R^d)} < \infty.
\end{equation}
Then, from the stability result contained in Theorem \ref{t:stability}
there exists a subsequence $\{n(k)\}$
and a limit function $u^*$ 
such that 
$(u^*,\alpha^*)\in \hj_q(Q,H)$ 
and $\cA( u^*, \alpha^*)  \geq \limsup_{n \to \infty} \cA( u_{n(k)},
\alpha_{n(k)})=A$.
Thus $(u^*,\alpha^*)$ attains the maximum of $\cA$.
\end{proof} 

\subsection{Optimality conditions and the weak formulation of the 
  Mean Field planning system}
\label{subsec:variational-optimal}
\GGG
Let us first derive the three crucial optimality conditions satisfied
by the optimal solutions of the primal and of the dual problem.
\begin{theorem}[Optimality conditions]
  \label{l:a_leq_B}
Let $\mu_i=m_i\LL^d\in \cP_{2,p}^r(\R^d)$ be given, 
let $(u,\alpha)\in \hj_q(Q,H)$ 
be a  weak subsolution   to \eqref{HJ_alpha} satisfying \eqref{eq:253},
and let $(m,\vv)\in \CE 2p{Q;\mu_0,\mu_1}$ be a weak solution to the continuity
equation connecting $\mu_0$ to $\mu_1$. 
The following properties are equivalent:
\begin{enumerate}[(i)]
\setlength\itemindent{0pt}
\item $\cA(u,\alpha) = \cB(m,\vv)$.
\item $(m,\vv)$ 
is an optimal solution to Problem \ref{p:control_eq:cont},
$(u,\alpha)$ is an optimal solution
to Problem \ref{p:control_HJ}
\item
$(u,\alpha)$ and $(m,\vv)$ satisfy the optimality conditions:
\begin{enumerate}[\rm (O1)]
\item $\alpha=f(x,m)$ 
  $\lambda $-a.e.~in $Q$. 
\item $\vv=-D_pH(x,Du)$ $\tilde \mu$-a.e.~in $Q$.
\item $u_0^+m_0,u_1^-m_1\in L^1(\R^d)$ and
  $u$ is a ``renormalized'' solution to 
  \begin{gather}
    \label{eq:161}
    \left\{
    \begin{aligned}
      &\partial_t(um)+\nabla\cdot (um\,\vv)=\Big(H(x,Du)+Du\cdot
      \vv-\alpha\Big)m&&\text{in }Q,\\
      &(um)_{t=0+}=u_0^+m_0,\quad (um)_{t=1-}=u_1^-m_1&&\text{on
      }\partial Q
    \end{aligned}
    \right.
  \end{gather}
  i.e.~for every $(\zeta,Z)\in \ZZ_c$ 
  \begin{equation}
    \label{eq:205}
    \partial_t(Z(u)m)+\nabla\cdot (Z(u)m\,\vv)=\zeta(u)\Big(H(x,Du)+Du\cdot
    \vv-\alpha\Big)m\quad\text{in }\DD'(Q)
  \end{equation}
  and
  \begin{equation}
    \label{eq:206}
      \aplim_{t\down0}\int_{\R^d}u(t,x)m(t,x)\,\d x=
      \int_{\R^d}u_0^+m_0\,\d x, \ 
    \aplim_{t\up1}\int_{\R^d}u(t,x)m(t,x)\,\d x=
    \int_{\R^d}u_1^-m_1\,\d x.
  \end{equation}
  \end{enumerate}
\end{enumerate}
\end{theorem}
\begin{proof}
The equivalence between (i) and (ii) follows by 
\eqref{eq:99} and Theorem \ref{thm:min-max}.
Corollary
\ref{cor:optimality} yields 
the equivalence between (ii) and (iii).
\end{proof}
Recall that \eqref{eq:205} and \eqref{eq:206} can be equivalently
formulated as 
  \begin{equation}
    \label{eq:161bis}
    \begin{aligned}
      -\int_Q Z(u)m& (\partial_t\varphi+D\varphi\cdot\vv)\,\d x\,\d t
      +\int_{\R^d}Z(u_0^+)m_0\varphi(0,\cdot)\,\d x-
      \int_{\R^d}Z(u_1^-)m_0\varphi(0,\cdot)\,\d x= 
      \\&=\int_Q
      \zeta(u)\Big(H(x,Du)+Du\cdot \vv-\alpha\Big)m\varphi \,\d x\,\d
      t\quad\text{for every $\varphi\in C^1_c(\overline Q)$.}
    \end{aligned}
  \end{equation}

The previous result provides a natural notion of solution 
to the system \eqref{eq:wMFPP}, interpreted as the optimality condition
for the two optimization problems \ref{p:control_eq:cont} and
\ref{p:control_HJ}.

\begin{definition}[Variational weak solutions to the MFPP system]
\label{def:solution}
Let $\mu_i=m_i\LL^d\in \cP_{2,p}^r(\R^d)$ 
be given
and let us suppose that the structural assumption Assumptions
\ref{h.1}
are satisfied.
A pair $(m,u)$ is a weak  solution to the MFPP system \eqref{eq:wMFPP} if
\begin{itemize}
\item[(i)] $m\in L^p\cap L^1_\weight(Q)$ 
  is probability density and
  $u\in L^1_{\loc}(Q)$ is weak subsolution to  
\begin{equation}
-\partial_t u + H(x,Du) \leq f(x,m)\quad\text{in }\DD'(Q).
\end{equation} 
In particular, the left and right traces $u^+_0, u^-_1$ 
at $t=0,1$ are well defined
in the sense of measure convergence of Theorem 
\ref{thm:precise_repr}.
\item[(ii)]
  $m$ belongs to $\rmA_2(Q)$ and it
  is a distributional solution to  
\begin{equation}\label{eqmDu}
\partial_t m  - \nabla\cdot  (m\,D_\spp
H(x,Du)) = 0  \qquad
\text{in }\DD'(Q).
\end{equation}
In particular $m$ has a precise representative $\mu$ in the sense of Lemma
\ref{le:precise-m} and left and right traces at $t=0,1$ satisfying the boundary conditions
$  \mu_{t=0}= \mu_0, \ \mu_{t=1} = \mu_1.
$
\item[(iii)] 
\GGG
$u_0^+m_0,\ u_1^-m_1\in L^1(\R^d)$ and 
$u$ is a ``renormalized'' solution to \eqref{eq:161},
in the sense of \eqref{eq:205}, \eqref{eq:206}.
\end{itemize}
\end{definition}
\begin{remark}\label{r:solution}
\upshape
\GGG Since $m\in \rmA_2(Q)$, it admits a precise representative
$\mu \in \ac^2(I; \cP_2(\R^d))$ so that the boundary conditions 
$\mu\restr{t=i}=\mu_i$ make sense, 
$m |D_\spp H(\cdot,Du)|^2\in L^1_{\rm loc}(Q)$ and 
\eqref{eqmDu} can be formulated in the usual distributional sense, see 
Theorem \ref{l:duality_hj_cont}.
\end{remark}
We can state Theorem \ref{l:a_leq_B} in the following form.
\begin{theorem}[Solutions to MFPP coincide with solutions to the
  primal-dual problem]
\label{thm:main}
Let us assume the structural conditions \ref{h.1} with
$\mu_i=m_i\LL^d\in \cP_{2,p}^r(\R^d)$. 
\begin{enumerate}
\item If $(m,\vv) \in \CE2p{Q;\mu_0,\mu_1}$ is a minimizer for Problem
  \ref{p:control_eq:cont} and $(u,\alpha) \in \hj_q(H,Q)$ is a
  maximizer of Problem \ref{p:control_HJ} also satisfying
  \eqref{eq:253}, then the pair $(m,u)$ is a variational weak solution
  of the system \eqref{eq:wMFPP} and we can identify
  $\alpha = f(\cdot, m)$ $\LL^{d+1}$-a.e.~in $Q$ and
  $\vv = -D_\spp H(\cdot, Du)$ $\tilde\mu$-a.e.~in $Q$.

 \item Conversely, if $(m,u)$ is a variational weak solution to the
  planning problem \eqref{eq:wMFPP} according to Definition
  \ref{def:solution}, then the two pairs $(m, -D_\spp H(\cdot, Du))$ and
  $(u,f(\cdot, m))$ are solutions to Problem \ref{p:control_eq:cont}
  and Problem \ref{p:control_HJ}, respectively.
\end{enumerate}
\end{theorem}

Existence of variational weak solutions  
to the MFPP system \eqref{eq:wMFPP} 
can now be easily obtained by 
the results of Sections \ref{subsec:variational-primal} 
and \ref{subsec:variational-dual}. As for uniqueness, this can only be obtained  for $m$ and for $Du$ (on the support of $m$) by strengthening the convexity assumptions.

\begin{theorem}[Existence and uniqueness of solutions to MFPP]
Under the structural assumptions \ref{h.1}
with $\mu_i=m_i\LL^d\in \cP_{2,p}^r(\R^d)$ there exists a variational weak solution $(u,m)$ 
to the planning  problem \eqref{eq:wMFPP} and
the function
\begin{equation}
  \label{eq:249}
  \alpha:=f(x,m)\quad
  \text{is independent of the choice of $m$,}
\end{equation}
in the sense that if 
$(u,m)$ and
$(u',m')$ are two solutions then $f(\cdot,m)=f(\cdot,m')$ $\LL^{d+1}$
a.e.~in $Q$. Moreover
\begin{enumerate}
  \item
  if $f(x,\cdot)$ is strictly increasing then $m$ 
  is unique (up to $\LL^{d+1}$-negligible sets) and 
  the vector field $\vv=-D_\spp H(x,Du)$ is uniquely determined
  $m\lambda$-a.e.~in $Q$
\item If $H$ is strictly convex (equivalently, if $L$ is
  differentiable
  w.r.t.~$\vv$) then also $Du$ is uniquely determined 
  $m\lambda$-a.e.
\end{enumerate}
\end{theorem}

We conclude this section with a last characterization of optimizers of
Problem \ref{p:control_eq:cont}, where a modified Lagrangian is
involved.
First of all, whenever $\alpha\in \cX^q(Q)$ we set
\begin{equation}
  \label{eq:244}
  L_\alpha(t,x,\vv):=L(x,\vv)+\alpha(t,x).
\end{equation}
\begin{theorem}[Optimizers of the modified Lagrangian cost]
  \label{thm:modified_Lagrangian_cost}
  Let $\mu_i=m_i\LL^d\in \cP_{2,p}^r(\R^d)$ and $(m,\vv)\in \CE
  2p{Q;\mu_0,\mu_1}$.  The pair 
  $(m,\vv)$ is a solution to Problem \ref{p:control_eq:cont} if and
  only if setting $\alpha:=f(x,m)$ the pair
  $(m,\vv)$ minimizes the modified Lagrangian dynamic cost
  \begin{equation}
    \label{eq:245}
    \mathcal L_\alpha(m',\vv'):=
    \int_Q L_\alpha(t,x,\vv')m'\,\d x\,\d t=
    \int_Q \Big(L(x,\vv')+\alpha(t,x)\Big)m'\,\d x\,\d t
  \end{equation}
  among all pairs $(m',\vv')\in \CE 2p{Q;\mu_0,\mu_1}$,
  i.e.
  \begin{equation}
    \label{eq:246}
        \mathcal L_\alpha(m,\vv)\le \mathcal L_\alpha(m',\vv')
        \quad
        \text{for every }(m',\vv')\in \CE 2p{Q;\mu_0,\mu_1}.
  \end{equation}
\end{theorem}
\begin{proof}
  The sufficiency of \eqref{eq:246} is easy:
  since $\alpha=f(x,m)$ 
  we have by Fenchel inequality
  \begin{equation}
    \label{eq:247}
    \alpha m=F(x,m)+F^*(x,\alpha), \quad
    \alpha m'\le F(x,m')+F^*(x,\alpha)\quad
    \text{$\lambda$-a.e.~in $Q$,} 
  \end{equation}
  so
  that \eqref{eq:246} yields
  \begin{align*}
    \mathcal B(m,\vv)
    &=
    \int_Q\Big(L(x,\vv)m+F(x,m)\Big)\,\d\lambda=
      \mathcal L_\alpha(m,\vv)-
      \int_Q F^*(x,\alpha)\,\d\lambda
      \\
    &\topref{eq:246}\le
      \mathcal L_\alpha(m',\vv')-
      \int_Q F^*(x,\alpha)\,\d\lambda
      \topref{eq:247}\le \mathcal B(m',\vv').
  \end{align*}
  In order to prove the converse implication, 
  we use Theorem \ref{thm:minimum_dual} to find a 
  maximizer $(u,\alpha^*)$ of Problem \ref{p:control_HJ};
  since $(m,\vv)$ is a solution to 
  Problem \ref{p:control_eq:cont}, Theorem \ref{thm:main} shows that
  $\alpha^*=f(x,m)=\alpha$.

  \eqref{eq:247} and Theorem \ref{l:duality_hj_cont}(3) then yield
  \begin{align*}
      \mathcal L_\alpha(m,\vv)
    &\topref{eq:247}=\mathcal B(m,\vv)
      + \int_Q F^*(x,\alpha)\,\d\lambda
      =\mathcal A(u,\alpha)+ \int_Q F^*(x,\alpha)\,\d\lambda
    \\&=
        \int_{\R^d}u^+_0\,\d\mu_0-
        \int_{\R^d}u^-_1\,\d\mu_1
        \le 
        \int_Q \Big(\alpha-H(x,Du)-Du\cdot \vv'\Big)m'\,\d\lambda
    \\&\le  \int_Q \Big(\alpha+L(x,\vv')\Big)m'\,\d\lambda=
        \mathcal L_\alpha(m',\vv').
  \end{align*}
\end{proof}
Notice that the function $\alpha$ defining the modified Lagrangian
dynamic cost \eqref{eq:245} 
does not depend on the particular choice of the solution $m$,
thanks to \eqref{eq:249}. The above proof also shows that 
if $(m,\vv)$ is any optimal solution to Problem \ref{p:control_eq:cont}
and $(u,\alpha)$ is any optimal solution to Problem \ref{p:control_HJ}
we have
\begin{equation}
  \label{eq:254}
  \cL_\alpha(m,\vv)=
   \int_{\R^d}u^+_0\,\d\mu_0-
   \int_{\R^d}u^-_1\,\d\mu_1
\end{equation}
so that these quantities do not depend on the particular choices of
$m,\vv$ and $u$ as well.
It is natural to interpret the above results
in terms of an Optimal Transport problem associated to the Lagrangian
\eqref{eq:244}. This will be the aim of the next section.

\section{Optimal plans: a Lagrangian viewpoint}
\label{sec:Lagrangian}

In this section we are concerned with optimality conditions at the level of single agents trajectories.
In particular we are interested in giving a good notion of Nash equilibria for the system.
To do it we look at the problem in a Lagrangian fashion, studying
suitable measures on admissible paths. The crucial step in this
procedure is a probabilistic representation of 
measure-valued solutions to the continuity equation, known as superposition principle.  A similar approach was already used for a  Lagrangian formulation of  mean field games  in \cite{cardaliaguet2016first} borrowing arguments from \cite{ambrosio2009geodesics}, \cite{carlier+al}.

\subsection{Continuity equation and
  measures on the space of continuous curves}
\label{subsec:dynamic-plans}
We introduce the (complete and separable metric) space 
\begin{equation}
  \label{eq:181}
  \Gamma:=C^0([0,1];\R^d)\quad
  \text{endowed with the uniform metric,}
\end{equation} 
and its Borel (in fact $F_\sigma$) subset 
of absolutely continuous curves 
$\ac_2([0,1];\R^d)$;
for every $\gamma\in \ac_2([0,1];\R^d)$ we call $\rmE_2$ the lower
semicontinuous
energy functional
\begin{equation}
  \label{eq:212}
  \rmE_2(\gamma):=\int_0^1|\dot \gamma(t)|^2\,\d t,
  \quad\text{with}\quad
  \rmE_2(\gamma):=+\infty\text{ if }\gamma\not\in\ac_2([0,1];\R^d).
\end{equation}
We denote by $\sfe_t:\Gamma\to \R^d$ and $\sfe:[0,1]\times \Gamma\to Q$ the
evaluation maps
\begin{equation}
  \label{eq:49}
  \sfe_t(\gamma):=\gamma(t),\quad
  \sfe(t,\gamma):=(t,\gamma(t))\quad
  \text{for every }t\in [0,1],\ \gamma\in \Gamma,
\end{equation}
and by $\sfd:(0,1)\times \ac_2([0,1];\R^d)\to (0,1)\times \R^d\times
\R^d$ 
the Borel map
\begin{equation}
  \label{eq:242}
  \sfd(t,\gamma):=
  \begin{cases}
    (t,\gamma(t),\dot\gamma(t))&\text{if $\gamma$ is differentiable at
      $t$},\\
    (t,\gamma(t),0)&\text{otherwise.}
  \end{cases}
\end{equation}
Every measure $\eeta\in \cP(\Gamma)$ 
induces a continuous curve $\mu \in \rmC^0([0,1];\cP(\R^d))$ 
and a probability measure $\tilde \mu\in \cP(Q)$ by push-forward
\begin{equation}
  \label{eq:179}
  \mu_t:=(\sfe_t)_\sharp \eeta,\quad
  \tilde\mu=\sfe_\sharp\eeta.
\end{equation}
If $\eeta$ is concentrated on $\ac^2([0,1];\R^d)$ 
(i.e.~$\eeta\big(\Gamma\setminus \ac^2([0,1];\R^d)\big)=0$)
and
\begin{equation}
  \label{eq:211}
  \int_\Gamma E_2(\gamma)\,\d\eeta(\gamma)<\infty,
\end{equation}
then $\mu\in \ac_2([0,1];\cP_2(\R^d))$ and it solves the continuity
equation
\begin{equation}
  \label{eq:213}
  \partial_t \mu_t+\nabla \cdot(\mu\vv)=0
\end{equation}
for the vector field $\vv$ which is barycenter 
of the measure
$\tilde \nu=\sfd_\sharp(\eeta)$ w.r.t.~$\tilde\mu$.
In fact, $\tilde\nu$ can be disintegrated w.r.t.~its marginal
$\tilde\mu$
by $\tilde\nu=\int_Q \nu_{t,x}\,\d\tilde\mu(t,x)$ and 
\begin{equation}
  \label{eq:214}
  \vv(t,x)=\int_{\R^d}v\,\d\nu_{t,x}(v).
\end{equation}
Since
\begin{equation}
  \label{eq:215}
  \int_\Gamma E_2(\gamma)\,\d\eeta(\gamma)
  =\int_\Gamma \int_0^1|\dot\gamma(t)|^2\,\d t\,\d\eeta(\gamma)
  =\int_{(0,1)\times \R^d\times \R^d}
  |v|^2\,\d\tilde\nu(t,x,v)=
  \int_{Q}\Big(\int_{\R^d}|v|^2\,\d\nu_{t,x}\Big)\,\d\tilde\mu(t,x)
\end{equation}
one immediately has by Jensen's inequality that 
\begin{equation}
  \label{eq:216}
  |\vv(t,x)|^2\le 
  \int_{\R^d}|v|^2\,\d\nu_{t,x}(v),\quad
  \int_{Q}|\vv|^2\,\d\tilde\mu\le 
  \int_{Q}\Big(\int_{\R^d}|v|^2\,\d\nu_{t,x}\Big)\,\d\tilde\mu(t,x)\le 
  \int_\Gamma E_2(\gamma)\,\d\eeta(\gamma).
\end{equation}
and similarly, for convex Lagrangians,
\begin{equation}
  \label{eq:217}
  \int_{Q}L(x,\vv)\,\d\tilde\mu\le 
  \int_{Q}\Big(\int_{\R^d}L(x,v)\,\d\nu_{t,x}\Big)\,\d\tilde\mu(t,x)\le 
  \int_\Gamma L(\gamma(t),\dot\gamma(t))\,\d\eeta(\gamma).
\end{equation}
The following result \cite[Section 8.2]{ambrosio2008gradient}
shows that 
any solution to the continuity equation \eqref{eq:213}
admits the representation \eqref{eq:179} for a measure $\eeta$ 
tighten to $\mu$ so that \eqref{eq:216} and \eqref{eq:217} become in
fact an
identity.
\begin{theorem}[Superposition principle]
  \label{thm:superposition_principle}
  If $\mu\in \ac^2([0,1];\cP_2(\R^d))$ 
  is a solution to the continuity
  equation
  \eqref{cont_eq} with respect to $\vv\in
L^2(Q,\tilde\mu;\R^d)$, then
there exists a measure 
$\eeta\in \cP(\Gamma)$ concentrated on the set of curves 
$\gamma\in \ac^2([0,1];\R^d)$ 
which are integral solutions to the ODE's
\begin{equation}
  \label{eq:46}
  \dot\gamma(t)=\vv(t,\gamma(t))\quad\text{$\LL^1$-a.e.~in }(0,1),
\end{equation}
such that $(\sfe_t)_\sharp \eeta=\mu_t$ for every $t\in [0,1]$ and
$\sfe_\sharp\eeta=\tilde\mu$.
\end{theorem}
In the following we call $\cP_{2,p}(\Gamma)$ 
the collection of measures $\eeta\in \cP(\Gamma)$ such that
\begin{equation}
  \int_\Gamma\rmE_2(\gamma)\,\d\eeta<\infty,\quad
  \sfe_\sharp \eeta=m\lambda\quad
  \text{for }m\in L^p(Q),
\end{equation}
and we say that $\eeta$ is \emph{tightened} to $(m,\vv)\in \CE 2p{Q}$
if $\eeta$ satisfies the conditions of Theorem
\ref{thm:superposition_principle}.
We want to study the variational properties of dynamic plans $\eeta$
tightened to minimizers of Problem \ref{p:control_eq:cont}.

\subsection{Lifting functions to $\Gamma$ and regularization of
  dynamic plans}
\label{subsec:lifting}
Let $w\in L^0(Q)$ and $\sfw$ be a Borel representative of $w$.
If $\eeta\in \cP_{2,p}(\Gamma)$ the measure
$\tilde\mu=\sfe_\sharp\eeta$
is absolutely continuous w.r.t.~$\LL^{d+1}$ so that 
the function
\begin{equation}
  \label{eq:218}
  w_\Gamma=\sfw\circ\sfe :(0,1)\times \Gamma\to\R,
  \quad
  w_\Gamma(t,\gamma):=\sfw(t,\gamma(t))
\end{equation}
is well defined and its equivalence class modulo
$\tilde\eeta$-negligible sets in $I\times \Gamma$ does not depend
on the particular $\lambda$-representative of $w$.
Similarly, we can consider the function
$\ell(t,x,y):=L(x,y)$
\begin{equation}
  \label{eq:220}
  L_\Gamma:=\ell\circ \sfd:(0,1)\times\Gamma\to\R,\quad
  L_\Gamma(t,\gamma):=L(\gamma(t),\dot\gamma(t)).
\end{equation}

\begin{lemma}\label{l:energy_inequality_plan}
Let $u$ be a weak subsolution to $\hj(H,\alpha)$ 
with (right continuous) representative $\sfu$ and $\alpha \in
L^q(Q)$ and let 
$\eeta \in \cP_{2,p}(\Gamma)$ be a given plan, with
$\mu_t:= (\sfe_t)_\# \eta$. 
If 
\begin{equation}
  \label{eq:219}
  B_\Gamma(t,\gamma):=
  \int_0^t \Big(L(\gamma(r),\dot\gamma(r))+\alpha(r,\gamma(r))\Big)\,\d r
  =\int_0^t \Big(L_\Gamma(r,\gamma)+\alpha_\Gamma(r,\gamma)\Big)\,\d r
\end{equation}
we can find a $\eeta$-negligible set $N\subset \Gamma$ 
and a Borel representative $u_\star$ of $u_\Gamma$ such that
for every $0 \leq s \leq t \leq 1$
 \begin{equation}
   \label{eq:223}
   u_\star(s,\gamma)+B_\Gamma(s,\gamma)\le 
   u_\star(t,\gamma)+B_\Gamma(t,\gamma)
   \quad\text{for every }\gamma\in \Gamma\setminus N.
 \end{equation}
 Morever, there exists a countable set $J$ such that 
 for every $t\in D_p[\mu]\setminus J$
 \begin{equation}
   \label{eq:225}
    u_\star(t,\gamma)= u_\Gamma(t,\gamma)=\sfu(t,\gamma(t))
    \quad \text{for $\eeta$-a.e.~$\gamma\in \Gamma$}
  \end{equation}
  and \eqref{eq:223} also holds at $s=0$ and $t=1$ with the traces of
  $\sfu$:
  \begin{equation}
    \label{eq:256}
    u^+_0(\gamma(0))-u^-_1(\gamma(1))\le 
    \int_0^1
    \Big(L(\gamma(r),\dot\gamma(r))+\alpha(\gamma(r),\dot\gamma(r))\Big)\,\d
    r=B_\Gamma(1,\gamma)
    \quad\text{for $\eeta$-a.e.~$\gamma$.}
  \end{equation}
\end{lemma}
\begin{proof}
  Let us choose a nonnegative 
  $\varphi\in C_b(\Gamma)$ such that 
  $M(\varphi):=\int_\Gamma\varphi\,\d\eeta>0$.
  The corresponding measure $\eeta^\varphi:=M(\varphi)^{-1}\varphi\eeta$
  induces a curve $\mu^\varphi$ with density $m^\varphi\in L^p(Q)$ 
  and a vector field $\vv^\varphi$ 
as in \eqref{eq:179} and \eqref{eq:214} 
such that $(m^\varphi,\vv^\varphi)\in \CE2pQ$.
From \eqref{duality_formula} of 
Theorem \ref{l:duality_hj_cont}  and Fenchel inequality we get
for every $s,t\in D_p[\mu]$ (with the obvious modifications 
for $s=0$ and $t=1$)
\begin{equation}
\begin{split}
  \int_\Gamma \sfu(s,\gamma(s))\,\d\eeta^\varphi=
  \int_{\R^d} \sfu_s \,\d\mu^\varphi_s &\leq  \int_{\R^d}
  \sfu^\varphi_t
  \,\d\mu^\varphi_t + \int_s^t \int_{\R^d}  \Big[ L(x,\vv^\varphi) +
    \alpha(r, x) \Big] 
  \,\d\tilde \mu^\varphi \\
  &\leq  \int_\Gamma\sfu(t,\gamma(t)) \,\d\eeta^\varphi + 
  \int_{\Gamma}\int_s^t \Big[ L(\gamma(r), \dot{\gamma}(r))  + 
    \alpha(r, \gamma(r))\Big] \,\d r \, \d \eeta^\varphi(\gamma),
\end{split}
\end{equation}
and therefore 
\begin{equation}
  \label{eq:224}
  \int_\Gamma \Big(\sfu_\Gamma(s,\gamma)+B_\Gamma(s,\gamma)\Big)
   \varphi(\gamma)\,\d\eeta\le 
   \int_\Gamma \Big(\sfu_\Gamma(t,\gamma)+B_\Gamma(t,\gamma)\Big)
   \varphi(\gamma)\,\d\eeta.
\end{equation}
Since $\varphi$ is arbitrary we deduce that 
\begin{equation}
  \label{eq:243}
  \sfu_\Gamma(s,\cdot)+B_\Gamma(s,\cdot)\le 
  \sfu_\Gamma(t,\cdot)+B_\Gamma(t,\cdot)\quad
  \text{$\eeta$-a.e.~in $\Gamma$}.
\end{equation}
The implication $(ii)\Rightarrow(iii)$ of Lemma
\ref{le:increasing-main} yields 
\eqref{eq:223}; property (I.2) of the same Lemma yields \eqref{eq:225}
and \eqref{eq:256}. 
\end{proof}
We want now to extend the previous Lemma to arbitrary dynamic plans
in $\cP(\Gamma)$ concentrated on $\ac^2([0,1];\R^d)$ 
with initial and final marginals
$\mu_i=(\sfe_i)_\sharp\eeta\ll\LL^d$.
We will slightly reinforce the structural assumptions \ref{h.1} by
supposing that 
\begin{equation}
  \label{eq:268}
  f,L\text{ are nonnegative.}\quad
  L\text{ is continuous in $\R^d \times \R^d$}.
\end{equation}
Following the approach of \cite{ambrosio2009geodesics} (see also
\cite{cardaliaguet2016first}),
for every Borel function $\alpha\in L^q(Q)$
we introduce a
precise representative
of $\alpha$, obtained by convolving $\alpha$ with the (symmetric) Heat kernel 
$g_\eps$ \eqref{eq:190}. To simplify some aspects concerning measurability, we 
use here a definition based on a given sequence $(\eps(n))_{n\in
  \N}\subset (0,1]$ such that $\eps(n)\downarrow0$ as $n\to\infty$ 
(we may choose, e.g.~$\eps(n):=2^{-n}$).
\begin{equation}
\label{eq:274}
\hat \alpha(t,x) := \limsup_{n\to\infty} 
\alpha_n(t,x),
\quad
\alpha_n
(t,x):=(\alpha(t,\cdot)\ast g_{\eps(n)})(x)= \int_{\R^d} \alpha(x+ \varepsilon(n) y)
g_1(y)\,\d y.
\end{equation}  
Notice that $\alpha=\hat\alpha$ $\lambda$-a.e.~in $\R^d$;
moreover, if $\alpha'=\alpha$ $\lambda$-a.e.~in $Q$, then 
there exists a $\LL^1$-negligible set $N\subset (0,1)$ such that 
$\hat \alpha'(t,x)=\hat\alpha(t,x)$ for every $t\in (0,1)\setminus N$
and
$x\in \R^d$. 
In particular, the function $\hat\alpha_\Gamma=\hat\alpha\circ\sfe$ 
is well defined for every $\eeta\in \cP_2(\Gamma)$, since
the first marginal of $\sfe_\sharp\eeta$ is $\LL^1$,
and $\hat\alpha'(\cdot,\gamma(\cdot))=\hat\alpha(\cdot,\gamma(\cdot))$ 
$\LL^1$-a.e.~in $(0,1)$ for $\eeta$-a.e.~$\gamma$.
Notice that the composition 
\eqref{eq:218} for arbitrary $\alpha\in L^q(Q)$ 
is well defined only for $\eeta\in \cP_{2,p}(\Gamma)$:
the choice of the precise representative $\hat\alpha$ extends this
operation to arbitrary $\eeta\in \cP_2(\Gamma)$.

We will also consider the corresponding maximal function (depending on
the sequence $\eps(n)$)
\begin{equation}
\label{eq:275}
\hat M\alpha(t,x):= \sup_{n\in \N} 
\int_{\R^d} |\alpha(x+ \varepsilon(n) y)|
g_1(y)\,\d y
=\sup_{n\in \N} 
\int_{\R^d} |\alpha(z)|
g_{\eps(n)}(x-z)\,\d z.
\end{equation}
Since clearly $\hat M\alpha$ is bounded by the maximal function
$M\alpha$, i.e.
\begin{equation}
  \label{eq:276}
  \hat M\alpha(t,x)\le M\alpha(t,x):= \sup_{r>0} 
  \int_{\R^d} |\alpha(x+ r y)|
  g_1(y)\,\d y,
\end{equation}
as noticed in \cite[p.~456]{ambrosio2009geodesics}, 
we have
\begin{equation}
  \label{eq:248}
  \hat M\alpha(t,x)=\hat M\hat \alpha(t,x),\quad
  \alpha_n(t,x)\le \hat M\alpha(t,x)\ \text{if }n\in \N,\quad
  \hat M\alpha\in L^q(Q).
\end{equation}
Let us fix a few measurability properties of $\hat \alpha$ and
$\hat M\alpha$ which will turn to be useful in the following.
Recall that a Borel function $g:Q\to [0,+\infty]$ is a normal integrand
if it is lower semicontinuous 
w.r.t.~$x$ for $\LL^1$-a.e.~$t\in
(0,1)$
(see \cite[Chap.~VIII, Def.~1.1]{Ekeland-Temam}).
\begin{lemma}
  \label{le:measurability}
  Let $\alpha:Q\to\R$ be a Borel map in $L^q(Q)$ and let $\hat\alpha$,
  $\hat M\alpha$ be defined as in \eqref{eq:274} and \eqref{eq:275}.
  \begin{enumerate}[\rm 1.]
  \item 
    $\hat M\alpha$ is a Borel and normal positive integrand
  so that 
  \begin{gather}
    \label{eq:278}
    \text{the functional}\quad
    \gamma\mapsto \cM_\alpha[\gamma]:=\int_0^1 \hat
    M\alpha(t,\gamma(t))\,\d t \quad
    \text{is lower semicontinuous in $\Gamma$},\\
    \label{eq:279}
    \rmD(\cM_\alpha):=\Big\{\gamma\in \Gamma:
    \cM_\alpha[\gamma]<\infty\Big\}\quad
    \text{is a $F_\sigma$, thus Borel, subset of $\Gamma$}.
  \end{gather}
  \item
    $\hat \alpha$ is a Borel map with values in $[0,+\infty]$, 
  the composition $t\mapsto \hat\alpha(t,\gamma(t))$ is a Borel curve
  for every $\gamma\in \Gamma$, 
  and the map $\hat\alpha_\Gamma:\gamma\to
  \hat\alpha(\cdot,\gamma(\cdot))$ is Borel
  from $\Gamma$ to $L^0((0,1);[0,+\infty])$.
\item
  The integral functional $\cI_\alpha:\Gamma\to[0,+\infty]$ 
  \begin{equation}
    \label{eq:280}
    \cI_\alpha[\gamma]:=\int_0^1 \hat\alpha(t,\gamma(t))\, \d t
  \end{equation}
  is a Borel map from $\Gamma$ to $[0,+\infty]$.
  \end{enumerate}
\end{lemma}
\begin{proof}
  Since we have chosen a Borel representative $\alpha:Q\to \R$,
  for every $n\in \N$
  the functions $\alpha_n:Q\to [0,+\infty]$ are Borel and also Caratheodory integrands,
  since they are finite and continuous w.r.t.~$x$ for a.e.~$t\in
  (0,1)$; moreover they satisfy
  \begin{equation}
    \label{eq:277}
    \sup_{x\in \R^n}|\alpha_n(t,x)|\le c_n
    \Big(\int_{\R^d}|\alpha(t,x)|^q\Big)^{1/q},\quad
    \int_0^1\Big(\sup_{x\in \R^n}|\alpha_n(t,x)|\Big)^q\,\d t\le c_n \|\alpha\|_{L^q(Q)}^q
  \end{equation}
  for some constant $c_n>0$.
  It follows that for every $\gamma\in \Gamma$ the curve $t\mapsto
  \alpha_n(t,\gamma(t))$ is Borel and the map $\gamma\mapsto
  \alpha_n(\cdot,\gamma(\cdot))$ is continuous from $\Gamma$ to
  $L^q(0,1)$ and a fortiori to $L^0(0,1)$.
  
  We deduce that $\hat M\alpha$ is Borel and normal, so that
  \eqref{eq:278} follows by \cite[Chap.~VIII,
  Prop.~1.4]{Ekeland-Temam}.
  
  Concerning the second point, it easily follows since setting
  $\alpha_{n,m}:=\sup_{m\le k\le n}\alpha_k$,
  $\hat\alpha_m:=\lim_{n\to\infty}\alpha_{m,n}$, 
  we have
  $\hat\alpha(t,x):=
  \lim_{m\to\infty}\hat\alpha_m$  
  which is therefore pointwise limit of Borel maps.
  The same structure holds for the corresponding composition maps
  from the (complete and separable metric space) 
  $\Gamma$ to the (complete and separable metric space) 
  $L^0((0,1);[0,+\infty])$, so that 
  the map $\hat\alpha_\Gamma:\gamma\to
  \hat\alpha(\cdot,\gamma(\cdot))$ is Borel
  from $\Gamma$ to $L^0((0,1);[0,+\infty])$. 
  
  It follows that the integral functional $\cI_\alpha$ 
  is a Borel map from $\Gamma$ to $[0,+\infty]$,
  being the composition of a lower semicontinuous map (the integral,
  from $L^0((0,1);[0,+\infty]))$ to $\R$) with a Borel map.
\end{proof}
If $\eeta\in \cP_{2,p}(\Gamma)$ then the function
$(\hat M\alpha)\circ \sfe$ belongs to $L^1(I\times\Gamma,\tilde\eeta)$.
Since also $\hat M\alpha(t,\cdot)$ is everywhere defined in $\R^d$ 
for $\LL^1$-a.e.~$t\in (0,1)$, 
we will more generally consider dynamic plans $\eeta\in \cP_2(\Gamma)$
concentrated on $\rmD(\cM_\alpha)$, so that 
$ \int_0^1 (\hat M\alpha)(t,\gamma(t))\,\d t=
\int_0^1 (\hat M\alpha)_\Gamma(t,\gamma)\,\d t<\infty $
for $\eeta$-a.e.~$\gamma$.

Let us introduce the sets
\begin{equation}
  \label{eq:270}
\begin{aligned}
  \Gamma_{\alpha,k}:={}&\Big\{\gamma\in \ac^2([0,1];\R^d):
  \rmE_2[\gamma]+\cM_\alpha[\gamma]\le k
    \Big\},\quad k\in \N,\\
    \Gamma_\alpha:=&
  \Big\{\gamma\in \ac^2([0,1];\R^d):
    \cM_\alpha[\gamma]<+\infty  \Big\}
                     ={}\bigcup_{k\in\N}\Gamma_{\alpha,k}.
\end{aligned}
\end{equation}
\begin{lemma}[Trace inequality along plans in $\cP_2(\Gamma)$]
  \label{le:regularization}
  Let $u$ be a weak subsolution to $\hj(H,\alpha)$ 
  with nonnegative
  $\alpha\in L^q(Q)$ and 
  traces $u^+_0,u^-_1\in L^0(\R^d,\bar \R)$ 
  and let $\eeta\in \cP(\Gamma)$ be a dynamic plan satisfying
  \begin{equation}
    \label{eq:257}
    \mu_i=(\sfe_i)_\sharp\eeta\ll\LL^d,\ i=0,1;\quad
    \gamma\in \Gamma_\alpha
    \quad
    \text{for $\eeta$-a.e.~$\gamma$.}
  \end{equation}
  Then \eqref{eq:256} holds.
\end{lemma}
\begin{proof}
  Let us consider two Borel representatives of $u_0^+,u_1^-$ and
  set $U_{0k}:=\{x\in \R^d:-k\le u^+_0(x)\le k\}$ and
  $U_{1k}:=\{x\in \R^d:-k\le u^-_1(x)\le k\}$; since $u_0^+$ takes
  values
  in $\R\cup\{-\infty\}$ and $u^-_1$ takes values in
  $\R\cup\{+\infty\}$, we have
  \begin{displaymath}
    U_0:=\bigcup_{k\in \N}U_{0k}=\{x\in \R^d:u_0^+>-\infty\},\quad
    U_1:=\bigcup_{k\in \N}U_{1k}=\{x\in \R^d:u_1^-<+\infty\}.
  \end{displaymath}
  We introduce the truncations   $T_k(r):=-k\lor r\land k$
  and the corresponding maps $u_{0k}:=T_k(u_0^+)$ and 
  $u_{1k}:=T_k(u_1^-)$. 
  We set
  \begin{equation}
U_{ik}^r:=\Big\{x\in U_k:
\lim_{n\to\infty}
\int_{\R^d}|u_{ik}(x+\eps(n) y)-u_{i}(x)| g_1(y)\,\d
  y=0\Big\}.\label{eq:271}
  \end{equation}
  Since $u_{ik}\in L^\infty(\R^d)$ we know that
  $ \lim_{n\to\infty}
  \int_{\R^d}|u_{ik}(x+\eps(n) y)-u_{ik}(x)| g_1(y)\,\d
  y=0$ for $\LL^d$-a.e.~point of $\R^d$, so that 
  $\LL^d(U_{ik}\setminus U_{ik}^r)=0$ for every $k\in \R^d$. 
  We eventually set $U_i^r:=\bigcup_{k\in \N} U_{ik}^r$ and notice that 
  \begin{equation}
    \label{eq:269}
    \eeta\Big(\big\{\gamma\in \Gamma:
    \gamma(i)\in U_i\setminus U_i^r,\ i=0,1\big\}\Big)=0
  \end{equation}
  since $\mu_i\ll\LL^d$.
  For every $k\in \N$ let us introduce the sets
  \begin{align}
    \label{eq:283}
    \Sigma_k:={}&\Big\{\gamma\in \Gamma_{\alpha,k}:
    \gamma(i)\in U_{ik}^r,\ i=0,1\Big\}.
  \end{align}
  Since \eqref{eq:256} is trivially satisfied if $\gamma(0)\in
  \R^d\setminus U_0$ or $\gamma(1)\in \R^d\setminus U_1$, 
  it is sufficient to prove \eqref{eq:256} in the case when
  $\eeta$ is concentrated in the set $\Sigma_k.$ 
  We consider the cartesian product $X:=\Gamma\times \R^d$ endowed
  with the
  Probability measure $\ssigma:=\eeta\otimes (g_1\LL^d)$ and the Borel map
  $R_n:\Gamma\times \R^d\to \Gamma$,
  $R_n(\gamma,y):= \gamma+\eps(n) y$.
  Setting $\eeta_n:=(R_n)_\sharp \ssigma$
  it is easy to check that 
  $\eeta_n$ is concentrated on $\ac^2([0,1];\R^d)$ with
  \begin{displaymath}
    \int \rmE_2[\gamma]\,\d\eeta_n=
    \int \rmE_2[R_n(\gamma,y)]\,\d\ssigma(\gamma,y)=
    \int \rmE_2[\gamma]\,\d\ssigma(\gamma,y)=
    \int \rmE_2[\gamma]\,\d\eeta(\gamma),
  \end{displaymath}
  since translations by vectors $\eps(n) y$ do not modify the energy of a
  curve.
  On the other hand, setting $\mu_t:=(\sfe_t)_\sharp\eeta$, we have
  \begin{equation}
    \label{eq:259}
    (\sfe_t)_\sharp\eeta_n=(\mu_t)\ast g_{\eps(n)}\quad\text{for every
    }t\in [0,1],    
  \end{equation}
  since for every $\phi\in C_b(\R^d)$
  \begin{align*}
    \int_{\R^d}\phi(x)\,\d (\sfe_t)_\sharp\eeta_n
    &=
      \int_{\Gamma}\phi(\gamma(t))
      \,\d \eeta_n(\gamma)
      =\int_{\Gamma\times \R^d}\phi(\gamma(t)+\eps(n) y)g_1(y)\,
      \,\d \eeta(\gamma)\, \d y
     \\ &=\int_{\R^d}\Big(\int_{\R^d}\phi(x+\eps(n) y)\,\d\mu_t(x)\Big)g_1(y)
          \,\d y=
          \int_{\R^d} \phi\,\d(\mu_t\ast g_{\eps(n)}).
  \end{align*}
  \eqref{eq:194} shows that $\eeta_n\in \cP_{2,p}(\R^d)$. 
  Applying \eqref{eq:256} to the truncated subsolutions $u_k:=T_k(u)$
  and observing that 
  $u_k$ is a subsolution to $-\partial_t u_k+H(x,Du_k)\le \alpha$, 
\begin{equation}	  
  \begin{split}
    T_k\big(u^+_0(\gamma(0)+\eps(n) y)\big) &-
    T_k\big(u^-_1(\gamma(1)+\eps(n) y)\big) \\
    &\le 
    \int_0^1
    \Big(L(\gamma(r)+\eps(n) y,\dot\gamma(r))+\alpha(r,\gamma(r)+\eps(n) y)\Big)\,\d
    r
    \quad\text{for $\ssigma$-a.e.~$(\gamma,y)$,}
  \end{split}
  \end{equation}
  A further integration w.r.t.~$y$ yields 
  that there exists a Borel set $\Sigma_\star\subset \Sigma_k$ such that 
  $\eeta(\Sigma_k\setminus\Sigma_\star)=0$ and 
  \begin{equation}
    \label{eq:261}
    \begin{aligned}
      \int_{\R^d}T_k\big(u^+_0(\gamma(0)+\eps(n) y)\big)\,&g_1(y)\,\d y -
      \int_{\R^d}T_k\big(u^-_1(\gamma(1)+\eps(n) y)\big)g_1(y)\,\d y
      \\&\le
      \int_{\R^d}\Big(\int_0^1
      \Big(L(\gamma(r)+\eps(n) y,\dot\gamma(r))+\alpha(r,\gamma(r)+\eps(n)
      y)\Big)\,\d r\Big)g_1(y)\,\d y,
    \end{aligned}
  \end{equation}
  for every $n\in \N$ 
  and $\gamma\in \Sigma_\star$.
  We can then pass to the limit as $n\to\infty$: 
  the structural bounds \eqref{eq:9} (with $c_H^+=0$) yield
  \begin{align*}
    0\le \int_0^1
    L(\gamma(r)+\eps(n) y,\dot\gamma(r))\,\d r\le 
    c_H^-(1+2|y|^2)+
    2c_H^-\int_0^1|\gamma(r)|^2\,\d r+\frac {c_H}2\rmE_2[\gamma]
  \end{align*}
  so that Lebesgue Dominated Convergence theorem and the continuity of
  $L$ yield
  \begin{equation}
    \label{eq:264}
    \lim_{n\to\infty}
    \int_{\R^d}\Big(\int_0^1
    L(\gamma(r)+\eps(n) y,\dot\gamma(r))\,\d
    r\Big)g_1(y)\,\d y=
    \int_0^1
    L(\gamma(r),\dot\gamma(r))\,\d
    r.
  \end{equation}
  Similarly
  we get
  \begin{equation}
    \label{eq:266}
    \int_{\R^d}\Big(\int_0^1
    \alpha(r,\gamma(r)+\eps(n) y)\,\d
    r\Big)g_1(y)\,\d y=
    \int_0^1 \alpha_{n}(r,\gamma(r))\,\d r,\quad
    \alpha_{n}:=\alpha(t,\cdot)\ast g_{\eps(n)};
  \end{equation}
  since $0\le \alpha_n\le \hat M\alpha$ 
  and $\int_0^1 \hat M\alpha(t,\gamma(t))\,\d t\le k$ by assumption, we
  conclude that
  \begin{equation}
    \label{eq:267}
    \limsup_{n\to\infty}
    \int_{\R^d}\Big(\int_0^1
    \alpha(r,\gamma(r)+\eps(n) y)\,\d
    r\Big)g_1(y)\,\d y\le 
    \int_0^1 \hat\alpha(r,\gamma(r))\,\d r.
  \end{equation}
  Combining \eqref{eq:264} and \eqref{eq:267} and the fact that 
  $\gamma(i)\in U_{ik}^r$ we eventually get
  \begin{equation}
    \label{eq:261bis}
    u^+_0(\gamma(0))-
    u^-_1(\gamma(1))\le 
    \int_0^1
    \Big(L(\gamma(r),\dot\gamma(r))+\hat\alpha(r,\gamma(r))\Big)\,\d
    r\quad\text{for every }\gamma\in \Sigma_\star.
  \end{equation}
\end{proof}

\subsection{Modified Lagrangian cost and optimal dynamic plans}
\label{subsec:optimal_dynamic_plans}
We can now introduce a cost function associated to the modified
Lagrangian 
\eqref{eq:244}
for any nonnegative Borel map $\alpha\in L^q(Q)$. 
Since we want to compute
$\alpha$ 
along dynamic plans $\eeta'\in \cP(\Gamma)$ which may not have 
absolutely continuous marginals $\mu_t'=(\sfe_t)_\sharp\eeta'$, 
we will add an extra summability condition in terms of the maximal
function $\hat M\alpha$, as in  \cite{ambrosio2009geodesics}, and we
replace $L_\alpha$ by
\begin{equation}
  \label{eq:244bis}
  L_{\hat\alpha}(t,x,\vv):=L(x,\vv)+\hat\alpha(t,x)\quad\text{with}\quad
  \rmL_{\hat\alpha}[\gamma]:=\int_0^1
  L_{\hat\alpha}(t,\gamma(t),\dot\gamma(t))\,\d t
  \quad\text{for every }\gamma\in \Gamma_\alpha,
\end{equation}
where $\Gamma_\alpha$ is defined by \eqref{eq:270}.
\begin{definition}[Modified Lagrangian cost]
  For every nonnegative Borel $\alpha\in L^q(Q)$ 
  let $\hat\alpha$, $\hat M\alpha$, and $\Gamma_\alpha$ be defined as 
  in \eqref{eq:274}, \eqref{eq:275} and \eqref{eq:270}.
  We consider the set
    \begin{equation}
    X_{\alpha}:=(\sfe_0,\sfe_1)(\Gamma_\alpha)=
    \big\{(\gamma(0),\gamma(1)): 
    \gamma\in \Gamma_\alpha\big\}.
\label{eq:286}
\end{equation}
The Lagrangian cost
  $c_\alpha:\R^d\times\R^d\to [0,+\infty]$ is defined by
  \begin{equation}
    \label{eq:250}
    c_{\hat\alpha}(x_0,x_1):=
      \inf\Big\{
    \int_0^1 L_{\hat\alpha}(t,\gamma(t),\dot\gamma(t))\,\d t
       :
       \gamma\in \Gamma_\alpha, 
       \ 
    \gamma(0)=x_0,\ \gamma(1)=x_1\Big\}\quad
    \text{if }(x_0,x_1)\in X_{\alpha},
  \end{equation}
with the usual convention to set $c_{\hat\alpha}(x_0,x_1):=+\infty$ 
if $(x_0,x_1)\not\in X_\alpha$.
\end{definition}
\newcommand{\oomega}{\boldsymbol \omega}
Let us preliminary study the measurability properties of $c_{\hat\alpha}$.
\begin{lemma}
  \label{le:c-measurability}
  $X_\alpha$ is a $F_\sigma$ set in $\R^d\times \R^d$.
  For every $\mmu\in \cP(\R^d\times\R^d)$ 
  $c_{\hat\alpha}$ is $\mmu$-measurable (equivalently, $c_{\hat\alpha}$ is
  universally measurable). 
  Moreover, for every $\eps>0$ 
  there exists a $\mmu$-measurable map $\oomega:X_{\hat\alpha}\to 
  \Gamma$ such that for every $(x_0,x_1)\in X_{\hat\alpha}$ the curve
  $\gamma=\oomega(x_0,x_1)$ satisfies
  \begin{equation}
    \label{eq:287}
    \gamma\in \Gamma_{\hat\alpha},\quad
    \gamma(i)=x_i,\quad
    \int_0^1L_{\hat\alpha}(t,\gamma(t),\dot\gamma(t))\,\d t
    \le c_{\hat\alpha}(x_0,x_1)+\eps.
  \end{equation} 
\end{lemma}
\begin{proof}
  We can write 
  \begin{equation}\label{eq:281}
    X_\alpha=\bigcup_{k\in \N}\Big\{(x_0,x_1)\in \R^d\times\R^d:
    |x_0|+|x_1|\le k, \ 
    \exists\,\gamma\in \Gamma_{\alpha,k}:\gamma(i)=x_i,\ i=0,1\Big\},
  \end{equation}
  and it is easy to check that each set contributing to the countable
  union in \eqref{eq:281} is compact.
  
  By Lemma \ref{le:measurability} $\Gamma_\alpha$ 
  is a $F_\sigma$ subset of $\Gamma$ and
  the map $\rmL_{\hat\alpha}:\Gamma_\alpha\to [0,+\infty]$ 
  is Borel, so that for every $r\in (0,+\infty]$ the set
  $\Xi_{\hat\alpha}(r):=\big\{\gamma\in \Gamma_\alpha:
  \rmL_{\hat\alpha}[\gamma]<r\big\}$ is Borel in $\Gamma$.

  Since the sublevel $\{(x_0,x_1)\in \R^d\times\R^d:
  c_{\hat\alpha}(x_0,x_1)<r\}$ 
  coincides with the image of $\Xi_{\hat\alpha}(r)$ 
  through the continuous map $(\sfe_0,\sfe_1):\Gamma\to \R^d\times\R^d$, 
  the Projection
  Theorem (see, e.g.~,\cite[Thm.~III.23]{Castaing-Valadier77}, \cite[Thm.~7.4.1]{BogachevII-07})
  then shows that 
  $\{c_{\hat\alpha}<r\}$ is a
  Souslin set and therefore universally measurable. $c_{\hat\alpha}$ 
  is universally measurable as well.

  Let us now consider the sets
  \begin{displaymath}
    W':=\Big\{\big((x_0,x_1),\gamma\big)\in (\R^d\times\R^d)\times\Gamma:
    x_0=\gamma(0),\ x_1=\gamma(1)\Big\},\quad
    W'':=W\cap (X_{\alpha}\times \Gamma_\alpha);
  \end{displaymath}
  since $W'$ is closed, $W''$ is a Borel subset of 
  $(\R^d\times\R^d)\times\Gamma$.
  For every $\eps>0$ we eventually set
  \begin{displaymath}
    W_\eps:=\Big\{\big((x_0,x_1),\gamma\big)\in W'':
    \rmL_{\hat\alpha}[\gamma]-c_{\hat\alpha}(x_0,x_1)<\eps\Big\}.
  \end{displaymath}
  $W_\eps$ is a $\BB_\smmu\otimes \BB(\Gamma)$-measurable 
  subset of $(\R^d\times\R^d)\times\Gamma$,
  where 
  $\BB_\smmu$ is the $\sigma$-algebra of $\mmu$-measurable subsets of 
  $\R^d\times \R^d$ and $\BB(\Gamma)$ is the $\sigma$-algebra of Borel
  subsets of $\Gamma$.
  The projection of $W_\eps$ on the first component coincides with 
  $X_{\alpha}$. 
  Applying Aumann Selection Theorem 
  \cite[Thm.~III.22]{Castaing-Valadier77}, \cite[Thm.~6.9.13]{BogachevII-07} we can find
  a $\BB_\smmu$-measurable map $\oomega:X_{\alpha}
  \to \Gamma$ such that 
  $(x_0,x_1,\oomega(x_0,x_1))\in W_\eps$ for every $x_0,x_1\in X_{\alpha}$.
\end{proof}
We can now state the main result connecting solutions of MFPP and
dynamic plans.
\begin{theorem}
  Let us assume that the structural properties 
  \ref{h.1} hold together with \eqref{eq:268}. 
  Let $\mu_i=m_i\LL^d\in \cP_{2,p}^r(\R^d)$,
  let $\eeta\in \cP_{2,p}(\Gamma)$ a dynamic plan
  satisfying $(\sfe_i)_{\sharp}\eeta=\mu_i$, $i=0,1$,
  and let us call $ m\in L^p(Q)$ a Borel density of
  $\mu=\sfe_\sharp\eeta$ and
  $\alpha:=f(\cdot, m)$.
  
  Then the following conditions are equivalent:
  \begin{enumerate}[(i)]
  \item $\eeta$ is tightened to an optimal solution $(m,\vv)$ of 
    Problem \ref{p:control_eq:cont}.
     \item 
    There exist a weak subsolution $u\in L^1_{loc}(Q)$ to 
          $-\partial_t u+H(x,Du)\le \alpha$ such that 
          \begin{equation}
            \label{eq:255}
            u^+_0(\gamma(0))-u^-_1(\gamma(1))=
            \int_0^1 L_{\hat \alpha}(t,\gamma(t),\dot\gamma(t))\,\d t
            \quad\text{for $\eeta$-a.e.~$\gamma$}.
          \end{equation}
        \item
          We have
    \begin{equation}
      \label{eq:272}
      \begin{gathered}
        \int_\Gamma\Big(\int_0^1 L_{\hat
          \alpha}(t,\gamma(t),\dot\gamma(t))\,\d t\Big)\,\d\eeta\le
        \int_\Gamma\Big(\int_0^1 L_{\hat
          \alpha}(t,\gamma(t),\dot\gamma(t))\,\d t\Big)\,\d\eeta'\\
        \text{if $\eeta'\in \cP(\Gamma)$,
          $(\sfe_i)_\sharp \eeta'=\mu_i$,
          $\gamma\in \ac^2([0,1];\R^d)$ and
          $\int_0^1 \hat M\alpha(t,\gamma(t))\,\d t<\infty$ for $\eeta'$-a.e.~$\gamma$.}
      \end{gathered}
    \end{equation}
  \item We have
    \begin{enumerate}[\rm E.1]
    \item $\eeta$ is concentrated on $c_{\alpha}$-minimizing paths, i.e.
      \begin{equation}
        \label{eq:251}
        \int_0^1 L_{\hat \alpha}(t,\gamma(t),\dot\gamma(t))\,\d t
        =c_{\alpha}(\gamma(0),\gamma(1))
        \quad\text{for $\eeta$-a.e.~$\gamma\in \Gamma$}.
      \end{equation}
      \item
        The plan $\mmu_\seeta:=\sfe_\sharp\eeta= (\sfe_0,\sfe_1)_\sharp\eeta$ is a solution to
        the optimal transport problem induced by $c_{\hat\alpha}$:
        \begin{equation}
          \label{eq:252}
          \int_{\R^d\times \R^d}c_{\hat\alpha}(x_0,x_1)\,\d\mmu_\seeta
          =\min\Big\{
          \int_{\R^d\times \R^d}c_{\hat\alpha}(x_0,x_1)\,\d\mmu:
          \mmu\in \cP(\R^d\times\R^d),\ 
          \pi^i_\sharp\mmu=\mu_i\Big\}.
        \end{equation}
    \end{enumerate}
  \end{enumerate}
  If one of the above equivalent conditions hold, then every 
  optimal solution $(u,\alpha)$ of Problem \ref{p:control_HJ}
  satisfies \eqref{eq:255}, and provides a pair of optimal
  Kantorovich potentials for the problem \eqref{eq:252}, i.e.
  \begin{equation}
    \label{eq:291}
    u_0^+(x_0)-u_1^-(x_1)\le c_{\hat\alpha}(x_0,x_1)\quad\text{in
    }\R^d\times \R^d;\quad
    u_0^+(x_0)-u_1^-(x_1)= c_{\hat\alpha}(x_0,x_1)
    \quad\text{for $\mmu_\seeta$-a.e.~$x_0,x_1$.}    
  \end{equation}
  Moreover
  $\sfu(\cdot,\gamma(\cdot))\in W^{1,1}(0,1)$
  with
  \begin{align}
    \label{eq:263}
    \dot\gamma(t)&=-D_{\spp}H(x,Du(t,\gamma(t))&\text{a.e.~in }(0,1),\\
    \label{eq:260}
    \frac\d{\d
      t}\sfu(t,\gamma(t))&=L(\gamma(t),\dot\gamma(t))+\alpha(t,\gamma(t))
                                              &\text{a.e.~in $(0,1)$,}
  \end{align}
  for $\eeta$-a.e.~$\gamma\in \Gamma$.
\end{theorem}
\begin{proof}
  We divide the proof in various steps.
  
  \smallskip\noindent\textbf{Claim 1: $(i)\Rightarrow (ii)$}
  We can select an optimal pair $(u,\alpha)$ solving Problem
  \ref{p:control_HJ}.
  Since $\eeta\in \cP_{2,p}(\Gamma)$ \eqref{eq:256} shows 
  that 
  \begin{equation}
    \label{eq:255-}
    u^+_0(\gamma(0))-u^-_1(\gamma(1))\le
    \int_0^1 L_{\hat \alpha}(t,\gamma(t),\dot\gamma(t))\,\d t
    \quad\text{for $\eeta$-a.e.~$\gamma$}.
  \end{equation}
  On the other, thanks to \eqref{eq:254} we have
  \begin{align*}
    \int_\Gamma \Big( u^+_0(\gamma(0))-u^-_1(\gamma(1))\Big)\,\d\eeta
    &=
      \int_{\R^d}u_0^+\,\d\mu_0-
      \int_{\R^d}u_1^-\,\d\mu_1=
      \int_Q L_\alpha(t,x,\vv)m\,\d x\,\d t\\&=
      \int_Q L_{\hat\alpha}(t,x,\vv)m\,\d x\,\d t
    =\int_\Gamma\Big(\int_0^1 L_{\hat \alpha}(t,\gamma(t),\dot\gamma(t))\,\d t\Big)\,\d\eeta,
  \end{align*}
  so that \eqref{eq:255} holds.
  
  \smallskip\noindent\textbf{Claim 2: $(ii)\Rightarrow (iii)$}
  It is sufficient to apply Lemma \ref{le:regularization}
  and use the fact that $(\sfe_i)_\sharp\eeta'=\mu_i$:
  \begin{align*}
    \int_\Gamma\Big(\int_0^1 L_{\hat
          \alpha}(t,\gamma(t),\dot\gamma(t))\,\d t\Big)\,\d\eeta
    &= \int_\Gamma \Big(
      u^+_0(\gamma(0))-u^-_1(\gamma(1))\Big)\,\d\eeta
      =
           \int_{\R^d}u_0^+\,\d\mu_0-
      \int_{\R^d}u_1^-\,\d\mu_1
          \\&=
              \int_\Gamma \Big(
      u^+_0(\gamma(0))-u^-_1(\gamma(1))\Big)\,\d\eeta'
              \le \int_\Gamma\Big(\int_0^1 L_{\hat
          \alpha}(t,\gamma(t),\dot\gamma(t))\,\d t\Big)\,\d\eeta'
  \end{align*}
  
  \smallskip\noindent\textbf{Claim 3: $(iii)\Rightarrow (iv)$}
  Since $\hat M\alpha\in L^p(Q)$ we have
  $\int_\Gamma \int_0^1 \hat M\alpha(t,\gamma(t))\,\d
  t\d\eeta(\gamma)<\infty$ so that 
  $\eeta$ is concentrated on 
  $\Gamma_\alpha$.
  Conditions E.1, E.2 are then equivalent to say that 
  \begin{equation}
    \label{eq:273}
    \int_\Gamma \int_0^1 L_{\hat
      \alpha}(t,\gamma(t),\dot\gamma(t))\,\d t\,\d\eeta(\gamma)=
    \min \Big\{
          \int_{\R^d\times \R^d}c_{\hat\alpha}(x_0,x_1)\,\d\mmu:
          \mmu\in \cP(\R^d\times\R^d),\ 
          \pi^i_\sharp\mmu=\mu_i\Big\}.
  \end{equation}
  We argue by contradiction and we assume that there exists $\delta>0$
  and a plan
  $\mmu\in \cP(\R^d\times\R^d)$ with marginals $\mu_0,\mu_1$ such that 
  \begin{equation}
    \label{eq:258}
    \int_{\R^d\times \R^d}c_{\hat\alpha}(x_0,x_1)\,\d\mmu<
     \int_\Gamma \int_0^1 L_{\hat
      \alpha}(t,\gamma(t),\dot\gamma(t))\,\d t\,\d\eeta(\gamma)-\delta.
  \end{equation}
  By using the measurable selection theorem 
  stated in Lemma \ref{le:c-measurability}, we can find 
  a $\mmu$-measurable map $\oomega:X_\alpha
  \to \Gamma_\alpha$ 
  such that 
  for every $x_0,x_1$ the curve $\gamma=\oomega(x_0,x_1)\in
  \ac_2([0,1];\R^d)$ 
  satisfies
  \begin{displaymath}
    \gamma(i)=x_i,\ \int_0^1 \hat M\alpha(t,\gamma(t))\,\d t<\infty,\ 
    \int_0^1 L_{\hat\alpha}(t,\gamma(t),\dot\gamma(t))\le c_{\hat\alpha}(x_0,x_1)+\delta.
  \end{displaymath}
  It is clear that the plan $\eeta':=\oomega_\sharp\mmu$ contradicts
  \eqref{eq:272}.

  \smallskip\noindent\textbf{Claim 4: $(iv)\Rightarrow (i)$}

  Let $(m',\vv')$ be in $\CE 2p{Q;\mu_0,\mu_1}$. 
  By Theorem \ref{thm:superposition_principle} we can find 
  a plan $\eeta'\in \cP_{2,p}(\Gamma)$ tightened to $(m',\vv')$
  and thus satisfying $(\sfe_i)_\sharp\eeta'=\mu_i$,
  $\hat M\alpha(\cdot,\gamma)\in L^1(0,1)$ for $\eeta'$-a.e.~$\gamma$.
  Setting $\mmu':=(\sfe_0,\sfe_1)_\sharp \eeta'$ we have
  \begin{align*}
    \cL_\alpha(m',\vv')
    &=
    \int_Q L_\alpha(t,x,\vv')m'\,\d x\,\d t=
    \int_Q L_{\hat\alpha}(t,x,\vv')m'\,\d x\,\d t
    =
      \int_\Gamma \int_0^1
      L_{\hat\alpha}(t,\gamma(t),\dot\gamma(t))\,\d
      t\,\d\eeta'(\gamma)
    \\&\ge
        \int_{\R^d\times\R^d} c_{\hat\alpha}(\gamma(0),\gamma(1))\,\d\eeta'(\gamma)
        =\int_{\R^d\times\R^d} c_{\hat\alpha}(x_0,x_1)\,\d\mmu'(x_0,x_1)
        \ge \int_{\R^d\times\R^d} c_{\hat\alpha}(x_0,x_1)\,\d\mmu(x_0,x_1)
        \\&=
            \int_\Gamma \int_0^1
      L_{\hat\alpha}(t,\gamma(t),\dot\gamma(t))\,\d
            t\,\d\eeta(\gamma)
            =
             \int_Q L_{\hat\alpha}(t,x,\vv)m\,\d x\,\d t=\cL_\alpha(m,\vv).
  \end{align*}
  Thus $(m,\vv)$ minimizes the modified Lagrangian dynamic cost and we
  can apply Theorem \ref{thm:modified_Lagrangian_cost}.  
\end{proof}

\appendix
\section{Appendix}
\subsection{Minimax and duality}
\label{subsec:minimax}

Let $\A,\B$ be convex sets of some vector spaces and let us suppose that
$\B$ is endowed with some Hausdorff topology. Let $\cL:\A\times \B\to \R$
be
a saddle function satisfying
\begin{align}
  \label{eq:112}
  a\mapsto \cL(a,b)&\quad\text{is concave in $\A$ for every
                     $b\in \B$},\\
  \label{eq:113}
  b\mapsto\cL(a,b)&\quad\text{is convex in $\B$ for every $a\in \A$.}
\end{align}
It is always true that
\begin{equation}
  \label{eq:114}
  \adjustlimits\inf_{b\in \B}\sup_{a\in \A}\cL(a,b)\ge
  \adjustlimits\sup_{a\in \A}\inf_{b\in \B }\cL(a,b).
\end{equation}
The next result provides an important sufficient condition to
guarantee the equality in \eqref{eq:114}:
we use a formulation which is slightly more general than the statement
of \cite[Thm. 3.1]{Simons98}, but it follows by the same argument.
We reproduce here the main part of the proof for the easy of the reader.
\begin{theorem}[Von Neumann]
  \label{thm:VonNeumann}
  Let us suppose that \eqref{eq:112}, \eqref{eq:113} hold and that
  there exists $a_\star\in \A$ and $C_\star>\adjustlimits\sup_{a\in \A}\inf_{b\in \B}\cL(a,b)$ such that
    \begin{gather}
          \label{eq:115}
      \B_\star:=\Big\{b\in \B:\cL(a_\star,b)\le
      C_\star\Big\}\quad\text{is not empty and compact in $\B$,}\\
      \label{eq:140}
      \text{$b\mapsto\cL(a,b)$ is lower semicontinuous in $\B_\star$ for
        every $a\in \A$}.
    \end{gather}
  Then
  \begin{equation}
  \label{eq:114bis}
  \adjustlimits\min_{b\in \B}\sup_{a\in \A}\cL(a,b)=
  \adjustlimits\sup_{a\in \A}\inf_{b\in \B }\cL(a,b).
\end{equation}
\end{theorem}
\begin{proof}
  Let $s:=\adjustlimits\sup_{a\in \A}\inf_{b\in \B }\cL(a,b)$ and
  let $\B_a:=\{b\in \B:\cL(a,b)\le s\}$,
  $\B_{a\star}:=\{b\in \B:\cL(a,b)\le s\}$.
  If $A\subset \A$ is a collection
  containing $a_\star$ then
  \begin{displaymath}
    \B_A=\bigcap_{a\in A}\B_a=
    \bigcap_{a\in A}\B_{a\star}
  \end{displaymath}
  so that $\B_A$ is a (possibly empty) compact set.
  The thesis follows if we check that $\B_\A$ contains a point $\bar
  b$, since in that case
  $\adjustlimits\min_{b\in \B}\sup_{a\in \A}\cL(a,b)\le 
  \sup_{a\in \A}\cL(a,\bar b)\le s$ by construction.

  Since $\B_A$ are compact whenever $a_\star\in A$, it is sufficient
  to prove that for every finite collection $A=\{a_1,\cdots,a_n\}$
  containing $a_\star$ the intersection $B_A$ is not empty.
  To this aim, since $b\mapsto \cL(a_k,b)$ are convex functions,
  \cite[Lemma 2.1]{Simons98} yields
  \begin{align*}
    {\adjustlimits\min_{b\in \B_\star}\sup_{1\le k\le n}\cL(a_k,b)=}
    \min_{b\in \B_\star}\sum_{k=1}^N \nchi_k \cL(a_k,b)
  \end{align*}
  for a suitable choice of nonnegative coefficients $\nchi_k\in
  [0,1]$ with $\sum_{k=1}^n\nchi_k=1$. We thus get by concavity
  \begin{align*}
    {\min_{b\in \B_\star}\sum_{k=1}^n\cL(a_k,b)\le }
    \min_{b\in \B_\star}\cL(\sum_{k=1}^N\nchi_k
    a_k,b)\le s.
  \end{align*}
\end{proof}

\subsection{Convergence in measure}
\begin{lemma}\label{l:conv_meas}
Let $u_n, u \in L^0(\Omega,\frm)$, $n\in \N$. 
The following properties are equivalent:
\begin{itemize}
\item[(a)] $\lim_{n \uparrow + \infty} d(u_n,u) = 0$;
\item[(b)] $u_n$ converges to $u$ in $\varrho$-measure, i.e. 
\begin{equation}
  \label{eq:23}
  \text{for every $\varepsilon > 0:$}\quad
  \lim_{n \uparrow + \infty} \varrho\left( \lbrace x \in \Omega: |u_n(x) - u(x)| \geq \varepsilon \rbrace \right) = 0;
\end{equation}
\item[(c)] $u_n$ converges to $u$ in measure, according to \eqref{eq:24}.
\end{itemize}
\GGG
Moreover, if $\zeta:\R\to \R$ is any continuous and strictly
increasing map, $u_n$ converges to $u$ in measure if and only if 
$\zeta\circ u_n$ converges in measure to $\zeta\circ u$.
\end{lemma}
\begin{proof}
\GGG 
The equivalence between (a) and (b) is well known (see
e.g.~\cite[4.7.60]{BogachevI-07}).  

(b) $\Rightarrow$ (c): 
 since $\rho>0$ $\frm$-a.e., $\frm$ is
 absolutely continuous w.r.t.~$\varrho$ with density $\rho^{-1}$, so that the finite measure 
 $\frm\restr F$ satisfies \cite[Theorem 2.5.7]{BogachevI-07}
 \begin{displaymath}
   \forall\, \eta>0\quad\exists\,\delta>0:\quad
   \frm(A)<\eta\quad\text{if }A\subset F,\ \varrho(A)<\delta.
 \end{displaymath}
 \eqref{eq:23} then yields \eqref{eq:24}.
 
(c) $\Rightarrow$ (b): we consider a measurable 
 partition $(F_k)_{k\in \N}$ of $\Omega$ with $\frm(F_k)<\infty$ for
 every $k\in \N$.
 If \eqref{eq:24} holds, setting $A_n:=\{x\in
 \Omega:|u_n(x)-u(x)|>\eps\}$ we get for every $k\in \N$
 \begin{displaymath}
   \lim_{n\to\infty}\varrho(A_n\cap F_k)=0
 \end{displaymath}
 since $\varrho\le \frm$, so that 
 \begin{displaymath}
   \lim_{n\to\infty}\varrho(A_n)=
   \lim_{n\to\infty}\sum_{k\in \N}\varrho(A_n\cap F_k)=
   \sum_{k\in \N}\lim_{n\to\infty}\varrho(A_n\cap F_k)=0
 \end{displaymath}
 where the interchange between the series and the integral is
 justified by the uniform domination
 \begin{displaymath}
   \varrho(A_n\cap F_k)\le \varrho(F_k),\quad
   \sum_{k\in \N}\varrho(F_k)=1.
 \end{displaymath}
 The last statement of the Lemma 
 concerning the composition map $f\mapsto \zeta\circ f$ 
 is well known (see e.g.~\cite[Cor.~2.2.6]{BogachevI-07}).
 \end{proof}

\subsection{Positive bilinear functionals}
\begin{theorem}
  \label{thm:bilinear-Riesz}
  Let $T:C_c(\R)\times C_c(\R^d)\to \R$ be a bilinear map which is
  positive on pairs of positive functions. 
  Then there exists a positive Radon measure $\vartheta$ on $\R\times
  \R^d$ such that $T(\phi,\psi)=\int_{\R\times \R^d}
  \varphi(t)\psi(x)\,\d\vartheta(t,x)$ for every $\varphi\in
  C^0_c(\R)$, $\psi\in C^0_c(\R^d)$.
\end{theorem}
\begin{proof}
  Let us briefly sketch the proof.  To this aim, we denote by
  $y=(t,x)$ the element in $\R^{d+1}$ and we first define a linear
  functional $T$ on functions $\psi\in C^0_c(\R\times \R^{d+1})$
  admitting the representation
  \begin{equation}
    \label{eq:165}
    \psi(u,y)=\sum_{i=1}^I \zeta_i(u)\varphi_i(y)\quad
    \zeta_i\in C^0_c(\R),\ \varphi_i\in C^0_c(\R^{d+1}).
  \end{equation}
  The natural definition would be
  \begin{equation}
    \label{eq:166}
    T(\psi):=\sum_{i=1}^IT_{\zeta_i}(\varphi_i).
  \end{equation}
  Let us check that if $\psi\ge0$ then the above expression is positive:
  this would imply that \eqref{eq:166} is independent of the
  representation of $\psi$ given by \eqref{eq:165}.

  In order to prove this property, let us consider a nonnegative
  symmetric regularization kernel $h\in C^\infty_c(\R)$ with
  $\supp(h)\subset (-1/2,1/2)$, $\int_\R h(t)\,\d t=1$,
  $h_\tau(t):=\tau^{-1}h(t/\tau)$.  We select an integer $K>0$
  sufficiently big so that $\supp(\zeta_i)\subset (-K+1,K-1)$ for
  every $i=1,2,\ldots,I$ and we fix a nonnegative function
  $\zeta_K\in C^0_c(\R)$ such that $\zeta_K(t)\equiv 1$ on $K$.

  For every $\tau\in (0,1)$ we set
  \begin{equation}
    \label{eq:168}
    \zeta_{i,\tau}(t):=\frac K{N}\sum_{k=-NK}^{NK} \zeta_i
    (k/N)h_\tau(t-k/N).
  \end{equation}
  Notice that $\zeta_{i,\tau}$ is supported in $(-K+1/2,K-1/2)$;
  moreover, since $\lim_{\tau\downarrow0}\zeta_i\ast h_\tau=\zeta_i$
  uniformly, $\zeta_i$ are uniformly continuous and
  $\supp(\zeta_i\ast h_\tau)\subset (-K+1/2,K-1/2)$ for $\tau<1$, for
  every $\eps>0$ we can find $N$ sufficiently big so that
  \begin{equation}
    \label{eq:167}
    \sup_{t\in \R}\Big|\zeta_{i,\tau}(t)-\zeta_i(t)\Big|\le \eps,\quad
    \supp(\zeta_{i,\tau}-\zeta_i)\subset (-K+1/2,K-1/2),\quad
    -\eps \zeta_K\le |\zeta_{i,\tau}-\zeta_i|\le \eps \zeta_K.
  \end{equation}
  It follows by the positivity of $T_\cdot$ that for every nonnegative
  $\varphi\in C^0_c(\R^{d+1})$
  \begin{equation}
    \label{eq:169}
    -\eps T_{\zeta_K}(\varphi)\le
    |T_{\zeta_{i,\tau}}(\varphi)-T_{\zeta_i}(\varphi)|\le 
    \eps T_{\zeta_K}(\varphi).
  \end{equation}
  On the other hand
  \begin{align*}
    \sum_{i=1}^I T_{\zeta_{i,\tau}}(\varphi_i)
    &=
      \sum_{i=1}^I \sum_{k=-KN}^{KN}
      \zeta_i(k/N)T_{h_\tau(\cdot-k/N)}(\varphi_i)
      =\sum_{k=-KN}^{KN}\sum_{i=1}^I 
      \zeta_i(k/N)T_{h_\tau(\cdot-k/N)}(\varphi_i)
    \\&=
        \sum_{k=-KN}^{KN}
        T_{h_\tau(\cdot-k/N)}\Big(\sum_{i=1}^N(\zeta_i(k/N)\varphi_i)\Big)
        =\sum_{k=-KN}^{KN}T_{h_\tau(\cdot-k/N)}\Big(\psi(k/N,\cdot)\Big)\ge0
  \end{align*}
  where in the last inequality we used the fact that $h$ is
  nonnegative and $y\mapsto \psi(k/N,y)\ge0$ by assumption.  Combining
  the last inequality with \eqref{eq:169} we conclude that
  \eqref{eq:166} defines a positive linear functionals on the algebra
  of functions admitting the decomposition \eqref{eq:165}.  Since this
  algebra is uniformly dense in $C^0_c(\R)$ we deduce that there
  exists a unique Radon measure $\vartheta$ on $\R\times \R^{d+1}$
  representing $T$ as in \eqref{eq:164}.

\end{proof}

\bibliography{biblio}
\bibliographystyle{plain}

\end{document}